\documentclass[11pt]{preprint}
\usepackage[full]{textcomp}
\usepackage[osf]{newtxtext} 
\usepackage[cal=boondoxo]{mathalfa}
\usepackage{colortbl}
\usepackage{comment}

\usepackage{amssymb}
\usepackage{mathtools}
\newcommand{\brsq}[1]{\left[#1\right]}
\usepackage{hyperref}
\usepackage{breakurl}
\usepackage{mhenvs}
\usepackage{mhequ} 
\usepackage{mhsymb}
\usepackage{booktabs}
\usepackage{tikz}
\usepackage{tcolorbox}
\usepackage{mathrsfs}
\usepackage[utf8]{inputenc}
\usepackage{longtable}
\usepackage{wrapfig}

\usepackage{microtype}
\usepackage{comment}
\usepackage{wasysym}
\usepackage{centernot}
\usepackage{enumitem}
\usepackage{bm}
\usepackage{stackrel}
\usepackage{graphicx}

\makeatletter
\newcommand{\globalcolor}[1]{%
  \color{#1}\global\let\default@color\current@color
}
\makeatother

\usetikzlibrary{calc}
\usetikzlibrary{decorations}
\usetikzlibrary{positioning}
\usetikzlibrary{shapes}
\usetikzlibrary{external}

\newif\ifdark
\darkfalse

\ifdark

\definecolor{darkred}{rgb}{0.9,0.2,0.2}
\definecolor{darkblue}{rgb}{0.7,0.3,1}
\definecolor{darkgreen}{rgb}{0.1,0.9,0.1}
\definecolor{franck}{rgb}{0,0.8,1}
\definecolor{pagebackground}{rgb}{.15,.21,.18}
\definecolor{pageforeground}{rgb}{.84,.84,.85}
\pagecolor{pagebackground}
\AtBeginDocument{\globalcolor{pageforeground}}
\definecolor{symbols}{rgb}{0,0.7,1}
\colorlet{connection}{red!80!black}
\colorlet{boxcolor}{blue!50}

\else

\definecolor{darkred}{rgb}{0.7,0.1,0.1}
\definecolor{darkblue}{rgb}{0.4,0.1,0.8}
\definecolor{darkgreen}{rgb}{0.1,0.7,0.1}
\definecolor{franck}{rgb}{0,0,1}
\definecolor{pagebackground}{rgb}{1,1,1}
\definecolor{pageforeground}{rgb}{0,0,0}
\colorlet{symbols}{blue!90!black}
\colorlet{connection}{red!30!black}
\colorlet{boxcolor}{blue!50!black}

\fi

\def\slash{\leavevmode\unskip\kern0.18em/\penalty\exhyphenpenalty\kern0.18em}
\def\dash{\leavevmode\unskip\kern0.18em--\penalty\exhyphenpenalty\kern0.18em}

\DeclareMathAlphabet{\mathbbm}{U}{bbm}{m}{n}

\DeclareFontFamily{U}{BOONDOX-calo}{\skewchar\font=45 }
\DeclareFontShape{U}{BOONDOX-calo}{m}{n}{
  <-> s*[1.05] BOONDOX-r-calo}{}
\DeclareFontShape{U}{BOONDOX-calo}{b}{n}{
  <-> s*[1.05] BOONDOX-b-calo}{}
\DeclareMathAlphabet{\mcb}{U}{BOONDOX-calo}{m}{n}
\SetMathAlphabet{\mcb}{bold}{U}{BOONDOX-calo}{b}{n}
\setlist{noitemsep,topsep=4pt,leftmargin=1.5em}

\DeclareMathAlphabet{\mathbbm}{U}{bbm}{m}{n}

\DeclareMathAlphabet{\mcb}{U}{BOONDOX-calo}{m}{n}
\SetMathAlphabet{\mcb}{bold}{U}{BOONDOX-calo}{b}{n}
\DeclareFontFamily{U}{mathx}{\hyphenchar\font45}
\DeclareFontShape{U}{mathx}{m}{n}{
      <5> <6> <7> <8> <9> <10>
      <10.95> <12> <14.4> <17.28> <20.74> <24.88>
      mathx10
      }{}
\DeclareSymbolFont{mathx}{U}{mathx}{m}{n}
\DeclareMathSymbol{\bigtimes}{1}{mathx}{"91}

\def\s{\mathfrak{s}}

\providecommand{\figures}{false}
{ \ifthenelse{\equal{\figures}{false}} {#1}{\[ {\rm Figure \ missing !} \]} }{}

\def\Id{\mathrm{Id}}

\newcommand{\cP}{{\mathcal P}}

\newcommand{\inprod}[2]{\left\langle #1,#2 \right\rangle}

\def\CP{\mathcal{P}}

\def\CC{\mathcal{C}}
\def\cC{\mathcal{C}}
\def\CQ{\mathcal{Q}}
\def\CB{\mathcal{B}}

\def\CT{\mathcal{T}}

\tikzstyle{tinydots}=[dash pattern=on \pgflinewidth off \pgflinewidth]
\tikzstyle{superdense}=[dash pattern=on 4pt off 1pt]

\newcommand{\U}{U}

\newcommand{\mcI}{\mathcal{I}}

\newcommand{\mcK}{\mathcal{K}}

\newcommand{\mcP}{\mathcal{P}}

\newcommand{\mcT}{\mathcal{T}}
\newcommand{\FF}{\hat{\mathbf{F}}}

\newcommand{\mbbN}{\mathbb{N}}
\newcommand{\mbbR}{\mathbb{R}}

\newcommand{\mbbT}{\mathbb{T}}

\newcommand{\mbX}{\mathbf{X}}

\newcommand{\mbn}{\mathbf{n}}

\newcommand\bone{\mathbf{1}}
\newcommand{\beps}{\bar\varepsilon}

\def\eps{\varepsilon}

\newcommand{\mfT}{\mathfrak{T}}
\newcommand{\mfn}{\mathfrak{n}}

\newcommand{\mfe}{\mathfrak{e}}

\newcommand{\mfl}{\mathfrak{l}}

\def\Labe{\mathfrak{e}}

\def\Labn{\mathfrak{n}}

\def\Lab{\mathfrak{L}}

\def\${|\!|\!|}

\def\Co{\mathscr{C}}

\def\scal#1{{\langle#1\rangle}}

\def\SS{\mathfrak{S}}

\def\graft#1{\curvearrowright_{#1}}

\def\Itoo{\mathop{\mathrm{It\hat o}}}

\newenvironment{DIFnomarkup}{}{} % see man latexdiff

\newtheorem{assumption}{Assumption}
 
\theorembodyfont{\rmfamily}
\newtheorem{example}[lemma]{Example}

\newfont{\indic}{bbmss12}

\def\PPi{\boldsymbol{\Pi}}

\def\Nabla_#1{\nabla_{\!#1}}

\makeatletter
\pgfdeclareshape{crosscircle}
{
  \inheritsavedanchors[from=circle] % this is nearly a circle
  \inheritanchorborder[from=circle]
  \inheritanchor[from=circle]{north}
  \inheritanchor[from=circle]{north west}
  \inheritanchor[from=circle]{north east}
  \inheritanchor[from=circle]{center}
  \inheritanchor[from=circle]{west}
  \inheritanchor[from=circle]{east}
  \inheritanchor[from=circle]{mid}
  \inheritanchor[from=circle]{mid west}
  \inheritanchor[from=circle]{mid east}
  \inheritanchor[from=circle]{base}
  \inheritanchor[from=circle]{base west}
  \inheritanchor[from=circle]{base east}
  \inheritanchor[from=circle]{south}
  \inheritanchor[from=circle]{south west}
  \inheritanchor[from=circle]{south east}
  \inheritbackgroundpath[from=circle]
  \foregroundpath{
    \centerpoint%
    \pgf@xc=\pgf@x%
    \pgf@yc=\pgf@y%
    \pgfutil@tempdima=\radius%
    \pgfmathsetlength{\pgf@xb}{\pgfkeysvalueof{/pgf/outer xsep}}%  
    \pgfmathsetlength{\pgf@yb}{\pgfkeysvalueof{/pgf/outer ysep}}%  
    \ifdim\pgf@xb<\pgf@yb%
      \advance\pgfutil@tempdima by-\pgf@yb%
    \else%
      \advance\pgfutil@tempdima by-\pgf@xb%
    \fi%
    \pgfpathmoveto{\pgfpointadd{\pgfqpoint{\pgf@xc}{\pgf@yc}}{\pgfqpoint{-0.707107\pgfutil@tempdima}{-0.707107\pgfutil@tempdima}}}
    \pgfpathlineto{\pgfpointadd{\pgfqpoint{\pgf@xc}{\pgf@yc}}{\pgfqpoint{0.707107\pgfutil@tempdima}{0.707107\pgfutil@tempdima}}}
    \pgfpathmoveto{\pgfpointadd{\pgfqpoint{\pgf@xc}{\pgf@yc}}{\pgfqpoint{-0.707107\pgfutil@tempdima}{0.707107\pgfutil@tempdima}}}
    \pgfpathlineto{\pgfpointadd{\pgfqpoint{\pgf@xc}{\pgf@yc}}{\pgfqpoint{0.707107\pgfutil@tempdima}{-0.707107\pgfutil@tempdima}}}
  }
}
\makeatother

\def\symbol#1{\textcolor{symbols}{#1}}

\def\decorate#1#2{
        \ifnum#2>0
    		\foreach \count in {1,...,#2}{
	       	let
				\p1 = (sourcenode.center),
                \p2 = (sourcenode.east),
				\n1 = {\x2-\x1},
				\n2 = {1mm},
				\n3 = {(1.3+0.6*(\count-1))*\n1},
				\n4 = {0.7*\n1}
			in 
        		node[rectangle,fill=symbols,rotate=30,inner sep=0pt,minimum width=0.2*\n2,minimum height=\n2] at ($(sourcenode.center) + (\n3,\n4)$) {}
				}
		\fi
        \ifnum#1>0
    		\foreach \count in {1,...,#1}{
	       	let
				\p1 = (sourcenode.center),
                \p2 = (sourcenode.east),
				\n1 = {\x2-\x1},
				\n2 = {1mm},
				\n3 = {(1.3+0.6*(\count-1))*\n1},
				\n4 = {0.7*\n1}
			in 
        		node[rectangle,fill=symbols,rotate=-30,inner sep=0pt,minimum width=0.2*\n2,minimum height=\n2] at ($(sourcenode.center) + (-\n3,\n4)$) {}
				}
		\fi
}

\tikzset{
    dectriangle/.style 2 args={
        triangle,
        alias=sourcenode,
        append after command={\decorate{#1}{#2}}
    },
    dectriangle/.default={0}{0},
}

\tikzset{
	cross/.style={path picture={ 
  		\draw[symbols]
			(path picture bounding box.south east) -- (path picture bounding box.north west) (path picture bounding box.south west) -- (path picture bounding box.north east);
		}},
root/.style={circle,fill=green!50!black,inner sep=0pt, minimum size=1.2mm},
        dot/.style={circle,fill=pageforeground,inner sep=0pt, minimum size=1mm},
        blank/.style={circle,fill=white,inner sep=0pt, minimum size=1mm},
        dotred/.style={circle,fill=pageforeground!50!pagebackground,inner sep=0pt, minimum size=2mm},
        var/.style={circle,fill=pageforeground!10!pagebackground,draw=pageforeground,inner sep=0pt, minimum size=3mm},
        sqvar/.style={rectangle,fill=pageforeground!10!pagebackground,draw=pageforeground,inner sep=0pt, minimum size=3mm},
        kernel/.style={semithick,shorten >=2pt,shorten <=2pt},
        kernels/.style={snake=zigzag,shorten >=2pt,shorten <=2pt,segment amplitude=1pt,segment length=4pt,line before snake=2pt,line after snake=5pt,},
        rho/.style={densely dashed,semithick,shorten >=2pt,shorten <=2pt},
           testfcn/.style={dotted,semithick,shorten >=2pt,shorten <=2pt},
        renorm/.style={shape=circle,fill=pagebackground,inner sep=1pt},
        labl/.style={shape=rectangle,fill=pagebackground,inner sep=1pt},
        xic/.style={very thin,circle,draw=symbols,fill=symbols,inner sep=0pt,minimum size=1.2mm},
        g/.style={very thin,rectangle,draw=symbols,fill=symbols!10!pagebackground,inner sep=0pt,minimum width=2.5mm,minimum height=1.2mm},
        xi/.style={very thin,circle,draw=symbols,fill=symbols!10!pagebackground,inner sep=0pt,minimum size=1.2mm},
	xies/.style={very thin,rectangle,fill=green!50!black!25,draw=symbols,inner sep=0pt,minimum size=1.1mm},
	xiesf/.style={very thin,rectangle,fill=green!50!black,draw=symbols,inner sep=0pt,minimum size=1.1mm},
        xix/.style={very thin,crosscircle,fill=symbols!10!pagebackground,draw=symbols,inner sep=0pt,minimum size=1.2mm},
        X/.style={very thin,cross,rectangle,fill=pagebackground,draw=symbols,inner sep=0pt,minimum size=1.2mm},
	xib/.style={thin,circle,fill=symbols!10!pagebackground,draw=symbols,inner sep=0pt,minimum size=1.6mm},
	xie/.style={thin,circle,fill=green!50!black,draw=symbols,inner sep=0pt,minimum size=1.6mm},
	xid/.style={thin,circle,fill=symbols,draw=symbols,inner sep=0pt,minimum size=1.6mm},
	xibx/.style={thin,crosscircle,fill=symbols!10!pagebackground,draw=symbols,inner sep=0pt,minimum size=1.6mm},
	kernels2/.style={very thick,draw=connection,segment length=12pt},
	keps/.style={thin,draw=symbols,->},
	kepspr/.style={thick,draw=connection,->},
	krho/.style={thin,draw=symbols,superdense,->},
	krhopr/.style={thick,draw=connection,superdense,->},
	triangle/.style = { regular polygon, regular polygon sides=3},
	not/.style={thin,circle,draw=connection,fill=connection,inner sep=0pt,minimum size=0.5mm},
	diff/.style = {very thin,draw=symbols,triangle,fill=red!50!black,inner sep=0pt,minimum size=1.6mm},
	diff1/.style = {very thin,dectriangle={1}{0},fill=red!50!black,draw=symbols,inner sep=0pt,minimum size=1.6mm},
	diff2/.style = {very thin,dectriangle={1}{1},fill=red!50!black,draw=symbols,inner sep=0pt,minimum size=1.6mm},
		diffmini/.style = {very thin,rectangle,fill=black,draw=black,inner sep=0pt,minimum size=0.75mm},
	 kernelsmod/.style={very thick,draw=connection,segment length=12pt},
	 rec/.style = {very thin,rectangle,fill=black,draw=black,inner sep=0pt,minimum size=2mm},
	cerc/.style={very thin,circle,draw=black,fill=symbols,inner sep=0pt,minimum size=2mm},
	stars/.style={very thin,star,star points=6,star point ratio=0.5, draw=black,fill=red,inner sep=0pt,minimum size=0.7mm},
	>=stealth,
        }

\makeatletter
\def\DeclareSymbol#1#2#3{%
	\expandafter\gdef\csname MH@symb@#1\endcsname{\tikzsetnextfilename{symbol#1}%
	\tikz[baseline=#2,scale=0.15,draw=symbols,line join=round]{#3}}%
	\expandafter\gdef\csname MH@symb@#1s\endcsname{\scalebox{0.75}{\tikzsetnextfilename{symbol#1}%
	\tikz[baseline=#2,scale=0.15,draw=symbols,line join=round]{#3}}}%
	\expandafter\gdef\csname MH@symb@#1ss\endcsname{\scalebox{0.65}{\tikzsetnextfilename{symbol#1}%
	\tikz[baseline=#2,scale=0.15,draw=symbols,line join=round]{#3}}}%
	}
\def\<#1>{\ifthenelse{\boolean{mmode}}{\mathchoice{\csname MH@symb@#1\endcsname}{\csname MH@symb@#1\endcsname}{\csname MH@symb@#1s\endcsname}{\csname MH@symb@#1ss\endcsname}}{\csname MH@symb@#1\endcsname}}
\makeatother

\DeclareSymbol{Xi22}{0.5}{\draw (0,0) node[xi] {} -- (-1,1) node[not] {} -- (0,2) node[xi] {};} % 1 not used in the text
\DeclareSymbol{Xi2}{-2}{\draw (-1,-0.25) node[xi] {} -- (0,1) node[xi] {};} % 2
\DeclareSymbol{Xi2C}{-2}{\draw (-1,-0.25) node[xi] {} -- (0,1) node[xi] {}; \node at (-1,1) (a) {\scriptsize 1};}
\DeclareSymbol{Xi2alpha}{-2}{\draw (-1,-0.25) node[xi] {} -- (0,1) node[xi] {}; \node at (-1.2,1) (a) {\scriptsize $\alpha$};}
\DeclareSymbol{I(Xi)}{0}{\node[xi] at (0,1) (a) {}; \draw (a) -- (0,0);}
\DeclareSymbol{pXi2}{-2}{\draw (-1,-0.25) node[xi] {} -- (0,1) node[xi] {}; \node at (-0.95,0.95) {\tiny $1$};}
\DeclareSymbol{puXi2}{-2}{\draw (-1,-0.25) node[xi] {} -- (0,1) node[xi] {}; \node at (-1.25,0.95) {\tiny $m$};}
\DeclareSymbol{PlantedNoise}{0}{\draw[black] (0,1.5) node[xi] {} -- (0,-0.25) {};}
\DeclareSymbol{PlantedNoiseC}{0}{\draw[black] (0,1.5) node[xi] {} -- (0,-0.25) {}; \node at (-0.75,0.5) (a) {\scriptsize 1};}
\DeclareSymbol{PlantedNoiseCC}{0}{\draw[black] (0,1.5) node[xi] {} -- (0,-0.25) {}; \node at (-0.75,0.5) (a) {\scriptsize 2};}
\DeclareSymbol{PlantedNoiseCX}{0}{\draw[kernels2] (0,1.5) node[xi] {} -- (0,-0.25) {}; \node at (-0.75,0.5) (a) {\scriptsize 1};}
\DeclareSymbol{PlantedNoiseX}{0}{\draw[kernels2] (0,1.5) node[xi] {} -- (0,-0.25) {};}
\DeclareSymbol{PlantedNoisea}{0}{\draw (0,1.5) node[xi] {} -- (0,-0.25) {}; \node at (-0.75,0.5) (a) {\scriptsize $\alpha$};}
\DeclareSymbol{XiIt}{2}{\node[xi] at (-1.5,-0.1) (a) {};
\node[var] at (0,2.25) (b) {\tiny{$\tau$ }};
\node[blank] at (-1.5,1.25) {\tiny{$k$}};
\draw[symbols]  (a) -- (b);}
\DeclareSymbol{Xi1It}{2}{\node[xi] at (-1.5,-0.1) (a) {};
\node[var] at (0,2.25) (b) {\tiny{$\tau$ }};
\node[blank] at (-1.5,1.25) {\tiny{$1$}};
\draw[symbols]  (a) -- (b);}
\DeclareSymbol{Xi0It}{2}{\node[xi] at (-1.5,-0.1) (a) {};
\node[var] at (0,2.25) (b) {\tiny{$\tau$ }};
\node[blank] at (-1.5,1.25) {};
\draw[symbols]  (a) -- (b);}
\DeclareSymbol{It}{2}{\node at (0,-1) (a) {};
\node[var] at (0,2.25) (b) {\tiny{$\tau$ }};
\node[blank] at (-0.75,0.5) {\tiny{$k$}};
\draw[symbols]  (a) -- (b);}
\DeclareSymbol{It0}{2}{\node at (0,-1) (a) {};
\node[var] at (0,2.25) (b) {\tiny{$\tau$ }};
\draw[symbols]  (a) -- (b);}
\DeclareSymbol{1It}{2}{\node at (0,-1) (a) {};
\node[var] at (0,2.25) (b) {\tiny{$\tau$ }};
\node[blank] at (-0.75,0.5) {\tiny{$1$}};
\draw[symbols]  (a) -- (b);}
\DeclareSymbol{1I1t1}{2}{\node at (0,-1) (a) {};
\node[var] at (0,2.25) (b) {\tiny{$\tau_1$ }};
\node[blank] at (-0.75,0.5) {\tiny{$1$}};
\draw[kernels2]  (a) -- (b);}
\DeclareSymbol{1I1t2}{2}{\node at (0,-1) (a) {};
\node[var] at (0,2.25) (b) {\tiny{$\tau_2$ }};
\node[blank] at (-0.75,0.5) {\tiny{$1$}};
\draw[kernels2]  (a) -- (b);}
\DeclareSymbol{1I1t}{2}{\node at (0,-1) (a) {};
\node[var] at (0,2.25) (b) {\tiny{$\tau$ }};
\node[blank] at (-0.75,0.5) {\tiny{$1$}};
\draw[kernels2]  (a) -- (b);}
\DeclareSymbol{I1t}{2}{\node at (0,-1) (a) {};
\node[var] at (0,2.25) (b) {\tiny{$\tau$ }};
\node[blank] at (-0.75,0.5) {\tiny{$1$}};
\draw[kernels]  (a) -- (b);}
\DeclareSymbol{I1t1}{2}{\node at (0,-1) (a) {};
\node[var] at (0,2.25) (b) {\tiny{$\tau_1$}};
\node[blank] at (-0.75,0.5) {};
\draw[kernels2]  (a) -- (b);}
\DeclareSymbol{Ita}{2}{\node at (0,-1) (a) {};
\node[var] at (0,2.25) (b) {\tiny{$\tau$ }};
\node[blank] at (-0.75,0.5) {\tiny{$\alpha$}};
\draw[symbols]  (a) -- (b);}
\DeclareSymbol{I1t1I1t2}{0}{\node[var] at (-2,2) (a) {\tiny{$\tau_1$}}; 
\node[var] at (2,2) (b) {\tiny{$\tau_2$}}; 
\draw[kernels2] (0,0) -- (a); 
\draw[kernels2] (0,0) -- (b); 
\node at (-1.5,-0.05) {{\tiny $\ell$}}; 
\node at (1.5,-0.25) {{\tiny $m$}};}
\DeclareSymbol{0I1t10I1t2}{0}{\node[var] at (-2,2) (a) {\tiny{$\tau_1$}}; 
\node[var] at (2,2) (b) {\tiny{$\tau_2$}}; 
\draw[kernels2] (0,0) -- (a); 
\draw[kernels2] (0,0) -- (b);}
\DeclareSymbol{0I1t11I1t2}{0}{\node[var] at (-2,2) (a) {\tiny{$\tau_1$}}; 
\node[var] at (2,2) (b) {\tiny{$\tau_2$}}; 
\draw[kernels2] (0,0) -- (a); 
\draw[kernels2] (0,0) -- (b); 
\node at (-1.5,-0.05) {{\tiny $0$}}; 
\node at (1.5,-0.05) {{\tiny $1$}};}
\DeclareSymbol{I1XiItI1Xi}{0}{\node[xi] at (-2,2) (a) {}; 
\node[xi] at (2,2) (b) {}; 
\node[var] at (0,3) (c) {\tiny{$\tau$}};
\draw[kernels2] (0,0) -- (a); 
\draw[kernels2] (0,0) -- (b);
\draw[] (0,0) -- (c);
\node at (-1.5,-0.05) {{\tiny $\ell$}}; 
\node at (1.5,-0.25) {{\tiny $m$}};}
\DeclareSymbol{I1t1I1t2It3}{0}{\node[var] at (-3,2) (a) {\tiny{$\tau_1$}}; 
\node[var] at (3,2) (b) {\tiny{$\tau_3$}}; 
\node[var] at (0,4) (c) {\tiny{$\tau_2$}};
\draw[kernels2] (0,0) -- (a); 
\draw (0,0) -- (b);
\draw[kernels2] (0,0) -- (c);
\node at (-1.5,-0.05) {{\tiny $\ell$}};
\node at (-1,2) {\tiny{$m$}}; 
\node at (1.5,-0.25) {{\tiny $n$}};}
\DeclareSymbol{XiIt1It2It3}{0}{\node[var] at (-3,2) (a) {\tiny{$\tau_1$}}; 
\node[var] at (3,2) (b) {\tiny{$\tau_3$}}; 
\node[var] at (0,4) (c) {\tiny{$\tau_2$}};
\node[xi] at (0,0) (d) {};
\draw (d) -- (a); 
\draw (d) -- (b);
\draw (d) -- (c);
\node at (-1.5,-0.05) {{\tiny $\ell$}};
\node at (-1,2) {\tiny{$m$}}; 
\node at (1.5,-0.25) {{\tiny $n$}};}
\DeclareSymbol{It1It2It3}{0}{\node[var] at (-3,2) (a) {\tiny{$\tau_1$}}; 
\node[var] at (3,2) (b) {\tiny{$\tau_3$}}; 
\node[var] at (0,4) (c) {\tiny{$\tau_2$}};
\draw (0,0) -- (a); 
\draw (0,0) -- (b);
\draw (0,0) -- (c);
\node at (-1.5,-0.05) {{\tiny $\ell$}};
\node at (-1,2) {\tiny{$m$}}; 
\node at (1.5,-0.25) {{\tiny $n$}};}
\DeclareSymbol{It1It2I1t3}{0}{\node[var] at (-3,2) (a) {\tiny{$\tau_1$}}; 
\node[var] at (3,2) (b) {\tiny{$\tau_3$}}; 
\node[var] at (0,4) (c) {\tiny{$\tau_2$}};
\draw (0,0) -- (a); 
\draw[kernels2] (0,0) -- (b);
\draw (0,0) -- (c);
\node at (-1.5,-0.05) {{\tiny $\ell$}};
\node at (-1,2) {\tiny{$m$}}; 
\node at (1.5,-0.25) {{\tiny $n$}};}
\DeclareSymbol{It1I1t2It3}{0}{\node[var] at (-3,2) (a) {\tiny{$\tau_1$}}; 
\node[var] at (3,2) (b) {\tiny{$\tau_3$}}; 
\node[var] at (0,4) (c) {\tiny{$\tau_2$}};
\draw (0,0) -- (a); 
\draw (0,0) -- (b);
\draw[kernels2] (0,0) -- (c);
\node at (-1.5,-0.05) {{\tiny $\ell$}};
\node at (-1,2) {\tiny{$m$}}; 
\node at (1.5,-0.25) {{\tiny $n$}};}
\DeclareSymbol{It1It2It3It4}{0}{\node[var] at (-4,2) (a) {\tiny{$\tau_1$}}; 
\node[var] at (4,2) (b) {\tiny{$\tau_3$}}; 
\node[var] at (-1.5,4) (c) {\tiny{$\tau_2$}};
\node[var] at (1.5,4) (d) {\tiny{$\tau_2$}};
\draw (0,0) -- (a); 
\draw (0,0) -- (b);
\draw (0,0) -- (c);
\draw (0,0) -- (d);
\node at (-1.5,-0.05) {{\tiny $k$}};
\node at (-1.5,1.75) {\tiny{$\ell$}}; 
\node at (1.5,-0.25) {{\tiny $m$}};
\node at (1.5, 1.75) {{\tiny $n$}};}
\DeclareSymbol{It1I1t2It3It4}{0}{\node[var] at (-4,2) (a) {\tiny{$\tau_1$}}; 
\node[var] at (4,2) (b) {\tiny{$\tau_3$}}; 
\node[var] at (-1.5,4) (c) {\tiny{$\tau_2$}};
\node[var] at (1.5,4) (d) {\tiny{$\tau_2$}};
\draw (0,0) -- (a); 
\draw (0,0) -- (b);
\draw[kernels2] (0,0) -- (c);
\draw (0,0) -- (d);
\node at (-1.5,-0.05) {{\tiny $k$}};
\node at (-1.5,1.75) {\tiny{$\ell$}}; 
\node at (1.5,-0.25) {{\tiny $m$}};
\node at (1.5, 1.75) {{\tiny $n$}};}
\DeclareSymbol{It1I1t2I1t3}{0}{\node[var] at (-3,2) (a) {\tiny{$\tau_1$}}; 
\node[var] at (3,2) (b) {\tiny{$\tau_3$}}; 
\node[var] at (0,4) (c) {\tiny{$\tau_2$}};
\draw (0,0) -- (a); 
\draw[kernels2] (0,0) -- (b);
\draw[kernels2] (0,0) -- (c);
\node at (-1.5,-0.05) {{\tiny $\ell$}};
\node at (-1,2) {\tiny{$m$}}; 
\node at (1.5,-0.25) {{\tiny $n$}};}
\DeclareSymbol{It1It2I1t3I1t4}{0}{\node[var] at (-4,2) (a) {\tiny{$\tau_1$}}; 
\node[var] at (4,2) (b) {\tiny{$\tau_4$}}; 
\node[var] at (2,4) (c) {\tiny{$\tau_3$}};
\node[var] at (-2,4) (d) {\tiny{$\tau_2$}};
\draw (0,0) -- (a); 
\draw[kernels2] (0,0) -- (b);
\draw[kernels2] (0,0) -- (c);
\draw (0,0) -- (d);
\node at (-1.75,-0.1) {{\tiny $k$}};
\node at (-2.15,2) {\tiny{$\ell$}};
\node at (2.15,2) {\tiny{$m$}}; 
\node at (1.75,-0.3) {{\tiny $n$}};}
\DeclareSymbol{XiItIXi}{0}{\node[var] at (-2,2) (a) {\tiny{$\tau$}}; 
\node[xi] at (2,2) (b) {}; 
\draw (0,0) -- (a); 
\draw (0,0) -- (b); 
\node at (-1.5,-0.05) {{\tiny $\ell$}}; 
\node at (1.5,-0.25) {{\tiny $m$}};
\node[xi] at (0,0) (c) {};}
\DeclareSymbol{It1It2}{0}{\node[var] at (-2,2) (a) {\tiny{$\tau_1$}}; 
\node[var] at (2,2) (b) {\tiny{$\tau_2$}}; 
\draw (0,0) -- (a); 
\draw (0,0) -- (b); 
\node at (-1.5,-0.05) {{\tiny $\ell$}}; 
\node at (1.5,-0.25) {{\tiny $m$}};}
\DeclareSymbol{It1I1t2}{0}{\node[var] at (-2,2) (a) {\tiny{$\tau_1$}}; 
\node[var] at (2,2) (b) {\tiny{$\tau_2$}}; 
\draw (0,0) -- (a); 
\draw[kernels2] (0,0) -- (b); 
\node at (-1.5,-0.05) {{\tiny $\ell$}}; 
\node at (1.5,-0.25) {{\tiny $m$}};}
\DeclareSymbol{Xi2b}{-2}{\draw (-1,-0.25) node[xic] {} -- (0,1) node[xic] {};} % 2
\DeclareSymbol{Xi2g}{-2}{\draw (-1,-0.25) node[xies] {} -- (0,1) node[xi] {};} % 2
\DeclareSymbol{Xi2g2}{-2}{\draw (-1,-0.25) node[xi] {} -- (0,1) node[xies] {};} % 2
\DeclareSymbol{cXi2}{-2}{\draw (0,-0.25) node[xi] {} -- (-1,1) node[xic] {};}%3
\DeclareSymbol{Xi3}{0}{\draw (0,0) node[xi] {} -- (-1,1) node[xi] {} -- (0,2) node[xi] {};}%4
\DeclareSymbol{XiIIXi}{0}{\draw (0,0) node[xi] {} -- (-1,1); \draw[kernels2] (-1,1) node[not] {} -- (0,2) node[xi] {};}%4

\DeclareSymbol{Xi4}{2}{\draw (0,0) node[xi] {} -- (-1,1) node[xi] {} -- (0,2) node[xi] {} -- (-1,3) node[xi] {};}%5
\DeclareSymbol{Xi4_1}{2}{\draw (0,0) node[xic] {} -- (-1,1) node[xic] {} -- (0,2) node[xi] {} -- (-1,3) node[xi] {};}%6
\DeclareSymbol{Xi4_2}{2}{\draw (0,0) node[xic] {} -- (-1,1) node[xi] {} -- (0,2) node[xi] {} -- (-1,3) node[xic] {};}%7
\DeclareSymbol{Xi2X}{-2}{\draw (0,-0.25) node[xi] {} -- (-1,1) node[xix] {};}%8
\DeclareSymbol{XXi2}{-2}{\draw (0,-0.25) node[xix] {} -- (-1,1) node[xi] {};}%9
\DeclareSymbol{IIXi}{0}{\draw (0,-0.25) node[not] {} -- (-1,1) node[xi] {} -- (0,2) node[xi] {};}%10
\DeclareSymbol{IXi^2}{-1}{\draw (-1,1) node[xi] {} -- (0,0) node[not] {} -- (1,1) node[xi] {};}%11
\DeclareSymbol{IIXi^2}{-4}{\draw (0,-1.5) node[not] {} -- (0,0);
\draw[kernels2] (-1,1) node[xi] {} -- (0,0) node[not] {} -- (1,1) node[xi] {};}%12
\DeclareSymbol{XiX}{-2.8}{\node[xibx] {};}%13
\DeclareSymbol{tauX}{-2.8}{ \node[X] {};}%14
\DeclareSymbol{Xi}{-2.8}{\node[xib] {};}%15

\DeclareSymbol{IXiX}{-1}{\draw (0,-0.25) node[not] {} -- (-1,1) node[xix] {};}%18
\DeclareSymbol{IXi3}{2}{\draw (0,-0.25) node[not] {} -- (-1,1) node[xi] {} -- (0,2) node[xi] {} -- (-1,3) node[xi] {};}%20
\DeclareSymbol{IXi}{-2}{\draw (0,-0.25) node[not] {} -- (-1,1) node[xi] {};}%21
\DeclareSymbol{XiI}{-2}{\draw (0,-0.25) node[xi] {} -- (-1,1) node[not] {};}%22

\DeclareSymbol{Xi4b}{0}{\draw(0,1.5) node[xi] {} -- (0,0); \draw (-1,1) node[xi] {} -- (0,0) node[xi] {} -- (1,1) node[xi] {};}%23
\DeclareSymbol{Xi4b'}{0}{\draw(0,1.5) node[xi] {} -- (0,-0.2); \draw (-1,1) node[xi] {} -- (0,-0.2) node[not] {} -- (1,1) node[xi] {};}%24 not used
\DeclareSymbol{Xi4c}{0}{\draw (0,1) -- (0.8,2.2) node[xi] {};\draw (0,-0.25) node[xi] {} -- (0,1) node[xi] {} -- (-0.8,2.2) node[xi] {};}%25
\DeclareSymbol{Xi4d}{-4.5}{\draw (0,-1.5) node[not] {} -- (0,0); \draw (-1,1) node[xi] {} -- (0,0) node[xi] {} -- (1,1) node[xi] {};}%26 not used
\DeclareSymbol{Xi4e}{0}{\draw (0,2) node[xi] {} -- (-1,1) node[xi] {} -- (0,0) node[xi] {} -- (1,1) node[xi] {};}%27
\DeclareSymbol{Xi4e'}{0}{\draw (0,2) node[xi] {} -- (-1,1) node[xi] {} -- (0,-0.2) node[not] {} -- (1,1) node[xi] {};}%28 not used

\DeclareSymbol{Xitwo}%29
{0}{\draw[kernels2] (0,0) node[not] {} -- (-1,1) node[not] {}
-- (-2,2) node[not]{} -- (-3,3) node[xi]  {};
\draw[kernels2] (0,0) -- (1,1) node[xi] {};
\draw[kernels2] (-1,1) -- (0,2) node[xi] {};
\draw[kernels2] (-2,2) -- (-1,3) node[xi] {};}

\DeclareSymbol{IXitwo}%30
{0}{\draw (-.7,1.2) node[xi] {} -- (0,-0.2) -- (.7,1.2) node[xi] {};}
\DeclareSymbol{I1Xitwo}%30
{0}{\draw[kernels2] (0,0) node[not] {} -- (-1,1) node[xi] {};
\draw[kernels2] (0,0) -- (1,1) node[xi] {};}
\DeclareSymbol{I1Xitwou}%30
{0}{\draw[kernels2] (0,0) node[not] {} -- (-1,1) node[xi] {};
\draw[kernels2] (0,0) -- (1,1) node[xi] {}; \node at (-0.85,-0.2) {{\tiny $1$}}; \node at (0.9,-0.2) {{\tiny $0$}};}

\DeclareSymbol{I1Xitwoub}%30
{0}{\draw[kernels2] (0,0) node[not] {} -- (-1,1) node[xi] {};
	\draw[kernels2] (0,0) -- (1,1) node[xi] {}; \node at (-0.85,-0.2) {{\tiny $0$}}; \node at (0.9,-0.2) {{\tiny $1$}};}

\DeclareSymbol{I1Xitwoab}%30
{0}{\draw[kernels2] (0,0) node[not] {} -- (-1,1) node[xi] {};
\draw[kernels2] (0,0) -- (1,1) node[xi] {}; \node at (-0.85,-0.1) {{\tiny $\alpha$}}; \node at (0.9,-0.2) {{\tiny $\beta$}};}
\DeclareSymbol{I1Xitwoup}%30
{0}{\draw[kernels2] (0,0) node[not] {} -- (-1,1) node[xi] {};
\draw[kernels2] (0,0) -- (1,1) node[xi] {}; \node at (-0.85,0) {{\tiny $k$}}; \node at (0.9,-0.1) {{\tiny $l$}};}
\DeclareSymbol{I1Xitwobis}%30
{0}{\draw[kernels2] (0,0) node[not] {} -- (-1,1) node[xies] {};
\draw[kernels2] (0,0) -- (1,1) node[xies] {};}

\DeclareSymbol{I1Xitwog}%30
{0}{\draw[kernels2] (0,0) node[not] {} -- (-1,1) node[xies] {};
\draw[kernels2] (0,0) -- (1,1) node[xi] {};}

\DeclareSymbol{cI1Xitwo}%31
{0}{\draw[kernels2] (0,0) node[not] {} -- (-1,1) node[xic] {};
\draw[kernels2] (0,0) -- (1,1) node[xi] {};}

\DeclareSymbol{I1IXi3}{0}{\draw (0,0) node[xi] {} -- (-1,1) ; %32
\draw[kernels2] (-1,1) node[not] {} -- (0,2) node[xi] {};
\draw[kernels2] (-1,1) node[not] {} -- (-2,2) node[xi] {};}

\DeclareSymbol{I1Xi3c}{-1}{\draw[kernels2](0,1.5) node[xi] {} -- (0,0) node[not] {}; \draw (-1,1) node[xi] {} -- (0,0) ; \draw[kernels2] (0,0) -- (1,1) node[xi] {};}%33

\DeclareSymbol{I1Xi3cbis}{-1}{\draw[kernels2](0,1.5) node[xies] {} -- (0,0) node[not] {}; \draw (-1,1) node[xies] {} -- (0,0) ; \draw[kernels2] (0,0) -- (1,1) node[xies] {};}%33

\DeclareSymbol{I1IXi3b}{0}{\draw[kernels2] (0,0) node[not] {} -- (-1,1) ; \draw[kernels2] (0,0)   -- (1,1) node[xi] {} ;
\draw (-1,1) node[xi] {} -- (0,2) node[xi] {};
}%34

\DeclareSymbol{I1IXi3c}{0}{\draw[kernels2] (0,0) node[not] {} -- (-1,1) ; \draw[kernels2] (0,0)   -- (1,1) node[xi] {} ;
\draw[kernels2] (-1,1) node[not] {} -- (0,2) node[xi] {};
\draw[kernels2] (-1,1) node[not] {} -- (-2,2) node[xi] {};}%35

\DeclareSymbol{I1IXi3cbis}{0}{\draw[kernels2] (0,0) node[not] {} -- (-1,1) ; \draw[kernels2] (0,0)   -- (1,1) node[xies] {} ;
\draw[kernels2] (-1,1) node[not] {} -- (0,2) node[xies] {};
\draw[kernels2] (-1,1) node[not] {} -- (-2,2) node[xies] {};}%35

\DeclareSymbol{I1Xi}{0}{\draw[kernels2] (0,0) node[not] {} -- (-1,1)  node[xi] {} ;}%36

\DeclareSymbol{I1Xi4a}{2}{\draw[kernels2] (0,0) node[not] {} -- (-1,1) ; \draw[kernels2] (0,0) node[not] {} -- (1,1) node[xi] {} ;%37
\draw (-1,1) node[xi] {} -- (0,2) node[xi] {} -- (-1,3) node[xi] {};}
\DeclareSymbol{cI1Xi4a}{2}{\draw[kernels2] (0,0) node[not] {} -- (-1,1) ; \draw[kernels2] (0,0) node[not] {} -- (1,1) node[xic] {} ;%39
\draw (-1,1) node[xic] {} -- (0,2) node[xi] {} -- (-1,3) node[xi] {};}
\DeclareSymbol{I1Xi4b}{2}{\draw (0,0) node[xi] {} -- (-1,1) node[xi] {} -- (0,2) ; \draw[kernels2] (0,2) node[not] {} -- (-1,3) node[xi] {};\draw[kernels2] (0,2)  -- (1,3) node[xi] {};
}%41

\DeclareSymbol{cI1Xi4b}{2}{\draw (0,0) node[xic] {} -- (-1,1) node[xic] {} -- (0,2) ; \draw[kernels2] (0,2) node[not] {} -- (-1,3) node[xi] {};\draw[kernels2] (0,2)  -- (1,3) node[xi] {};
}%42

\DeclareSymbol{I1Xi4c}{2}{\draw (0,0) node[xi] {} -- (-1,1) node[not] {}; \draw[kernels2] (-1,1) -- (0,2) ; 
\draw[kernels2] (-1,1) -- (-2,2) node[xi] {} ;
\draw (0,2) node[xi] {} -- (-1,3) node[xi] {};}%43

\DeclareSymbol{cI1Xi4c}{2}{\draw (0,0) node[xic] {} -- (-1,1) node[not] {}; \draw[kernels2] (-1,1) -- (0,2) ; 
\draw[kernels2] (-1,1) -- (-2,2) node[xic] {} ;
\draw (0,2) node[xi] {} -- (-1,3) node[xi] {};}%44

\DeclareSymbol{I1Xi4ab}{2}{\draw[kernels2] (0,0) node[not] {} -- (-1,1) ; \draw[kernels2] (0,0) node[not] {} -- (1,1) node[xi] {};\draw (-1,1) node[xi] {} -- (0,2) ; \draw[kernels2] (0,2) node[not] {} -- (-1,3) node[xi] {};\draw[kernels2] (0,2)  -- (1,3) node[xi] {}; }%45

\DeclareSymbol{cI1Xi4ab}{2}{\draw[kernels2] (0,0) node[not] {} -- (-1,1) ; \draw[kernels2] (0,0) node[not] {} -- (1,1) node[xic] {};\draw (-1,1) node[xic] {} -- (0,2) ; \draw[kernels2] (0,2) node[not] {} -- (-1,3) node[xi] {};\draw[kernels2] (0,2)  -- (1,3) node[xi] {}; }%46

\DeclareSymbol{I1Xi4bc}{2}{\draw (0,0) node[xi] {} -- (-1,1) node[not] {}; \draw[kernels2] (-1,1) -- (0,2) ; %47
\draw[kernels2] (-1,1) -- (-2,2) node[xi] {} ; \draw[kernels2] (0,2) node[not] {} -- (-1,3) node[xi] {};\draw[kernels2] (0,2)  -- (1,3) node[xi] {};
}

\DeclareSymbol{cI1Xi4bc}{2}{\draw (0,0) node[xic] {} -- (-1,1) node[not] {}; \draw[kernels2] (-1,1) -- (0,2) ; %48
\draw[kernels2] (-1,1) -- (-2,2) node[xic] {} ; \draw[kernels2] (0,2) node[not] {} -- (-1,3) node[xi] {};\draw[kernels2] (0,2)  -- (1,3) node[xi] {};
}

\DeclareSymbol{I1Xi4abcc1}{2}{\draw[kernels2] (0,0) node[not] {} -- (-1,1) node[not] {}%50
-- (-2,2) node[not]{} -- (-3,3) node[xic]  {};
\draw[kernels2] (0,0) -- (1,1) node[xic] {};
\draw[kernels2] (-1,1) -- (0,2) node[xi] {};
\draw[kernels2] (-2,2) -- (-1,3) node[xi] {};
}

\DeclareSymbol{I1Xi4abcc1b}{2}{\draw[kernels2] (0,0) node[not] {} -- (-1,1) node[not] {}%50
-- (-2,2) node[not]{} -- (-3,3) node[xi]  {};
\draw[kernels2] (0,0) -- (1,1) node[xic] {};
\draw[kernels2] (-1,1) -- (0,2) node[xic] {};
\draw[kernels2] (-2,2) -- (-1,3) node[xi] {};
}

\DeclareSymbol{I1Xi4abcc2}{2}{\draw[kernels2] (0,0) node[not] {} -- (-1,1) node[not] {}%51
-- (-2,2) node[not]{} -- (-3,3) node[xic]  {};
\draw[kernels2] (0,0) -- (1,1) node[xi] {};
\draw[kernels2] (-1,1) -- (0,2) node[xi] {};
\draw[kernels2] (-2,2) -- (-1,3) node[xic] {};
}

\DeclareSymbol{I1Xi4ac}{2}{\draw[kernels2] (0,0) node[not] {} -- (-1,1) ; \draw[kernels2] (0,0) node[not] {} -- (1,1) node[xi] {}; 
\draw[kernels2] (-1,1) node[not] {} -- (0,2) ; %52
\draw[kernels2] (-1,1) -- (-2,2) node[xi] {} ;
\draw (0,2) node[xi] {} -- (-1,3) node[xi] {};}

\DeclareSymbol{cI1Xi4ac}{2}{\draw[kernels2] (0,0) node[not] {} -- (-1,1) ; \draw[kernels2] (0,0) node[not] {} -- (1,1) node[xic] {}; 
\draw[kernels2] (-1,1) node[not] {} -- (0,2) ; %53
\draw[kernels2] (-1,1) -- (-2,2) node[xic] {} ;
\draw (0,2) node[xi] {} -- (-1,3) node[xi] {};}

\DeclareSymbol{I1Xi4acc1}{2}{\draw[kernels2] (0,0) node[not] {} -- (-1,1) ; \draw[kernels2] (0,0) node[not] {} -- (1,1) node[xic] {}; 
\draw[kernels2] (-1,1) node[not] {} -- (0,2) ; %54
\draw[kernels2] (-1,1) -- (-2,2) node[xi] {} ;
\draw (0,2) node[xic] {} -- (-1,3) node[xi] {};}

\DeclareSymbol{I1Xi4acc2}{2}{\draw[kernels2] (0,0) node[not] {} -- (-1,1) ; \draw[kernels2] (0,0) node[not] {} -- (1,1) node[xic] {}; 
\draw[kernels2] (-1,1) node[not] {} -- (0,2) ; %55
\draw[kernels2] (-1,1) -- (-2,2) node[xi] {} ;
\draw (0,2) node[xi] {} -- (-1,3) node[xic] {};}

\DeclareSymbol{2I1Xi4}{2}{\draw[kernels2] (0,0) node[not] {} -- (-1,1) node[not] {};
\draw[kernels2] (0,0) -- (1,1) node[not] {};%56
\draw[kernels2] (-1,1) -- (-1.5,2.5) node[xi] {};
\draw[kernels2] (-1,1) -- (-0.5,2.5) node[xi] {};
\draw[kernels2] (1,1) -- (0.5,2.5) node[xi] {};
\draw[kernels2] (1,1) -- (1.5,2.5) node[xi] {};
}

\DeclareSymbol{2I1Xi4dis}{2}{\draw[kernels2] (0,0) node[not] {} -- (-1,1) node[not] {};
\draw[kernels2] (0,0) -- (1,1) node[not] {};%56
\draw[kernels2] (-1,1) -- (-1.5,2.5) node[xies] {};
\draw[kernels2] (-1,1) -- (-0.5,2.5) node[xies] {};
\draw[kernels2] (1,1) -- (0.5,2.5) node[xies] {};
\draw[kernels2] (1,1) -- (1.5,2.5) node[xies] {};
}

\DeclareSymbol{2I1Xi4c1}{2}{\draw[kernels2] (0,0) node[not] {} -- (-1,1) node[not] {};%57
\draw[kernels2] (0,0) -- (1,1) node[not] {};
\draw[kernels2] (-1,1) -- (-1.5,2.5) node[xic] {};
\draw[kernels2] (-1,1) -- (-0.5,2.5) node[xi] {};
\draw[kernels2] (1,1) -- (0.5,2.5) node[xic] {};
\draw[kernels2] (1,1) -- (1.5,2.5) node[xi] {};
}

\DeclareSymbol{2I1Xi4c2}{2}{\draw[kernels2] (0,0) node[not] {} -- (-1,1) node[not] {};%58
\draw[kernels2] (0,0) -- (1,1) node[not] {};
\draw[kernels2] (-1,1) -- (-1.5,2.5) node[xic] {};
\draw[kernels2] (-1,1) -- (-0.5,2.5) node[xic] {};
\draw[kernels2] (1,1) -- (0.5,2.5) node[xi] {};
\draw[kernels2] (1,1) -- (1.5,2.5) node[xi] {};
}

\DeclareSymbol{2I1Xi4b}{2}{\draw[kernels2] (0,0) node[not] {} -- (-1,1) ;%59
\draw[kernels2] (0,0) -- (1,1);
\draw (-1,1) node[xi] {} -- (-1,2.5) node[xi] {};
\draw (1,1)  node[xi] {} -- (1,2.5) node[xi] {};
}

\DeclareSymbol{2I1Xi4bb}{2}{\draw[kernels2] (0,0) node[not] {} -- (-1,1) ;%59
\draw[kernels2] (0,0) -- (1,1);
\draw (-1,1) node[xi] {} -- (-1,2.5) node[xiesf] {};
\draw (1,1)  node[xi] {} -- (1,2.5) node[xic] {};
}

\DeclareSymbol{2I1Xi4c}{2}{\draw[kernels2] (0,0) node[not] {} -- (-1,1);%60
\draw[kernels2] (0,0) -- (1,1) node[not] {};
\draw (-1,1)  node[xi] {} -- (-1,2.5) node[xi] {};
\draw[kernels2] (1,1) -- (0.4,2.5) node[xi] {};
\draw[kernels2] (1,1) -- (1.6,2.5) node[xi] {};
}

\DeclareSymbol{2I1Xi4cc1}{2}{\draw[kernels2] (0,0) node[not] {} -- (-1,1);%61
\draw[kernels2] (0,0) -- (1,1) node[not] {};
\draw (-1,1)  node[xic] {} -- (-1,2.5) node[xi] {};
\draw[kernels2] (1,1) -- (0.4,2.5) node[xic] {};
\draw[kernels2] (1,1) -- (1.6,2.5) node[xi] {};
}

\DeclareSymbol{2I1Xi4cc2}{2}{\draw[kernels2] (0,0) node[not] {} -- (-1,1);%62
\draw[kernels2] (0,0) -- (1,1) node[not] {};
\draw (-1,1)  node[xic] {} -- (-1,2.5) node[xic] {};
\draw[kernels2] (1,1) -- (0.4,2.5) node[xi] {};
\draw[kernels2] (1,1) -- (1.6,2.5) node[xi] {};
}

\DeclareSymbol{Xi4ba}{0}{\draw(-0.5,1.5) node[xi] {} -- (0,0); \draw (-1.5,1) node[xi] {} -- (0,0) node[not] {}; \draw[kernels2] (0,0) -- (1.5,1) node[xi] {};
\draw[kernels2] (0,0) -- (0.5,1.5) node[xi] {} ;}%63

\DeclareSymbol{Xi4badis}{0}{\draw(-0.5,1.5) node[xies] {} -- (0,0); \draw (-1.5,1) node[xies] {} -- (0,0) node[not] {}; \draw[kernels2] (0,0) -- (1.5,1) node[xies] {};
\draw[kernels2] (0,0) -- (0.5,1.5) node[xies] {} ;}%63

\DeclareSymbol{Xi4ba1}{0}{\draw(-0.5,1.5) node[xi] {} -- (0,0); \draw (-1.5,1) node[xi] {} -- (0,0) node[not] {}; \draw[kernels2] (0,0) -- (1.5,1) node[xic] {};
\draw[kernels2] (0,0) -- (0.5,1.5) node[xic] {} ;}%64

\DeclareSymbol{Xi4ba1b}{0}{\draw(-0.5,1.5) node[xic] {} -- (0,0); \draw (-1.5,1) node[xic] {} -- (0,0) node[not] {}; \draw[kernels2] (0,0) -- (1.5,1) node[xi] {};
\draw[kernels2] (0,0) -- (0.5,1.5) node[xi] {} ;}%64

\DeclareSymbol{Xi4ba1bdiff}{0}{\draw(-0.5,1.5) node[xic] {} -- (0,0); \draw (-1.5,1) node[xic] {} -- (0,0) node[not] {}; \draw (0,0) -- (1.5,1) node[xi] {};
\draw (0,0) -- (0.5,1.5) node[xi] {};
\draw(0,0) node[diff] {};}%64

\DeclareSymbol{Xi4ba1bb}{0}{\draw(-0.5,1.5) node[xic] {} -- (0,0); \draw (-1.5,1) node[xiesf] {} -- (0,0) node[not] {}; \draw[kernels2] (0,0) -- (1.5,1) node[xi] {};
\draw[kernels2] (0,0) -- (0.5,1.5) node[xi] {} ;}%64

\DeclareSymbol{Xi4ba2}{0}{\draw(-0.5,1.5) node[xi] {} -- (0,0); \draw (-1.5,1) node[xic] {} -- (0,0) node[not] {}; \draw[kernels2] (0,0) -- (1.5,1) node[xi] {};
\draw[kernels2] (0,0) -- (0.5,1.5) node[xic] {} ;}%65

\DeclareSymbol{Xi4ba2b}{0}{\draw(-0.5,1.5) node[xi] {} -- (0,0); \draw (-1.5,1) node[xic] {} -- (0,0) node[not] {}; \draw[kernels2] (0,0) -- (1.5,1) node[xi] {};
\draw[kernels2] (0,0) -- (0.5,1.5) node[xiesf] {} ;}%65

\DeclareSymbol{Xi4ca}{0}{\draw (0,1) -- (-1,2.2) node[xi] {};\draw (0,-0.25) node[xi] {} -- (0,1) ; \draw[kernels2] (0,1) node[not] {} -- (1,2.2) node[xi] {};%67
\draw[kernels2] (0,1) {} -- (0,2.7) node[xi] {};
}

\DeclareSymbol{Xi4cb}{0}{\draw (-1,1) -- (-2,2) node[xi] {};\draw[kernels2] (0,0)  -- (-1,1) node[xi] {} ; \draw[kernels2] (0,0) node[not] {} -- (1,1) node[xi] {} ; 
\draw (-1,1) node[xi] {} -- (0,2) node[xi] {};}

\DeclareSymbol{Xi4cbb}{0}{\draw (-1,1) -- (-2,2) node[xiesf] {};\draw[kernels2] (0,0)  -- (-1,1) node[xi] {} ; \draw[kernels2] (0,0) node[not] {} -- (1,1) node[xi] {} ; 
\draw (-1,1) node[xi] {} -- (0,2) node[xic] {};}

\DeclareSymbol{Xi4cbc1}{0}{\draw (-1,1) -- (-2,2) node[xic] {};\draw[kernels2] (0,0)  -- (-1,1) node[xic] {} ; \draw[kernels2] (0,0) node[not] {} -- (1,1) node[xi] {} ; 
\draw (-1,1) node[xic] {} -- (0,2) node[xi] {};}

\DeclareSymbol{Xi4cbc2}{0}{\draw (-1,1) -- (-2,2) node[xi] {};\draw[kernels2] (0,0)  -- (-1,1) node[xi] {} ; \draw[kernels2] (0,0) node[not] {} -- (1,1) node[xic] {} ; 
\draw (-1,1) node[xic] {} -- (0,2) node[xi] {};}

\DeclareSymbol{Xi4cab}{0}{\draw (-1,1) -- (-2,2) node[xi] {};\draw[kernels2] (0,0)  -- (-1,1); \draw[kernels2] (0,0) node[not] {} -- (1,1) node[xi] {} ; 
\draw[kernels2] (-1,1)  {} -- (0,2) node[xi] {};
\draw[kernels2] (-1,1) node[not] {} -- (-1,2.5) node[xi] {};
}%69

\DeclareSymbol{Xi4cabdis}{0}{\draw (-1,1) -- (-2,2) node[xies] {};\draw[kernels2] (0,0)  -- (-1,1); \draw[kernels2] (0,0) node[not] {} -- (1,1) node[xies] {} ; 
\draw[kernels2] (-1,1)  {} -- (0,2) node[xies] {};
\draw[kernels2] (-1,1) node[not] {} -- (-1,2.5) node[xies] {};
}%69

\DeclareSymbol{Xi4cabc1}{0}{\draw (-1,1) -- (-2,2) node[xi] {};\draw[kernels2] (0,0)  -- (-1,1); \draw[kernels2] (0,0) node[not] {} -- (1,1) node[xic] {} ; 
\draw[kernels2] (-1,1)  {} -- (0,2) node[xic] {};
\draw[kernels2] (-1,1) node[not] {} -- (-1,2.5) node[xi] {};
}%69

\DeclareSymbol{Xi4cabc2}{0}{\draw (-1,1) -- (-2,2) node[xic] {};\draw[kernels2] (0,0)  -- (-1,1); \draw[kernels2] (0,0) node[not] {} -- (1,1) node[xic] {} ; 
\draw[kernels2] (-1,1)  {} -- (0,2) node[xi] {};
\draw[kernels2] (-1,1) node[not] {} -- (-1,2.5) node[xi] {};
}%69

\DeclareSymbol{Xi4ea}{1.5}{\draw (-1,2.5) node[xi] {} -- (-1,1) node[xi] {} -- (0,0); %71
 \draw[kernels2] (0,0)  -- (1,1) node[xi] {};
\draw[kernels2] (0,0) node[not] {} -- (0,1.5) node[xi] {}; }

\DeclareSymbol{Xi4eac1}{1.5}{\draw (-1,2.5) node[xic] {} -- (-1,1) node[xi] {} -- (0,0); %72
 \draw[kernels2] (0,0)  -- (1,1) node[xic] {};
\draw[kernels2] (0,0) node[not] {} -- (0,1.5) node[xi] {}; }

\DeclareSymbol{Xi4eac1b}{1.5}{\draw (-1,2.5) node[xic] {} -- (-1,1) node[xi] {} -- (0,0); %72
 \draw[kernels2] (0,0)  -- (1,1) node[xiesf] {};
\draw[kernels2] (0,0) node[not] {} -- (0,1.5) node[xi] {}; }

\DeclareSymbol{Xi4eac2}{1.5}{\draw (-1,2.5) node[xic] {} -- (-1,1) node[xic] {} -- (0,0); %73
 \draw[kernels2] (0,0)  -- (1,1) node[xi] {};
\draw[kernels2] (0,0) node[not] {} -- (0,1.5) node[xi] {}; }

\DeclareSymbol{Xi4eact1}{1.5}{\draw (-1,2.5) node[xic] {} -- (-1,1) node[xi] {} -- (0,0); %74
 \draw (0,0)  -- (1,1) node[xic] {};
\draw[rho] (0,0) node[not] {} -- (0,1.5) node[xi] {}; }

\DeclareSymbol{Xi4eact2}{1.5}{\draw[rho] (-1,2.5) node[xic] {} -- (-1,1) node[xi] {} -- (0,0); %75
 \draw (0,0)  -- (1,1) node[xic] {};
\draw (0,0) node[not] {} -- (0,1.5) node[xi] {}; }

\DeclareSymbol{Xi4eabis}{1.5}{\draw (-1,2.5) node[xi] {} -- (-1,1) ; \draw[kernels2] (-1,1) node[xi] {} -- (0,0); 
 \draw (0,0)  -- (1,1) node[xi] {};%76
\draw[kernels2] (0,0) node[not] {} -- (0,1.5) node[xi] {}; }

\DeclareSymbol{Xi4eabisc1}{1.5}{\draw (-1,2.5) node[xic] {} -- (-1,1) ; \draw[kernels2] (-1,1) node[xi] {} -- (0,0); 
 \draw (0,0)  -- (1,1) node[xi] {};%77
\draw[kernels2] (0,0) node[not] {} -- (0,1.5) node[xic] {}; }

\DeclareSymbol{Xi4eabisc1b}{1.5}{\draw (-1,2.5) node[xic] {} -- (-1,1) ; \draw[kernels2] (-1,1) node[xi] {} -- (0,0); 
 \draw (0,0)  -- (1,1) node[xi] {};%77
\draw[kernels2] (0,0) node[not] {} -- (0,1.5) node[xiesf] {}; }

\DeclareSymbol{Xi4eabisc1bis}{1.5}{\draw (-1,2.5) node[xi] {} -- (-1,1) ; \draw[kernels2] (-1,1) node[xi] {} -- (0,0); 
 \draw (0,0)  -- (1,1) node[xi] {};%78
\draw[kernels2] (0,0) node[not] {} -- (0,1.5) node[xi] {};
\draw (-2,1) node[] {\tiny{$i$}};
\draw (-2,2.5) node[] {\tiny{$\ell$}};
\draw (2,1) node[] {\tiny{$k$}};
\draw (0,2.5) node[] {\tiny{$j$}};
 }

\DeclareSymbol{Xi4eabisc1tris}{1.5}{\draw (-1,2.5) node[xi] {} -- (-1,1) ; \draw[kernels2] (-1,1) node[xi] {} -- (0,0); 
 \draw (0,0)  -- (1,1) node[xi] {};%78
\draw[kernels2] (0,0) node[not] {} -- (0,1.5) node[xi] {};
\draw (-2,1) node[] {\tiny{i}};
\draw (-2,2.5) node[] {\tiny{j}};
\draw (2,1) node[] {\tiny{j}};
\draw (0,2.5) node[] {\tiny{i}};
 }

\DeclareSymbol{Xi4eabisc1quater}{1.5}{\draw (-1,2.5) node[xic] {} -- (-1,1) ; \draw[kernels2] (-1,1) node[xi] {} -- (0,0); 
 \draw (0,0)  -- (1,1) node[xic] {};%78
\draw[kernels2] (0,0) node[not] {} -- (0,1.5) node[xi] {};
 }

\DeclareSymbol{Xi4eabisc2}{1.5}{\draw (-1,2.5) node[xic] {} -- (-1,1) ; \draw[kernels2] (-1,1) node[xi] {} -- (0,0); %79
 \draw (0,0)  -- (1,1) node[xic] {};
\draw[kernels2] (0,0) node[not] {} -- (0,1.5) node[xi] {}; }

\DeclareSymbol{Xi4eabisc2l}{1.5}{\draw (-1,2.5) node[xiesf] {} -- (-1,1) ; \draw[kernels2] (-1,1) node[xi] {} -- (0,0); %79
 \draw (0,0)  -- (1,1) node[xic] {};
\draw[kernels2] (0,0) node[not] {} -- (0,1.5) node[xi] {}; }

\DeclareSymbol{Xi4eabisc2r}{1.5}{\draw (-1,2.5) node[xic] {} -- (-1,1) ; \draw[kernels2] (-1,1) node[xi] {} -- (0,0); %79
 \draw (0,0)  -- (1,1) node[xiesf] {};
\draw[kernels2] (0,0) node[not] {} -- (0,1.5) node[xi] {}; }

\DeclareSymbol{Xi4eabisc3}{1.5}{\draw (-1,2.5) node[xic] {} -- (-1,1) ; \draw[kernels2] (-1,1) node[xic] {} -- (0,0); %80
 \draw (0,0)  -- (1,1) node[xi] {};
\draw[kernels2] (0,0) node[not] {} -- (0,1.5) node[xi] {}; }

\DeclareSymbol{Xi4eb}{0}{%81
\draw[kernels2] (0,2) node[xi] {} -- (-1,1) ; \draw[kernels2] (-2,2)  node[xi] {} -- (-1,1) ; \draw (-1,1)  node[not] {} -- (0,0); 
 \draw (0,0) node[xi] {}  -- (1,1) node[xi] {};
}

\DeclareSymbol{Xi4eab}{1.5}{\draw[kernels2] (-1,2.5) node[xi] {} -- (-1,1) ; \draw[kernels2] (-2,2)  node[xi] {} -- (-1,1) ; \draw (-1,1)  node[not] {} -- (0,0); %82
 \draw[kernels2] (0,0)  -- (1,1) node[xi] {};
\draw[kernels2] (0,0) node[not] {} -- (0,1.5) node[xi] {}; 
}

\DeclareSymbol{Xi4eabdis}{1.5}{\draw[kernels2] (-1,2.5) node[xies] {} -- (-1,1) ; \draw[kernels2] (-2,2)  node[xies] {} -- (-1,1) ; \draw (-1,1)  node[not] {} -- (0,0); %82
 \draw[kernels2] (0,0)  -- (1,1) node[xies] {};
\draw[kernels2] (0,0) node[not] {} -- (0,1.5) node[xies] {}; 
}

\DeclareSymbol{Xi4eabc1}{1.5}{\draw[kernels2] (-1,2.5) node[xic] {} -- (-1,1) ; \draw[kernels2] (-2,2)  node[xi] {} -- (-1,1) ; \draw (-1,1)  node[not] {} -- (0,0); %83
 \draw[kernels2] (0,0)  -- (1,1) node[xic] {};
\draw[kernels2] (0,0) node[not] {} -- (0,1.5) node[xi] {}; 
}

\DeclareSymbol{Xi4eabc2}{1.5}{\draw[kernels2] (-1,2.5) node[xi] {} -- (-1,1) ; \draw[kernels2] (-2,2)  node[xi] {} -- (-1,1) ; \draw (-1,1)  node[not] {} -- (0,0); %84
 \draw[kernels2] (0,0)  -- (1,1) node[xic] {};
\draw[kernels2] (0,0) node[not] {} -- (0,1.5) node[xic] {}; 
}

\DeclareSymbol{Xi4eabbis}{1.5}{\draw[kernels2] (-1,2.5) node[xi] {} -- (-1,1) ; \draw[kernels2] (-2,2)  node[xi] {} -- (-1,1) ; \draw[kernels2] (-1,1)  node[not] {} -- (0,0); %85
 \draw (0,0)  -- (1,1) node[xi] {};
\draw[kernels2] (0,0) node[not] {} -- (0,1.5) node[xi] {}; 
}

\DeclareSymbol{Xi4eabbisc1}{1.5}{\draw[kernels2] (-1,2.5) node[xic] {} -- (-1,1) ; \draw[kernels2] (-2,2)  node[xi] {} -- (-1,1) ; \draw[kernels2] (-1,1)  node[not] {} -- (0,0); %86
 \draw (0,0)  -- (1,1) node[xic] {};
\draw[kernels2] (0,0) node[not] {} -- (0,1.5) node[xi] {}; 
}

\DeclareSymbol{Xi4eabbisc1perm}{1.5}{\draw[kernels2] (-1,2.5) node[xi] {} -- (-1,1) ; \draw[kernels2] (-2,2)  node[xic] {} -- (-1,1) ; \draw[kernels2] (-1,1)  node[not] {} -- (0,0); %86
 \draw (0,0)  -- (1,1) node[xic] {};
\draw[kernels2] (0,0) node[not] {} -- (0,1.5) node[xi] {}; 
}

\DeclareSymbol{Xi4eabbisc2}{1.5}{\draw[kernels2] (-1,2.5) node[xi] {} -- (-1,1) ; \draw[kernels2] (-2,2)  node[xi] {} -- (-1,1) ; \draw[kernels2] (-1,1)  node[not] {} -- (0,0); %87
 \draw (0,0)  -- (1,1) node[xic] {};
\draw[kernels2] (0,0) node[not] {} -- (0,1.5) node[xic] {}; 
}

\DeclareSymbol{Xi2cbis}{0}{\draw[kernels2] (0,1) -- (0.8,2.2) node[xi] {};\draw[kernels2] (0,-0.25) node[not] {} -- (0,1); \draw[kernels2] (0,1) node[not] {} -- (-0.8,2.2) node[xi] {};}%88

\DeclareSymbol{Xi2cbis1}{0}{\draw (0,1) -- (-0.8,2.2) node[xi] {};\draw[kernels2] (0,-0.25) node[not] {} -- (0,1) node[xi] {}; }%89

\DeclareSymbol{Xi2Xbis}{-2}{\draw[kernels2] (0,-0.25)  -- (-1,1) ; \draw (-1,1) node[xix] {};%91
\draw[kernels2] (0,-0.25) node[not] {} -- (1,1) node[xi] {};}

\DeclareSymbol{XXi2bis}{-2}{\draw[kernels2] (0,-0.25) -- (-1,1) node[xi] {};%92
\draw[kernels2] (0,-0.25) node[X] {} -- (1,1) node[xi] {};}

\DeclareSymbol{I1XiIXi}{0}{\draw[kernels2] (0,-0.25) -- (1,1) node[xi] {};%93
\draw (0,-0.25) node[not] {} -- (-1,1) node[xi] {};}

\DeclareSymbol{I1XiIXib}{0}{\draw  (0,-0.25) node[xi] {} -- (0,1) node[not] {};
\draw[kernels2] (0,1) -- (0,2.25) ; \draw (0,2.25) node[xi]{}; }%94

\DeclareSymbol{I1XiIXic}{0}{%95
\draw[kernels2] (0,0) -- (1,1) node[xi] {} ; 
\draw[kernels2] (0,0) node[not] {}  -- (-1,1) node[not] {} -- (0,2) node[xi] {};
}

\DeclareSymbol{thin}{1.4}{\draw[pagebackground] (-0.3,0) -- (0.3,0); \draw  (0,0) -- (0,2);}
\DeclareSymbol{thin2}{1.4}{\draw[pagebackground] (-0.3,0) -- (0.3,0); \draw[tinydots]  (0,0) -- (0,2);}

\DeclareSymbol{thick}{1.4}{\draw[pagebackground] (-0.3,0) -- (0.3,0); \draw[kernels2]  (0,0) -- (0,2);}
\DeclareSymbol{thick2}{1.4}{\draw[pagebackground] (-0.3,0) -- (0.3,0); \draw[kernels2,tinydots]  (0,0) -- (0,2);}

\DeclareSymbol{Xi4ind}{2}{\draw (0,0) node[xi,label={[label distance=-0.2em]right: \scriptsize  $ i $}]  { } -- (-1,1) node[xi,label={[label distance=-0.2em]left: \scriptsize  $ j $}] {} -- (0,2) node[xi,label={[label distance=-0.2em]right: \scriptsize  $ k $}] {} -- (-1,3) node[xi,label={[label distance=-0.2em]left: \scriptsize  $ \ell $}] {};}%115

\DeclareSymbol{Xi4c1}{2}{\draw (0,0) node[xic] {} -- (-1,1) node[xi] {} -- (0,2) node[xic] {} -- (-1,3) node[xi] {};} %119
\DeclareSymbol{IXi2ex}{0}{\draw (0,-0.25) node[xie] {} -- (-1,1) node[xi] {} ; \draw (0,-0.25)-- (1,1) node[xi] {};}
\DeclareSymbol{IXi2ex1}{0}{\draw (0,-0.25) node[xie] {} -- (-1,1) node[xi] {} -- (0,2) node[xi] {};}

\DeclareSymbol{Xi4b1}{0}{\draw(0,1.5) node[xic] {} -- (0,0); \draw (-1,1) node[xic] {} -- (0,0) node[xi] {} -- (1,1) node[xi] {};}

\DeclareSymbol{Xi4ec1}{0}{\draw (0,2) node[xi] {} -- (-1,1) node[xic] {} -- (0,0) node[xic] {} -- (1,1) node[xi] {};}
\DeclareSymbol{Xi4ec2}{0}{\draw (0,2) node[xic] {} -- (-1,1) node[xi] {} -- (0,0) node[xic] {} -- (1,1) node[xi] {};}
\DeclareSymbol{Xi4ec3}{0}{\draw (0,2) node[xic] {} -- (-1,1) node[xic] {} -- (0,0) node[xi] {} -- (1,1) node[xi] {};}

\DeclareSymbol{I1Xi4ac1}{2}{\draw[kernels2] (0,0) node[not] {} -- (-1,1) ; \draw[kernels2] (0,0) node[not] {} -- (1,1) node[xic] {} ;
\draw (-1,1) node[xi] {} -- (0,2) node[xic] {} -- (-1,3) node[xi] {};}%161

\DeclareSymbol{I1Xi4ac2}{2}{\draw[kernels2] (0,0) node[not] {} -- (-1,1) ; \draw[kernels2] (0,0) node[not] {} -- (1,1) node[xic] {} ;
\draw (-1,1) node[xi] {} -- (0,2) node[xi] {} -- (-1,3) node[xic] {};}

\DeclareSymbol{I1Xi4bp}{2}{\draw (0,0) node[not] {} -- (-1,1) node[xi] {} -- (0,2) ; \draw[kernels2] (0,2) node[not] {} -- (-1,3) node[xi] {};\draw[kernels2] (0,2)  -- (1,3) node[xi] {};
}

\DeclareSymbol{I1Xi4bc1}{2}{\draw (0,0) node[xic] {} -- (-1,1) node[xi] {} -- (0,2) ; \draw[kernels2] (0,2) node[not] {} -- (-1,3) node[xi] {};\draw[kernels2] (0,2)  -- (1,3) node[xic] {};
}%165

\DeclareSymbol{I1Xi4bc2}{2}{\draw (0,0) node[xic] {} -- (-1,1) node[xi] {} -- (0,2) ; \draw[kernels2] (0,2) node[not] {} -- (-1,3) node[xic] {};\draw[kernels2] (0,2)  -- (1,3) node[xi] {};
}

\DeclareSymbol{I1Xi4cp}{2}{\draw (0,0) node[not] {} -- (-1,1) node[not] {}; \draw[kernels2] (-1,1) -- (0,2) ; 
\draw[kernels2] (-1,1) -- (-2,2) node[xi] {} ;
\draw (0,2) node[xi] {} -- (-1,3) node[xi] {};}

\DeclareSymbol{I1Xi4cc1}{2}{\draw (0,0) node[xic] {} -- (-1,1) node[not] {}; \draw[kernels2] (-1,1) -- (0,2) ; 
\draw[kernels2] (-1,1) -- (-2,2) node[xi] {} ;
\draw (0,2) node[xic] {} -- (-1,3) node[xi] {};}%169

\DeclareSymbol{I1Xi4cc2}{2}{\draw (0,0) node[xic] {} -- (-1,1) node[not] {}; \draw[kernels2] (-1,1) -- (0,2) ; 
\draw[kernels2] (-1,1) -- (-2,2) node[xi] {} ;
\draw (0,2) node[xi] {} -- (-1,3) node[xic] {};}

\DeclareSymbol{I1Xi4abc1}{2}{\draw[kernels2] (0,0) node[not] {} -- (-1,1) ; \draw[kernels2] (0,0) node[not] {} -- (1,1) node[xic] {};\draw (-1,1) node[xi] {} -- (0,2) ; \draw[kernels2] (0,2) node[not] {} -- (-1,3) node[xic] {};\draw[kernels2] (0,2)  -- (1,3) node[xi] {}; }

\DeclareSymbol{I1Xi4abc2}{2}{\draw[kernels2] (0,0) node[not] {} -- (-1,1) ; \draw[kernels2] (0,0) node[not] {} -- (1,1) node[xic] {};\draw (-1,1) node[xi] {} -- (0,2) ; \draw[kernels2] (0,2) node[not] {} -- (-1,3) node[xi] {};\draw[kernels2] (0,2)  -- (1,3) node[xic] {}; }%173

\DeclareSymbol{R1}{0}{\draw (-1,1) node[xi] {} -- (0,0) node[not] {};%131
\draw[kernels2] (0,1.5) node[xic] {} -- (0,0) -- (1,1) node[xic] {};}
\DeclareSymbol{R2}{0}{\draw (-1,1) node[xic] {} -- (0,0) node[not] {};
\draw[kernels2] (0,1.5)  {} -- (0,0) -- (1,1)  {};
\draw (0,1.5) node[xi] {};
\draw (1,1) node[xic] {};
}%132
\DeclareSymbol{R3}{1}{\draw[kernels2] (-1,1.5)  {} -- (0,0) node[not] {} -- (1,1.5);%133
\draw (-1,1.5) node[xi] {};
\draw[kernels2] (0,3) {} -- (1,1.5) -- (2,3)  {};
\draw  (0,3) node[xic] {} ;
\draw (2,3) node[xic] {};}
\DeclareSymbol{R4}{1}{\draw[kernels2] (-1,1.5) node[xic] {} -- (0,0) node[not] {} -- (1,1.5);%134
\draw[kernels2] (0,3) {} -- (1,1.5) -- (2,3) node[xic] {};
\draw (0,3) node[xi] {};}

\DeclareSymbol{I1Xi4bcp}{2}{\draw (0,0) node[not] {} -- (-1,1) node[not] {}; \draw[kernels2] (-1,1) -- (0,2) ; %175
\draw[kernels2] (-1,1) -- (-2,2) node[xi] {} ; \draw[kernels2] (0,2) node[not] {} -- (-1,3) node[xi] {};\draw[kernels2] (0,2)  -- (1,3) node[xi] {};
}

\DeclareSymbol{I1Xi4bcc1}{2}{\draw (0,0) node[xic] {} -- (-1,1) node[not] {}; \draw[kernels2] (-1,1) -- (0,2) ; 
\draw[kernels2] (-1,1) -- (-2,2) node[xi] {} ; \draw[kernels2] (0,2) node[not] {} -- (-1,3) node[xi] {};\draw[kernels2] (0,2)  -- (1,3) node[xic] {};
}

\DeclareSymbol{I1Xi4bcc2}{2}{\draw (0,0) node[xic] {} -- (-1,1) node[not] {}; \draw[kernels2] (-1,1) -- (0,2) ; 
\draw[kernels2] (-1,1) -- (-2,2) node[xi] {} ; \draw[kernels2] (0,2) node[not] {} -- (-1,3) node[xic] {};\draw[kernels2] (0,2)  -- (1,3) node[xi] {};
} %177

\DeclareSymbol{2I1Xi4bc1}{2}{\draw[kernels2] (0,0) node[not] {} -- (-1,1) ;
\draw[kernels2] (0,0) -- (1,1);
\draw (-1,1) node[xic] {} -- (-1,2.5) node[xi] {};
\draw (1,1)  node[xic] {} -- (1,2.5) node[xi] {};
}

\DeclareSymbol{2I1Xi4bc2}{2}{\draw[kernels2] (0,0) node[not] {} -- (-1,1) ;
\draw[kernels2] (0,0) -- (1,1);
\draw (-1,1) node[xi] {} -- (-1,2.5) node[xic] {};
\draw (1,1)  node[xic] {} -- (1,2.5) node[xi] {};
}

\DeclareSymbol{diff2I1Xi4bc2}{2}{\draw (0,0) node[diff] {} -- (-1,1) ;
\draw (0,0) -- (1,1);
\draw (-1,1) node[xi] {} -- (-1,2.5) node[xic] {};
\draw (1,1)  node[xic] {} -- (1,2.5) node[xi] {};
}

\DeclareSymbol{2I1Xi4bc3}{2}{\draw[kernels2] (0,0) node[not] {} -- (-1,1) ;
\draw[kernels2] (0,0) -- (1,1);
\draw (-1,1) node[xic] {} -- (-1,2.5) node[xic] {};
\draw (1,1)  node[xi] {} -- (1,2.5) node[xi] {};
}

\DeclareSymbol{Xi41}{0}{\draw (0,1) -- (0.8,2.2) node[xic] {};\draw (0,-0.25) node[xi] {} -- (0,1) node[xi] {} -- (-0.8,2.2) node[xic] {};} 

\DeclareSymbol{Xi42}{0}{\draw (0,1) -- (0.8,2.2) node[xi] {};\draw (0,-0.25) node[xic] {} -- (0,1) node[xi] {} -- (-0.8,2.2) node[xic] {};}

\DeclareSymbol{Xi4ca1}{0}{\draw (0,1) -- (-1,2.2) node[xic] {};\draw (0,-0.25) node[xi] {} -- (0,1) ; \draw[kernels2] (0,1) node[not] {} -- (1,2.2) node[xic] {};
\draw[kernels2] (0,1) {} -- (0,2.7) node[xi] {};
}

\DeclareSymbol{Xi4ca2}{0}{\draw (0,1) -- (-1,2.2) node[xi] {};\draw (0,-0.25) node[xi] {} -- (0,1) ; \draw[kernels2] (0,1) node[not] {} -- (1,2.2) node[xic] {};
\draw[kernels2] (0,1) {} -- (0,2.7) node[xic] {};
}

\DeclareSymbol{Xi4cap}{0}{\draw (0,1) -- (-1,2.2) node[xi] {};\draw (0,-0.25) node[not] {} -- (0,1) ; \draw[kernels2] (0,1) node[not] {} -- (1,2.2) node[xi] {};
\draw[kernels2] (0,1) {} -- (0,2.7) node[xi] {};
}

\DeclareSymbol{Xi3a}{0}{
 \draw (-1,1)  node[xi] {} -- (0,0); 
 \draw (0,0) node[xi] {}  -- (1,1) node[xi] {};
}

\DeclareSymbol{Xi4ebc1}{0}{
\draw[kernels2] (0,2) node[xi] {} -- (-1,1) ; \draw[kernels2] (-2,2)  node[xic] {} -- (-1,1) ; \draw (-1,1)  node[not] {} -- (0,0); 
 \draw (0,0) node[xic] {}  -- (1,1) node[xi] {};
}

\DeclareSymbol{Xi4ebc2}{0}{
\draw[kernels2] (0,2) node[xi] {} -- (-1,1) ; \draw[kernels2] (-2,2)  node[xi] {} -- (-1,1) ; \draw (-1,1)  node[not] {} -- (0,0); 
 \draw (0,0) node[xic] {}  -- (1,1) node[xic] {};
}

\DeclareSymbol{Xi2cbispex}{0}{\draw[kernels2] (0,1) -- (0.8,2.2) node[xi] {};\draw (0,-0.25) node[xie] {} -- (0,1); \draw[kernels2] (0,1) node[not] {} -- (-0.8,2.2) node[xi] {};}

\DeclareSymbol{Xi2cbis1p}{0}{\draw (0,1) -- (-0.8,2.2) node[xi] {};\draw (0,-0.25) node[not] {} -- (0,1) node[xi] {}; }

\DeclareSymbol{Xi2Xp}{-2}{\draw (0,-0.25) node[not] {} -- (-1,1) node[xix] {};} % 229 not used

\DeclareSymbol{I1XiIXib}{0}{\draw  (0,-0.25) node[xi] {} -- (0,1) node[not] {};
\draw[kernels2] (0,1) -- (0,2.25) ; \draw (0,2.25) node[xi]{}; }

\DeclareSymbol{IXi2b}{0}{\draw  (0,-0.25) node[xi] {} -- (0,1) node[not] {};
\draw (0,1) -- (0,2.25) ; \draw (0,2.25) node[xi]{}; }

\DeclareSymbol{IXi2bex}{0}{\draw  (0,-0.25) node[xi] {} -- (0,1) node[xie] {};
\draw (0,1) -- (0,2.25) ; \draw (0,2.25) node[xi]{}; }

 \def\1{\mathbf{\symbol{1}}}

\def\one{\mathbf{1}}
\def\eps{\varepsilon}

\DeclareSymbol{diff}{0}{
\draw (0,0.5) node[diff] {};
}

\DeclareSymbol{diff1}{0}{
\draw (0,0.5) node[diff1] {};
}

\DeclareSymbol{diff2}{0}{
\draw (0,0.5) node[diff2] {};
}

\DeclareSymbol{geo}{0}{
\draw (0,0) node[diff] {};
\draw (0.3,0) node[diff] {};
}

\DeclareSymbol{generic}{0}{
\draw (0,0.6) node[xi] {};
}

\DeclareSymbol{g}{0}{
\draw (0,0.6) node[g] {};
}

\DeclareSymbol{Ito}{0}{
\draw (0,0.6) node[xies] {};
}

\DeclareSymbol{Itob}{0}{
\draw (0,0.6) node[xiesf] {};
}

\DeclareSymbol{greycirc}{0}{
\draw (0,0.3) node[xi] {};
}

\DeclareSymbol{not}{0}{
\draw (0,0.6) node[not] {};
\draw[tinydots] (0,0.6) circle (0.8);
}

\DeclareSymbol{genericb}{0}{
\draw (0,0.6) node[xic] {};
}

\DeclareSymbol{bluecirc}{0}{
\draw (0,0.3) node[xic] {};
}

\DeclareSymbol{genericxix}{0}{
\draw (0,0.6) node[xix] {};
}

\DeclareSymbol{genericX}{0}{
\draw (0,0.6) node[X] {};
}

\DeclareSymbol{diffIto}{1}{
\draw  (0,2.5) -- (0,0) ;
\draw (0,-0.1) node[diff] {};
\draw (0,2.5) node[xies] {};
}
\DeclareSymbol{Itodiff}{2}{
\draw(0,2.9) -- (0,-0.2);
\draw (0,2.9) node[diff] {};
\draw (0,-0.1) node[xies] {};
}

\DeclareSymbol{diffgeneric}{1}{
\draw  (0,2.5) -- (0,0) ;
\draw (0,-0.1) node[diff] {};
\draw (0,2.5) node[xi] {};
}

\DeclareSymbol{genericdiff}{2}{
\draw(0,2.9) -- (0,-0.2);
\draw (0,2.9) node[diff] {};
\draw (0,-0.1) node[xi] {};
}

\DeclareSymbol{diffdot}{2}{
\draw  (0,3) -- (0,-0.1) ;
\draw (0,3) node[not] {};
\draw (0,-0.1) node[diff] {};
}

\DeclareSymbol{diffdotmini}{0}{
\draw  (0,0) -- (0,1.2) ;
\draw (0,1.2) node[not] {};
\draw (0,0) node[diffmini] {};
}

\DeclareSymbol{dotdiff}{2}{
\draw[kernelsmod]  (0,3) -- (0,-0.1) ;
\draw (0,3) node[diff] {};
\draw (0,-0.1) node[not] {};
}

\DeclareSymbol{3}{-2}{\draw (0,-0.25) node[xi] {} -- (-1,1) node[xi] {};}
\DeclareSymbol{AAA}{-0.5}{\draw (0,0) node[xi] {} -- (-1,1) node[xi] {} -- (0,2) node[xi] {};
	\draw (0,0) node[xi] {} -- (-1,-1) node[xi] {};}
\DeclareSymbol{AAM}{0.5}{\draw (-1,1)  -- (0,2) node[xi] {};
	\draw (-1,1) node[xi] {} -- (0,0) node[xi] {} -- (1,1) node[xi] {};}
\DeclareSymbol{AMA}{-4}{\draw (-1,1) node[xi] {} -- (0,0) node[xi] {} -- (1,1) node[xi] {};
	\draw (0,0) node[xi] {} -- (0,-1.4) node[xi] {};}
\DeclareSymbol{AMM}{0.5}{\draw (0,0)  -- (0,1.4) node[xi] {};
	\draw (-1,1) node[xi] {} -- (0,0) node[xi] {} -- (1,1) node[xi] {};}

\DeclareSymbol{dotdiff1}{2}{
\draw[kernelsmod]  (0,3) -- (0,-0.1) ;
\draw (0,3) node[diff1] {};
\draw (0,-0.1) node[not] {};
}

\DeclareSymbol{dotdiff1mini}{0}{
\draw[kernelsmod]  (0,1.2) -- (0,0) ;
\draw (0,1.2) node[diffmini] {};
\draw (0,0) node[not] {};
}

\DeclareSymbol{dotdiff2}{2}{
\draw (0,3) -- (0,-0.1) ;
\draw (0,3) node[diff] {};
\draw (0,-0.1) node[not] {};
}

\DeclareSymbol{dotdiff2mini}{0}{
\draw (0,1.2) -- (0,0) ;
\draw (0,1.2) node[diffmini] {};
\draw (0,0) node[not] {};
}

\DeclareSymbol{dotdiffstraight}{0}{
\draw  (0,3) -- (0,-0.1) ;
\draw (0,3) node[diff] {};
\draw (0,-0.1) node[not] {};
}

\DeclareSymbol{arbre1}{0}{
\draw  (0,0) -- (1.5,1.5) ;
\draw (1.5,1.5) node[not] {};
\draw (0,0) node[not] {};
}

\DeclareSymbol{arbre2}{0}{
\draw  (0,0) -- (1.5,1.5) ;
\draw[kernelsmod] (0,0) -- (-1.5,1.5);
\draw (1.5,1.5) node[not] {};
\draw (0,0) node[not] {};
\draw (-1.5,1.5) node[xi] {};
}

\DeclareSymbol{arbre3}{0}{
\draw  (0,0) -- (1.5,1.5) ;
\draw[kernelsmod] (1.5,1.5) -- (0,3);
\draw (0,0) node[not] {};
\draw (1.5,1.5) node[not] {};
\draw (0,3) node[xi] {};
}

\DeclareSymbol{treeeval}{0}{
\draw (0,0) -- (1,1);
\draw (0,0) node[xi] {};
\draw (1.25,1.25) node[xi] {};
\draw (-0.6,0.6) node[]{\tiny{$i$}};
\draw (0.65,1.85) node[]{\tiny{$j$}};
}

\DeclareSymbol{testeval}{0}{
\draw (0,0) -- (1,1);
\draw (0,0) -- (-1,1);
\draw (0,0) node[xi] {};
\draw (1.25,1.25) node[xi] {};
\draw (-1.25,1.25) node[xi] {};
\draw (-0.6,-0.6) node[]{\tiny{$i$}};
\draw (0.65,1.85) node[]{\tiny{$j$}};
\draw (-1.95,1.85) node[]{\tiny{$k$}};
}

\DeclareSymbol{treeeval2}{0}{
\draw[kernelsmod] (-0.25,-1) -- (1,0.5) ;
\draw[kernelsmod] (1,0.5) -- (-0.25,2);
\draw (1,0.5) node[diff2] {};
\draw (-0.25,-1) node[not] {};
\draw (-0.25,2) node[xi] {};
\draw (-0.6,1.2) node[]{\tiny{1}};
}

\DeclareSymbol{arbreact}{1}{
\draw (0,0) node[not] {};
\draw[kernelsmod] (0,0) -- (1,1);
\draw[kernelsmod] (0,0) -- (-1,1);
\draw (-1,1) node[xic] {};
\draw  (0,2) -- (1,1) ;
\draw (0,2) node[xic] {};
\draw (1,1) node[xi] {};
}

\DeclareSymbol{arbreact1}{0}{
\draw (0,-1.5) -- (0,0);
\draw[kernelsmod] (0,0) -- (1,1);
\draw[kernelsmod] (0,0) -- (-1,1);
\draw  (0,2) -- (1,1) ;
\draw (0,-1.5) node[diff] {};
\draw (0,0) node[not] {};
\draw (-1,1) node[xic] {};
\draw (0,2) node[xic] {};
\draw (1,1) node[xi] {};
}

\DeclareSymbol{arbreact2}{0}{
\draw (0,-0.75) -- (-1,0.5); 
\draw (0,-0.75) -- (1,0.5);
\draw (0,1.5) -- (1,0.5);
\draw (0,1.5) node[xic] {};
\draw (1,0.5) node[xi] {};
\draw (-1,0.5) node[xic] {};
\draw (0,-0.75) node[diff] {};
}

\DeclareSymbol{arbreact3}{0}{
\draw[kernelsmod] (0,-0.75) -- (-1,0.5); 
\draw[kernelsmod] (0,-0.75) -- (1,0.5);
\draw (0,1.75) -- (1,0.5);
\draw (2,1.75) -- (1,0.5);
\draw (0,1.75) node[xic] {};
\draw (1,0.5) node[diff] {};
\draw (-1,0.5) node[xic] {};
\draw (2,1.75) node[xi] {};
\draw (0,-0.75) node[not] {};
}

\DeclareSymbol{pre_im_1}{0}{
\draw[kernels2] (0,-0.5) node[not] {} -- (-0.6,0.5) ;
\draw[kernels2] (0,-0.5) -- (0.6,0.5);
\draw (0,1.1)  -- (-0.55,2);
\draw (0,1.1)  -- (0.55,2);
\draw (0,0.7) node[g] {};
\draw (0,2.2) node[g] {};
}

\DeclareSymbol{disconnect}{0}{
\draw[kernels2] (0,-0.5) node[not] {} -- (-0.6,0.5) ;
\draw[kernels2] (0,-0.5) -- (0.6,0.5);
\draw (-0.55,1.1)  -- (-0.55,2.3);
\draw (0.55,2.3) -- (0.55,1.5) -- (1.2,1.5) -- (1.2,3.5) -- (0.55,3.5) -- (0.55,2.7);
\draw (0,0.7) node[g] {};
\draw (0,2.5) node[g] {};
}

\DeclareSymbol{pre_im_2}{2}{\draw[kernels2] (0,0) node[not] {} -- (-1,1) node[not] {};
\draw[kernels2] (0,0) -- (1,1) node[not] {};
\draw[kernels2] (-1,1) -- (-1.5,2.5);
\draw[kernels2] (-1,1) -- (-0.5,2.5);
\draw[kernels2] (1,1) -- (0.5,2.5);
\draw[kernels2] (1,1) -- (1.5,2.5);
\draw (-1,2.7) node[g] {};
\draw (1,2.7) node[g] {};
}

\DeclareSymbol{CX_rec}{0}{
\draw [black] (-0.3,1) to (-0.3,-0.3);
\draw [black] (0.3,1) to (0.3,-0.3);
\draw [black] (-0.3,1) to (-0.3,2.3);
\draw [black] (0.3,1) to (0.3,2.3);
\draw (0,1) node[rec] {};
}

\DeclareSymbol{CX_cerc}{0}{
\draw [black] (0,1) to (0,-0.3);
\draw (0,1) node[cerc] {};
}

\DeclareSymbol{proof0}{0}{
\draw (0,-3) node[] {\tiny$\tau_0$};
\draw (0,-2.3) -- (0,0.5);`
\draw (0,0.5) -- (-1.5,2.5);
\draw (0,0.5) -- (1.5,2.5);
\draw (-1,-2) node[] {\tiny$v$};
\draw (-1.5,3.1) node[] {\tiny$\tau_1$};
\draw (1.5,3.1) node[] {\tiny$\tau_2$};
\draw (0,0.5) node[diff] {};
}

\DeclareSymbol{proof0b}{0}{
\draw (0,0.5) -- (-1.5,2.5);
\draw (0,0.5) -- (1.5,2.5);
\draw (-1.5,3.1) node[] {\tiny$\tau_1$};
\draw (1.5,3.1) node[] {\tiny$\tau_2$};
\draw (0,0.5) node[diff] {};
}

\DeclareSymbol{proof}{0}{
\draw[kernelsmod] (-2,3) -- (0,0);
\draw[kernelsmod] (2,3) -- (0,0);
\draw (0,0) node[not] {};
}

\DeclareSymbol{prooftri}{0}{
\draw[kernelsmod] (-2,3) -- (0,0);
\draw[kernelsmod] (0,4) -- (0,0);
\draw (2,2.7)--(0,0);
\draw (0,0) node[not] {};
}

\DeclareSymbol{proofqua}{0}{
\draw[kernelsmod] (-3,3) -- (0,0);
\draw[kernelsmod] (-1,4) -- (0,0);
\draw (1,3.6)--(0,0);
\draw (3,2.7)--(0,0);
\draw (0,0) node[not] {};
}

\DeclareSymbol{proof1_1}{0}{
\draw (0,-2.7) node[] {\tiny$\tau_0$};
\draw (0,-2) -- (0,0.5);`
\draw[kernelsmod] (0,0.5) -- (-1.5,2.5); 
\draw[kernelsmod] (0,0.5) -- (1.5,2.5);
\draw (-1,-1.7) node[] {\tiny$v$};
\draw (-1.5,3.1) node[] {\tiny$\tau_1$};
\draw (1.5,3.1) node[] {\tiny$\tau_2$};
\draw (0,0.5) node[not] {};
}

\DeclareSymbol{proof1b_1}{0}{
\draw[kernelsmod] (0,0.5) -- (-1.5,2.5); 
\draw[kernelsmod] (0,0.5) -- (1.5,2.5);
\draw (-1.5,3.1) node[] {\tiny$\tau_1$};
\draw (1.5,3.1) node[] {\tiny$\tau_2$};
\draw (0,0.5) node[not] {};
}

\DeclareSymbol{proof1_2}{0}{
\draw (0,-2.7) node[] {\tiny$\tau_0$};
\draw[kernelsmod] (0,-2) -- (0,0.5);`
\draw[kernelsmod] (0,0.5) -- (-1.5,2.5); 
\draw[kernelsmod] (0,0.5) -- (1.5,2.5);
\draw (-1,-1.7) node[] {\tiny$v$};
\draw (-1.5,3.1) node[] {\tiny$\tau_1$};
\draw (1.5,3.1) node[] {\tiny$\tau_2$};
\draw (0,0.5) node[not] {};
}

\DeclareSymbol{proof2_1}{0}{
\draw (0,-2.7) node[] {\tiny$\tau_0$};
\draw (0,-2) -- (-1.8,-0.3);
\draw (-1.8,0.7) -- (0,2.5); 
\draw (1,-1.7) node[] {\tiny$v$};
\draw (-2.5,1.2) node[] {\tiny$r_1$};
\draw (0,3.1) node[] {\tiny$\tau_2$};
\draw (-2.5,0) node[] {\tiny$\tau_1$};
}

\DeclareSymbol{proof2b_1}{0}{
\draw (-1.8,0.7) -- (0,2.5); 
\draw (-2.5,1.2) node[] {\tiny$r_1$};
\draw (0,3.1) node[] {\tiny$\tau_2$};
\draw (-2.5,0) node[] {\tiny$\tau_1$};
}

\DeclareSymbol{proof2b_2}{0}{
\draw (-1.8,0.7) -- (0,2.5); 
\draw (-2.5,1.2) node[] {\tiny$r_2$};
\draw (0,3.1) node[] {\tiny$\tau_1$};
\draw (-2.5,0) node[] {\tiny$\tau_2$};
}

\DeclareSymbol{proof2_2}{0}{
\draw (0,-2.7) node[] {\tiny$\tau_0$};
\draw[kernelsmod] (0,-2) -- (-1.8,-0.3);
\draw (-1.8,0.7) -- (0,2.5); 
\draw (1,-1.7) node[] {\tiny$v$};
\draw (-2.5,1.2) node[] {\tiny$r_1$};
\draw (0,3.1) node[] {\tiny$\tau_2$};
\draw (-2.5,0) node[] {\tiny$\tau_1$};
}

\DeclareSymbol{proof3_1}{0}{
\draw (0,-2.7) node[] {\tiny$\tau_0$};
\draw (0,-2) -- (1.8,-0.3);
\draw (1.8,0.7) -- (0,2.5); 
\draw (-1,-1.7) node[] {\tiny$v$};
\draw (2.5,1.2) node[] {\tiny$r_2$};
\draw (0,3.1) node[] {\tiny$\tau_1$};
\draw (2.5,0) node[] {\tiny$\tau_2$};
}

\DeclareSymbol{proof3_2}{0}{
\draw (0,-2.7) node[] {\tiny$\tau_0$};
\draw[kernelsmod] (0,-2) -- (1.8,-0.3);
\draw (1.8,0.7) -- (0,2.5); 
\draw (-1,-1.7) node[] {\tiny$v$};
\draw (2.5,1.2) node[] {\tiny$r_2$};
\draw (0,3.1) node[] {\tiny$\tau_1$};
\draw (2.5,0) node[] {\tiny$\tau_2$};
}

\DeclareSymbol{proof4_1}{0}{
\draw (0,-2) node[] {\tiny$\tau_0$};
\draw (0,-1.3) -- (-1.5,1.8); 
\draw  (0,-1.3) -- (1.5,1.8);
\draw (-1,-1) node[] {\tiny$v$};
\draw (-1.5,2.4) node[] {\tiny$\tau_1$};
\draw (1.5,2.4) node[] {\tiny$\tau_2$};
}

\DeclareSymbol{proof4_2}{0}{
\draw (0,-2) node[] {\tiny$\tau_0$};
\draw[kernelsmod] (0,-1.3) -- (-1.5,1.8); 
\draw  (0,-1.3) -- (1.5,1.8);
\draw (-1,-1) node[] {\tiny$v$};
\draw (-1.5,2.4) node[] {\tiny$\tau_1$};
\draw (1.5,2.4) node[] {\tiny$\tau_2$};
}

\DeclareSymbol{proof4_3}{0}{
\draw (0,-2) node[] {\tiny$\tau_0$};
\draw (0,-1.3) -- (-1.5,1.8); 
\draw[kernelsmod]  (0,-1.3) -- (1.5,1.8);
\draw (-1,-1) node[] {\tiny$v$};
\draw (-1.5,2.4) node[] {\tiny$\tau_1$};
\draw (1.5,2.4) node[] {\tiny$\tau_2$};
}

\DeclareSymbol{prooftriple}{0}{
\draw (0,-2.7) node[] {\tiny$\tau_0$};
\draw (0,-2) -- (0,0.25);`
\draw (0,0.25) -- (-1.5,2.5); 
\draw (0,0.25) -- (1.5,2.5);
\draw (0,0.25) -- (0,2.5);
\draw (-1,-1.7) node[] {\tiny$v$};
\draw (-2,3.1) node[] {\tiny$\tau_1$};
\draw (0,3.1) node[] {\tiny$\tau_2$};
\draw (2,3.1) node[] {\tiny$\tau_3$};
\draw (0,0.25) node[diff] {};
}

\DeclareSymbol{prooftriple1}{0}{
\draw (0,-2.7) node[] {\tiny$\tau_0$};
\draw (0,-2) -- (0,0.25);
\draw[kernelsmod] (0,0.25) -- (-1.5,2.5); 
\draw (0,0.25) -- (1.5,2.5);
\draw[kernelsmod] (0,0.25) -- (0,2.5);
\draw (-1,-1.7) node[] {\tiny$v$};
\draw (-2,3.1) node[] {\tiny$\tau_1$};
\draw (0,3.1) node[] {\tiny$\tau_2$};
\draw (2,3.1) node[] {\tiny$\tau_3$};
\draw (0,0.25) node[not] {};
}

\DeclareSymbol{prooftripleperm1}{0}{
\draw (0,-2.7) node[] {\tiny$\tau_0$};
\draw (0,-2) -- (0,0.25);
\draw[kernelsmod] (0,0.25) -- (-3.5,2.5); 
\draw (0,0.25) -- (2.5,2.5);
\draw[kernelsmod] (0,0.25) -- (0,2.5);
\draw (-1,-1.7) node[] {\tiny$v$};
\draw (-3,3.1) node[] {\tiny$\tau_{\sigma_1}$};
\draw (0,3.1) node[] {\tiny$\tau_{\sigma_{\!2}}$};
\draw (3,3.1) node[] {\tiny$\tau_{\sigma_3}$};
\draw (0,0.25) node[not] {};
}

\DeclareSymbol{proofdouble}{0}{
\draw (0,0.5) -- (-1.5,2.5);
\draw (0,0.5) -- (1.5,2.5);
\draw (-1.5,3.1) node[] {\tiny$\tau_1$};
\draw (1.5,3.1) node[] {\tiny$\tau_2$};
\draw (0,0.5) node[diff] {};
}

\DeclareSymbol{proofquadruple}{0}{
\draw (0,0) -- (-2.5,2.5); 
\draw (0,0) -- (-1,2.5);
\draw (0,0) -- (1,2.5);
\draw (0,0) -- (2.5,2.5);
\draw (-3,3.1) node[] {\tiny$\tau_1$};
\draw (-1,3.1) node[] {\tiny$\tau_2$};
\draw (1,3.1) node[] {\tiny$\tau_3$};
\draw (3,3.1) node[] {\tiny$\tau_4$};
\draw (0,0) node[diff] {};
}

\DeclareSymbol{proofquadruple1}{0}{
\draw[kernelsmod] (0,0) -- (-2.5,2.5); 
\draw[kernelsmod] (0,0) -- (-1,2.5);
\draw (0,0) -- (1,2.5);
\draw (0,0) -- (2.5,2.5);
\draw (-3,3.1) node[] {\tiny$\tau_1$};
\draw (-1,3.1) node[] {\tiny$\tau_2$};
\draw (1,3.1) node[] {\tiny$\tau_3$};
\draw (3,3.1) node[] {\tiny$\tau_4$};
\draw (0,0) node[not] {};
}

\DeclareSymbol{proofquadrupleperm1}{0}{
\draw[kernelsmod] (0,-0.4) -- (-3.5,2.3); 
\draw[kernelsmod] (0,-0.4) -- (-1,2.3);
\draw (0,-0.4) -- (1,2.5);
\draw (0,-0.4) -- (3.5,2.5);
\draw (-4.5,3.1) node[] {\tiny$\tau_{\sigma_1}$};
\draw (-1.5,3.1) node[] {\tiny$\tau_{\sigma_2}$};
\draw (1.5,3.1) node[] {\tiny$\tau_{\sigma_3}$};
\draw (4.5,3.1) node[] {\tiny$\tau_{\sigma_4}$};
\draw (0,-0.4) node[not] {};
}

\DeclareMathAlphabet{\mathpzc}{OT1}{pzc}{m}{it}

\let\d\partial
\let\eps\varepsilon

\def\eqref#1{(\ref{#1})}

\def\loc{{\text{\rm\tiny loc}}}

\makeatletter % Stolen from the internet to make a fat \cdot which isn't as fat as a \bullet
\newcommand*{\bigcdot}{}% Check if undefined
\DeclareRobustCommand*{\bigcdot}{%
  \mathbin{\mathpalette\bigcdot@{}}%
}
\newcommand*{\bigcdot@scalefactor}{.5}
\newcommand*{\bigcdot@widthfactor}{1.15}
\newcommand*{\bigcdot@}[2]{%
  \sbox0{$#1\vcenter{}$}% math axis
  \sbox2{$#1\cdot\m@th$}%
  \hbox to \bigcdot@widthfactor\wd2{%
    \hfil
    \raise\ht0\hbox{%
      \scalebox{\bigcdot@scalefactor}{%
        \lower\ht0\hbox{$#1\bullet\m@th$}%
      }%
    }%
    \hfil
  }%
}
\makeatother

\def\sat{{\sf sat}}

\tcbset
{colframe=boxcolor,colback=symbols!7!pagebackground,coltext=pageforeground,
fonttitle=\bfseries,nobeforeafter,center title,size=fbox,boxsep=1.5pt,
top=0mm,bottom=0mm,boxsep=0mm,tcbox raise base}

\def\two{{\<generic>\kern0.05em\<genericb>}}
\def\twoI{{\<Ito>\kern0.05em\<Itob>}}

\def\mail#1{\burlalt{#1}{mailto:#1}}

\begin{document}

\title{Quasi-generalised KPZ equation}
\author{Y. Bruned$^1$, M. Gerencs\'er$^2$, U. Nadeem$^3$}
\institute{IECL (UMR 7502), Université de Lorraine\and 
TU Wien\and The Maxwell Institute, University of Edinburgh
\\
Email:\ \begin{minipage}[t]{\linewidth}
\mail{yvain.bruned@univ-lorraine.fr},
\\ \mail{mate.gerencser@tuwien.ac.at},
\\ \mail{m.u.nadeem@sms.ed.ac.uk}.
\end{minipage}}

\maketitle

\begin{abstract}
We derive the renormalised equation for the quasi-generalised KPZ equation with space-time white noise, by complementing the program initiated in \cite{MH,Mate19} for solving quasi-linear equations using regularity structures by an algebraic machinery that gives a systematic tool to remove non-local counterterms and provide a precise expression of the renormalised equation that is consistent with the semilinear case. In particular, the solution theory satisfies the chain rule and a natural notion of It\^o isometry, which can be combined to obtain global in time solution.
% for the chain rule that allows us to remove non-local counter-terms. %The Itô Isometry assumption on the renormalisation is the key for getting global in time solutions. 
\end{abstract}

\setcounter{tocdepth}{2}
\tableofcontents

\section{Introduction}

In this paper, we consider the quasi-generalised KPZ equation:
\begin{equs} \label{eq:quasi KPZ}
\partial_t u - a(u) \partial_x^{2} u = f(u) (\partial_x u)^{2} + k(u) \partial_x u + h(u) + g(u) \xi,
\end{equs}
where $ (t,x) \in \mathbb{R}_+ \times  \mathbb{T}$ with $ \mathbb{T} $ denoting one dimensional torus, and the noise $ \xi $ is the space-time white noise.
When $ a \equiv 1 $, this singular stochastic partial differential equation (SPDE) is well-understood and its resolution relies on a series of papers \cite{BHZ,CH,BCCH,BGHZ} which have laid the foundation of a general resolution of singular SPDEs. Good surveys on these developments are given in \cite{FrizHai,BaiHos}. For a geometric context of the equation when one constructs a natural stochastic process taking values in the space of loops in a compact Riemannian manifold, see \cite{proc,BGHZ}. In this context, it has been understood that with a careful choice of renormalisation constants, one has a unique counter-term that guarantees the chain rule property and the Itô isometry. 
All these main results have been obtained using the theory of Regularity Structures invented by Martin Hairer in \cite{reg}.

When $ a $ is non-constant several techniques inspired by the theory of Regularity Structures \cite{reg} and Paracontrolled Calculus \cite{Gub}  have been used for tackling this equation. 
The first result on solving a quasilinear SPDE was via a rough approach in \cite{Otto2016} but it works only for a noise still singular but much more regular than the space-time white noise. Then, we can distinguish three different categories of techniques:
\begin{itemize}
	\item Paracontrolled calculus initiated on the equation in \cite{furlan2019} and  \cite{Bailleul2019}. Then, thanks to the high-order Paracontrolled calculus developed in \cite{BB19}, one is able to reach space-time white noise without a convergence theorem for the model in \cite{BM23}.
	\item Regularity Structures via decorated trees, this is the approach proposed in \cite{MH} and it was pursued in  \cite{Mate19}. This is the approach that we use in the present paper. Let us mention also the work \cite{BHS} that treats the equation \eqref{eq:quasi KPZ} with an expansion closer to the one performed via multi-indices. 
	\item Multi-indices that are the natural continuation of the works \cite{Otto2016,OSSW18} and appeared at the first time in \cite{OSSW}. They have reached a high degree of generality covering many singular SPDEs. Indeed, their algebraic structure has been investigated in the recent works \cite{LOT,BK23,JZ,Li23,BD23} with \cite{BL23} extending them to a large class of singular SPDEs. 
	 Convergence of the multi-indices model has been obtained in \cite{LOTT} via a spectral gap assumption on the noise and an inductive procedure see also \cite{T,GT} for further extensions. One can look at \cite{LO23,OST,BOT24} for surveys on multi-indices. The  approach of \cite{LOTT} is recursive and it does not use diagrams like in \cite{CH}. Such an approach is also valid within the context of decorated trees (see \cite{BN23,BB23,HS,BH23}).
\end{itemize}

None of these approaches were able to provide a fully automated solution theory that reaches the generality obtained in the semilinear case.
From the perspective of the approach of Regularity Structures, the missing ingredient is the derivation of the renormalised equations (i.e. an analogue of \cite{BCCH}), and in particular the locality of the renormalisation counterterms.
The goal of the present work is to give a systematic algebraic understanding of these counterterms, replacing lengthy by-hand computations in \cite{MH, Mate19}. We explain below the reason for our restriction to the space-time white regularity regime, but a large part of our approach promises to be effective in the full subcritical regime.

Before moving to a detailed setup of our result, we state it in a somewhat informal way, assuming some familiarity with notations from \cite{BCCH}.
We start by introducing a class of mollifiers denoted by $ \mathrm{Moll} $ which is
the set of all compactly supported smooth functions $ \varrho : \mathbb{R}^2
\rightarrow \mathbb{R}  $ integrating to $1$, such that
$\varrho(t, -x) = \varrho(t, x)$, and such that $\varrho(t, x) = 0$ for $t \leq 0 $ (i.e. is non-anticipative). Let $ \varepsilon > 0 $, we replace $ \xi $ by its regularisation
$ \xi_{\varepsilon} = \varrho_{\varepsilon} * \xi $, the space-time convolution of the noise $ \xi $ 
with $ \varrho_{\varepsilon} $ given by:
\begin{equs}
	\varrho_{\varepsilon} = \varepsilon^{-3} \varrho(\eps^{-2}t,\eps^{-1}x)
	\end{equs}
where we have used the parabolic scaling $(2,1)$ for the rescaling.
%Then, our main theorem is the following:
Recall also from \cite{BCCH} (explained in detail in Section \ref{sec::2}) that for $c,\eps>0$, $\rho \in \mathrm{Mol}$, and \emph{semilinear} singular SPDEs of the form:
\begin{equ}
\partial_t u - c \partial_x^{2} u = \big(f(u) -a'(u)\big)(\partial_x u)^{2} + k(u) \partial_x u + h(u) + g(u) \xi,
\end{equ}
the renormalisation counterterm is given by
		\begin{equs}
			\sum_{\tau \in \SS_{\<generic>} } C_{\eps}^{c}(\tau) \frac{\Upsilon_{ F}[\tau]}{S(\tau)}(\cdot).
		\end{equs}
Here $ C_{\eps}^{c}(\tau) $ are smooth functions of the parameter $ c $, $ \SS_{\<generic>} $ is a finite set that parametrises the functions $C_{\eps}^{c}(\tau)$ and $ S(\tau) $ is a symmetry factor. The terms $ \Upsilon_{ F}[\tau] $ are coefficients or elementary differentials computed from $ F$ the nonlinearity of \eqref{eq:quasi KPZ} but with $f$ replaced by $f-a'$. With this notation at hand, we present our main result:
\begin{theorem}\label{thm:main}
	Let $a\in\mathcal{C}^6$, $f,g\in \mathcal{C}^{5}$ such that $a$ takes values in $[\lambda,\lambda^{-1}]$ for some $\lambda>0$.
	%$g'+\tfrac{(f-a')g}{a}$ is bounded.
	%\begin{equ}
	%\sup_{x}|\int_0^xf(r)\,dr|<\infty
	%\end{equ}
	Let $u_0\in\CC^\alpha(\mathbb{T})$ for some $\alpha>0$. For every $ \varrho \in \mathrm{Moll} $, $\eps>0$,	
		the renormalised equation of \eqref{eq:quasi KPZ} is given by:
		\begin{equs}[eq:renorm nonlocal]
			\partial_t u_{\eps} - a(u_{\eps}) \partial_x^{2} u_{\eps} & = f(u_{\eps}) (\partial_x u_{\eps})^{2} + k(u_{\eps}) \partial_x u_{\eps} + h(u_{\eps}) + g(u_{\eps}) \xi_{\eps} \\ & + \sum_{\tau \in {\SS}_{\<generic>} } C_{\eps}^{a(u_\eps)}(\tau) \frac{\Upsilon_{ F}[\tau](u_{\eps})}{ S(\tau)}\,.
		\end{equs}
		and the solutions $u_\eps$ of the random PDEs \eqref{eq:renorm nonlocal} converge as $\eps\to 0$ in probability, locally in time, to a nontrivial limit $u$. The functions $C^c_\eps(\tau)$ are chosen in such a way that the equations \eqref{eq:renorm nonlocal} transform according to the chain rule under composition with diffeomorphisms.
		
		If furthermore, $g,g',f,\int\!f,a'$ are bounded, then it is possible to choose the functions $C^c_\eps(\tau)$ in such a way that the solutions converge globally in time, that is, for $t\in[0,1]$.
%	If we assume that the counter-terms satisfy both the chain rule and Itô Isometry then the solution $ u_{\epsilon} $ converges in probability to a limit $ u $ globaly in time that is for $ t \in [0,1] $.
\end{theorem}

The global in time existence of the solution is obtained from relating \eqref{eq:quasi KPZ} rigorously to It\^o SPDEs, which can be seen as an instance of It\^o isometry.
The chain rule and It\^o isometry properties have been recently studied to give a unique solution in \cite{BGHZ} for the geometric KPZ equation in sufficiently high dimension for the Riemannian manifold considered. In our context, chain rule corresponds to Assumption~\ref{chain rule} and Itô isometry to Assumption~\ref{Assumption_Ito}.
Let us stress that the main challenge in proving Theorem ~\ref{thm:main} is that one does not have a mild formulation for quasi-linear equations. The strategy adopted in \cite{MH} is to instead solve the implicit system given in \eqref{eq:transformed}, which is locally-in-time equivalent to the quasilinear equation, whilst being mindful of the fact that this transformation introduces non-local functions of the solution into the system. One applies the theory of Regularity Structures to solve and renormalise this system and then one gets some counter-terms possibly non-local to pull back in the original equation. This approach was continued in \cite{Mate19} when $ g\equiv 1 $ and $ f\equiv g \equiv h \equiv 0 $, employing mainly ad-hoc computations with some integration by parts relations to show that ultimately the counter-terms are local. These integration by parts formulae correspond to the chain rule symmetry defined in \cite{BGHZ}.
This approach raised several issues/open problems:
\begin{enumerate}
	\item When $f,g,k,h$ are chosen more generally, computations by hand become untractable. To generalise the result hence, one needs a precise parametrisation of the counter-terms as given in \cite{BCCH}.
	\item Furthermore, one has to identify the correct symmetries on the renormalisation constants to get rid of the non-local counter-terms. With the work \cite{Mate19}, one can suppose that the chain rule is the correct symmetry to impose which is also very natural in the context of the generalised KPZ equation.
	\end{enumerate}

The first open problem is maybe the most challenging. A first attempt in a simpler context has been done \cite{GHM} where the noise is assumed to depend on the solution.
We use more recent techniques to circumvent these problems:
\begin{enumerate}
	\item We work with preparation maps introduced in \cite{BR18} and together with the new algebraic understanding of \cite{BM22}, we are able to present a simple proof of the renormalised equation in \cite{BB21a,BB21b}. We adapt this proof in Theorem~\ref{fisrt renormalisation}. Let us mention that this approach has been also successfully performed in \cite{BHS}.   
	\item One of the main difficult problems is to find the correct elementary differentials $ \Upsilon_{\hat{F}} $ and $ \Upsilon_{U} $ where $ \hat{F} $ and $ U $ are the two main non-linearities associated to the implicit system \eqref{eq:transformed}. We introduce new derivatives in Definition~\ref{eq:transformed} that do not commute which induces one to work with partially planar decorated trees. 
	In Theorem~\ref{prop:coherence}, we prove that these derivatives are the correct ones for describing the elementary differentials associated with the expansion of the solution. This could be seen as a new type of B-series as they are strongly connected to the approach of Regularity Structures see \cite{BR23}.
\end{enumerate}

From  Theorem~\ref{fisrt renormalisation} that depends on Theorem~\ref{prop:coherence}, one does not get exactly the counter-terms given in Theorem~\ref{thm:main} but
\begin{equs}
		\sum_{\tau \in \hat{\SS}_{\<generic>} } C_{\eps}^{c}(\tau) \frac{\Upsilon_{ \hat{F}}[\tau]}{qS(\tau)}(u_{\varepsilon})
\end{equs}
where  $ \hat{\SS}_{\<generic>} $ is a bigger combinatorial set that contains $ \SS_{\<generic>} $ and some of the terms $\Upsilon_{ \hat{F}}[\tau]$ are potentially non-local functions of the solution $u_{\varepsilon}$. Then, the chain rule symmetry given by Assumption~\ref{chain rule} allows us to get:
\begin{equs}
		\sum_{\tau \in \hat{\SS}_{\<generic>} } C_{\eps}^{a(u_{\varepsilon})}(\tau) \frac{\Upsilon_{ \hat{F}}[\tau]}{qS(\tau)}(u_{\varepsilon}) = 	\sum_{\tau \in \SS_{\<generic>} } C_{\eps}^{a(u_{\varepsilon})}(\tau) \frac{\Upsilon_{ F}[\tau]}{S(\tau)}(u_{\varepsilon})
	\end{equs}
see Theorem~\ref{main result 2}. This statement is obtained by the fact that we use covariant derivatives on trees as a basis for the elements satisfying the chain rule see Theorem ~\ref{th:covloc}. This is the only limitation of our approach that restricts us from reaching the complete subcritical regime. Indeed, considering a noise worse than space-time white noise would require one to increase the size of the combinatorial set that parametrises the counter-terms. It has been showed in \cite{BGHZ} that covariant derivatives generate linear combinations of decorated trees with four leaves that satisfy the chain rule. This open problem has been solved in \cite{BD24} for the full subcritical regime.  Theorem~\ref{prop:coherence} and Theorem~\ref{fisrt renormalisation} are not limited to the case of the space-time white noise but are true in the subcritical regime with exactly the same proof. By subcritical regime, we mean a space-time noise whose Hölder regularity in space-time is bigger than $-2$. It could be non-Gaussian. Then, one can associate a combinatorial set $  {\SS}_{\xi}$ for parametrising the renormalisation and gets the following Theorem (see \cite[Sec. 5]{BD24})   

\begin{theorem}\label{thm:main_BD}
	Let $a, f,g$ be smooth functions such that $a$ takes values in $[\lambda,\lambda^{-1}]$ for some $\lambda>0$.
	%$g'+\tfrac{(f-a')g}{a}$ is bounded.
	%\begin{equ}
	%\sup_{x}|\int_0^xf(r)\,dr|<\infty
	%\end{equ}
	Let $u_0\in\CC^\alpha(\mathbb{T})$ for some $\alpha>0$. For every $ \varrho \in \mathrm{Moll} $, $\eps>0$,	
	the renormalised equation of \eqref{eq:quasi KPZ} is given by:
	\begin{equs}[eq:renorm nonlocal_bis]
		\partial_t u_{\eps} - a(u_{\eps}) \partial_x^{2} u_{\eps} & = f(u_{\eps}) (\partial_x u_{\eps})^{2} + k(u_{\eps}) \partial_x u_{\eps} + h(u_{\eps}) + g(u_{\eps}) \xi_{\eps} \\ & + \sum_{\tau \in {\SS}_{\xi} } C_{\eps}^{a(u_\eps)}(\tau) \frac{\Upsilon_{ F}[\tau](u_{\eps})}{ S(\tau)}\,.
	\end{equs}
	and the solutions $u_\eps$ of the random PDEs \eqref{eq:renorm nonlocal} converge as $\eps\to 0$ in probability, locally in time, to a nontrivial limit $u$. The functions $C^c_\eps(\tau)$ are chosen in such a way that the equations \eqref{eq:renorm nonlocal} transform according to the chain rule under composition with diffeomorphisms.
\end{theorem}

Let us stress that the long time existence depends on the choice of the noise and we have verified it for the space-time white noise in this article. It could potentially work for other type of noise but it will be more on a case by case basis.

\vspace{0.5em}

Let us outline the paper by summarising the content of its sections. In Section~\ref{sec::2}, we start by rewriting equation \eqref{eq:quasi KPZ} in a divergence form see equation \eqref{eq:quasi KPZ new}. Then, we set up the implicit system \eqref{eq:transformed} given in \cite{MH} that is equivalent to the equation we started with. After this short subsection, we introduce the decorated trees that are used for coding the iterated integrals appearing in the expansion of the solution, wherefrom we get our definition of the set $\SS_{\<generic>}$ (see \eqref{listPage}). Moreover, we introduce a family of grafting products $ (\graft{\alpha})_{\alpha \in \mathbb{N}^2} $ (see \eqref{grafting_product}) and operators that increase node decorations $ (\uparrow^{k})_{k \in \mathbb{N}^2} $ , first introduced in \cite{BCCH}. From these products, one is able to define a product $ \star $ in \eqref{star_product} that is crucial in the sequel. The formula for the product $ \star $ was first introduced in \cite[Prop. 3.17]{BM22}. Then, we present parametrised decorated trees which boils down to introducing extra decorations on the edges that correspond to derivative in some parameters (see \eqref{paremetrised_trees} and \eqref{eq:augdecomp}); this is the main motivation for considering partial planar decorated trees. The last assumption in this subsection is natural for renormalisation constants and it will be used in the sequel: one interprets the extra decoration on the edges as parameter derivatives (see \eqref{c_derivative}). The next subsection proposes the main definition of this paper which is Definition~\ref{definition_derivative}: It introduces new derivatives in the variables $ v_{\alpha} $ that have been used for describing the implicit system \eqref{eq:transformed}.
 It is followed by Definition~\ref{def:upsilonhatF} for the elementary differentials   $ \Upsilon_F $  for parametrised decorated trees that strongly rely on the previous derivatives. One key property at the end of this subsection is Proposition~\ref{morphism star} which establishes a morphism property for the $ \Upsilon_F $ toward the product $ \star $. It is one of the ingredients for getting a short proof for the renormalised equation (see \cite{BB21a}).
 The last subsection focuses on the coherent expansion which is to show that the elementary differentials described before are the correct coefficients for an expansion of the solution. This important property is proved in Theorem~\ref{prop:coherence}.

 In Section~\ref{sec::3}, we explain how to obtain the renormalised equation with only local counter-terms thanks to the chain rule Assumption \ref{chain rule}. This proves Theorem
\ref{thm:main} except for the global-in-time existence which is the subject of the last section.
We start by introducing (strong) preparation maps (see Definitions \ref{def:PrepMap} and \eqref{def:strongprepmap}). This allows us to give a recursive definition of the renormalised model. Then, in Theorem \ref{fisrt renormalisation}, we establish the first renormalisation result for the equation \eqref{eq:quasi KPZ} but with counter-terms that are potentially non-local. The proof relies strongly on the knowledge of the coefficient $ \Upsilon $ and the morphism property given in Proposition \ref{morphism star}. It is an adaptation of the short proof of the renormalised equation for semi-linear SPDEs in \cite{BB21a}. To make the non-local counter-terms vanish one needs to introduce Assumption \ref{chain rule} corresponding to the chain rule defined in \cite{BGHZ}, which says that our counter-terms are generated by covariant derivatives given in \eqref{covariant_derivative}. These derivatives contain the basis introduced in \cite{BGHZ} for describing the chain rule property, to which we have added derivatives in the parameters. With this machinery finally, we prove Theorem~\ref{main result 2} that provides local counter-terms. Its proof relies on Theorem \ref{th:covloc} that checks the locality inductively on terms described via covariant derivatives. A warm-up example is provided in the previous subsection to illustrate the main arguments of the general proof.
 
  In Section~\ref{sec::4}, we make use of the generality (in terms of allowing general nonlinearities $f$, $g$) considered in \eqref{eq:quasi KPZ} to obtain as an interesting byproduct that, under some mild assumption on the coefficients, solutions are global in time. Indeed, the fact that our solutions obey the chain rule and a form of It\^o isometry, allows us to transform \eqref{eq:quasi KPZ} rigorously to an It\^o SPDE, whose blowup can be excluded by more classical arguments.
The global-in-time existence of singular SPDEs is in general a rather challenging (and largely still open) problem, see e.g. \cite{MW, CMW, HR, HZZ} for various approaches, to which our proof contributes an alternative argument.
   Although this proof of global well-posedness admittedly relies on the structure of the noise, the global existence of quasilinear singular SPDEs is new even in the $f=g'=0$ case of \cite{Mate19}.

\subsection*{Acknowledgements}

{\small
 Y. B. gratefully acknowledges funding support from the European Research Council (ERC) through the ERC Starting Grant Low Regularity Dynamics via Decorated Trees (LoRDeT), grant agreement No.\ 101075208.
 M. G. thankfully acknowledges support from the Austrian Science Fund
(FWF) through the START project Stochastic PDEs and Renormalisation (grant STA119).
}

\section{Expansion of the solution}

\label{sec::2}

\subsection{Main equation}
We begin by transforming \eqref{eq:quasi KPZ} into the divergence form:
\begin{equs}\label{eq:quasi KPZ new}
	\partial_t u - \partial_x (a(u) \partial_x u) = (f(u) - a'(u)) (\partial_x u)^{2} + k(u) \partial_x u + h(u) + g(u) \xi.
\end{equs}
Let $ P(c,\cdot) $ be the Green's function of the operator $ \partial_t - c \partial_x^2  $.
It is in fact easy to check that one has
\begin{equ}\label{eq:kernel scaling}
	P(c,t,x)=c^{-1/2}P(1,t,c^{-1/2}x).
\end{equ}
For a multiindex $\alpha \in \mathbb{N} \times \mathbb{N}^2$, we define the operator $I_\alpha$ as
\begin{equ}\label{eq:Inta}
	I_\alpha(b,f)(z)=(\d_\alpha P(c,\cdot)\ast f)(z)|_{c=b(z)}.
\end{equ}
where for $ \alpha = (\hat{\alpha}, \bar \alpha) $ with $ \bar \alpha = (\bar \alpha_1, \bar \alpha_2) $, one has 
\begin{equs}
	\d^\alpha = \partial_c^{\hat{\alpha}} \partial^{\bar \alpha}
	\quad
\partial^{\bar \alpha} = 	\partial_t^{\bar{\alpha}_1} 	\partial_x^{\bar{\alpha}_2}.
\end{equs}
In the sequel, we will use the short hand notations $ c $ for $ (1,0,0) $ and $ x $ for $ (0,0,1) $.
For smooth input $\xi$, the equation \eqref{eq:quasi KPZ new} is locally in time equivalent to the following system:
\begin{equs} 
	u &  = I(a(u), \hat F), \\
	\hat F & = (1 - a'(u) I_{c}(a(u), \hat F))  \, F + (a (a')^2)(u) I_{cc}(a(u),\hat F) (\partial_x u)^2 \\
	&\quad + (a a'')(u) I_{c}(a(u),\hat F) (\partial_x u)^2    + 2 ( a a' )(u)( \partial_x u ) I_{cx}(a(u),\hat F)
	\\
	& \quad + a'(u) (\partial_x u ) I_{x}(a(u),\hat F),\label{eq:transformed}
\end{equs}
where $ F = (f(u) - a'(u)) (\partial_x u)^{2} + g(u) \xi $. 

\begin{remark}
	This formulation is different from the original one given in \cite{MH,Mate19} where the authors do not consider the equation in a divergence form. One of the reasons for this formulation is to underline the fact that part of the renormalisation acts on $ a(u) \partial_x^{2} $.
	%That's why we extract the divergent part and we put it on the right hand side of the equation.
	As we will show, rewriting the second-order operator in divergence form puts all the divergent terms on the right-hand side of the equation.
\end{remark}

By setting $ v_\alpha \coloneqq  I_{\alpha}(a(u), \hat F)) $ for multiindices $\alpha$, and $ q(u) \coloneqq 1-a'(u) v_c $ and we get
\begin{equs}
	\hat F & = q  \, F + (a (a')^2)(u) v_{cc} (\partial_x u)^2 \\ &+ (a a'')(u) v_{c} (\partial_x u)^2    + 2 ( a a' )(u)( \partial_x u ) v_{cx} + a'(u) (\partial_x u ) v_{x}.
\end{equs}
Rewriting the last equation, we get
\begin{equs}\label{eq:nonlinearity}
	\hat F = \hat f(u) (\partial_x u)^{2} + \hat g(u) \xi + 2 ( a a' )(u)( \partial_x u ) v_{cx} + a'(u) (\partial_x u ) v_{x},
\end{equs}
where 
\begin{equs}
	\hat f  = q (f - a') + a (a')^{2} v_{cc} + a a'' v_c, \quad \hat g = q g. 
\end{equs}
Motivated by the fact that the multi-index $\alpha$ refers to the number of derivatives taken in $c$ and $x$ of \eqref{eq:Inta}, we will at times eschew the multi-indices altogether for clarity and indicate the derivative being taken instead. So, for instance, $v = v_{(0,0)}$, $v_x = v_{(0,1)}$, $v_{cc} = v_{(2,0)}$ etc.

\subsection{Decorated trees and main products}
The application of regularity structures to the study of SPDEs is facilitated by decorated trees, an exposition whereof we present in this section. We begin by fixing first a finite set of types which will be the basic component for building the aforementioned trees - $\Lab \coloneqq \{\CI,\Xi_\xi,\Xi_1\}$. The first of these abstract symbols is meant to represent convolution with the kernel $P$, the second corresponds to the space-time noise $\xi$ and finally, the last one stands in for $1$. We then define a set of edge decorations $\CD \coloneqq \CL\times\mathbb{N}^2$, with which we are able to give the following definition:
\begin{definition}
A \textbf{decorated tree} over $\mathcal{D}$ is a $3$-tuple of the form  $\tau_{\Labe}^{\Labn} =  (\tau,\Labn,\Labe)$ where $\tau$ is a non-planar rooted tree with node set $N_\tau$ and edge set $E_\tau$. The maps $\Labn : N_{\tau} \rightarrow \mathbb{N}^{2}$ and $\Labe : E_\tau{\tiny } \rightarrow \mathcal{D}$ are node and edge decorations, respectively.
\end{definition}
\begin{remark}
To not clutter notation, we will often suppress the decorations and just write $\tau$ instead of $\tau^{\mfn}_{\mfe}$ when they are understood or of no consequence to the argument being made.
\end{remark}
We use $T$ to denote the set of decorated trees, and define a binary tree product $\cdot:T\times T\mapsto T$ by 
\begin{equation}  \label{treeproduct}
 	(\tau,\Labn,\Labe) \cdot  (\tau',\Labn',\Labe') 
 	= (\tau  \tau',\Labn + \Labn', \Labe + \Labe')\;, 
\end{equation} 
where $\tau  \tau'$ is the rooted tree obtained by identifying the roots of $ \tau$ and $\tau'$. The sums $ \Labn + \Labn'$ mean that decorations are added at the root and extended to the disjoint union by setting them to vanish on the other tree. Edges and vertices of either tree, keep their decoration, except at the roots which merge into a new root decorated by the sum of the previous two decorations. While this definition allows for greater generality, in our analysis we will only see edges such that they are:

\begin{enumerate}
   
  \item[(i)] decorated by $ (\Xi_i,0) \in \mathcal{D} $ and consequently denoted by $  \Xi_i $, for $i\in\{1,\xi\}$, or;
  \item[(ii)] decorated by  $ (\mcI,\alpha) \in \mathcal{D} $  and consequently denoted by $ \mcI_{\alpha} $.
\end{enumerate} 
We will also use $\CI_\alpha(\tau)$ to represent the tree that is constructed by taking $\tau\in T$ and grafting onto its root, an edge that carries the decoration $\CI_\alpha$. The root of $\CI_\alpha(\tau)$ will carry the node decoration of $0$. Such trees will be referred to as planted trees and denoted by $\CI(T)$.
Additionally, we will have as nodes, a family of symbols $\{\bullet^k\}_{k\in\mbbN^2}$ that encode a factor $ \mbX^k$  such that $\Labn(\bullet^k)=k$. We will write $ X_i$, $ i \in \lbrace 0,1\rbrace $, to denote $ \mbX^{e_i}$, where we have denoted by $ e_0,\,e_1 $ the canonical basis of $ \mathbb{N}^{2} $, so that one has $\mbX^k = \mbX^{2k_0}_0\mbX^{k_1}_1$, for $k=(k_0,k_1)\in\mbbN^2$. Of particular significance will be the monomial $ \mbX^0(\coloneqq\mbX^{(0,0)}) $ that will be identified with the empty tree $\mathbf{1}$. Finally, we will denote by $\bar T$ the space of all such monomials, and $\bar \CT$ the linear span of $\bar T$.

To these symbols we will add a notion of degree by postulating for a function $\deg:T\rightarrow \mbbR$ that
\begin{alignat*}{2}
&\deg(\Xi_\xi)=\alpha-\kappa,\qquad&&\deg(\Xi_1)=0\\
&\deg(\mbX_0) = 2,&&\deg(\mbX_1)=1,
\end{alignat*}
where $ \kappa > 0 $. The $\alpha$ here correspondence to regularity of the noise under question. If one wants to be in the subcritical regime, it is necessary that
\begin{equs}
	\alpha - \kappa > -2.
\end{equs}
Otherwise, one can produce an infinite number of decorated trees with negative degree.
 For the space-time white noise one has that $\alpha=-\tfrac{3}{2}$ and the degrees of the monomial are due to the parabolic scaling we impose. We then extend it recursively by
\begin{equation*}
\deg\left(\prod_{i=1}^m\tau_i\right)=\sum_{i=1}^m\deg(\tau_i),\qquad \deg\left(\CI_a(\tau)\right) = \deg(\tau)+2-|a|_\s,
\end{equation*}
where the product $``\prod_{i=1}^m"$ is the $m$-fold tree product, and the $``2"$ above is due to the Schauder estimate for $P$. A feature of this construction is that all $\tau\in T$, admit a unique, recursive decomposition of the form:
\begin{equation}\label{eq:treedecomp}
\tau = \mathbf{X}^{k} \Big( \prod_{i=1}^m \mathcal{I}_{\alpha_i}[\tau_i]\Big)\Xi_\mathfrak{l},
\end{equation}
with $\mathfrak{l}\in\{1,\xi\}$. Of particular importance in this work will be the following space of ``positive'' trees
\begin{equs}
T^{+}=\{\mbX^{k}\prod_{j}\CI^{+}_{a_j}(\tau_j) :\deg(\CI^{+}_{a_j}(\tau_j))>0,\,\tau_j\in T,k\in\mbbN^{2}\}
\end{equs}
where the symbol $\CI^{+}$ is meant to stress the difference between $T$ and $T^{+}$. $T_-$ is correspondingly the space of all negative degree trees.

Due to this recursive nature of trees, a representation as in  \eqref{eq:treedecomp} fast becomes cumbersome to handle. It is helpful therefore to institute a graphical program that is better suited to represent such trees. The standard approach in the literature is to use $ \<generic> $ for the space-white noise symbol $\Xi_\xi$ and $ \<thin> $ for $\CI$ while $\Xi_1$, does not warrant special depiction. In this graphical setting, we understand $\CI_\alpha(\tau)$ in the same way as before, except now the edge $\CI(\coloneqq\CI_{(0.0)})$ will be depicted by \<thin> (the same as the type $\CI$). For demonstration consider the tree $\tau = \CI(\Xi_\xi)$, which under our scheme will be depicted as \<PlantedNoise>. One might think that this will make the notation unwieldy because $a\in\mbbN^2$, but it turns out that the only other multi-index apart from $(0,0)$ that $a$ can be - at least so far as \eqref{eq:quasi KPZ} is concerned - is $(0,1)$. Due to this, we append to our scheme the symbol $\<thick> = (\<thin>, (0,1))\in\CD$ and use it in the same way as $\<thin>$. Finally, products of two trees are depicted, as explained before, by joining them at the root, so for example $\tau = \CI(\Xi_\xi)^2$ is represented by \begin{tikzpicture}[scale=0.8]
\draw[kernels2] (-0.25,0.25) node[xi] {} to (0,0); 
\draw[kernels2] (0.25,0.25) node[xi] {} to (0,0);
\end{tikzpicture}. Moving forward we will indulge in a slight abuse of notation by using $\CI(\Xi)$ instead of $\CI(\Xi_\xi)$, but this is harmless as we will have $\CI(\Xi_1)=0$.

Our construction thus far has relied on the fact that we will only encounter a very specific set of trees. This idea is made rigorous with the notion of a rule which is used to construct the regularity structure wherein we do our analysis. For a complete exposition, we refer to \cite{BHZ}, but remark that these rules are extracted very naturally from the SPDE under question. For \eqref{eq:quasi KPZ}, the rule map $ R:\Lab \mapsto \CD$ is defined by setting $R(\<generic>) = \{()\}$ and
\begin{equ}
R(\<thin>) = \{(\<thin>^k,\<generic>), (\<thick>^\ell,\<thin>^k)\,:\, k \ge 0,\, \ell \in \{0,1,2\} \}\;,
\end{equ}
where we used $\<thin>^k$ to denote $k$ repetitions of $\<thin>$ and similarly for $\<thick>^\ell$.
Moving forward the only trees we are interested in, are those that conform to $ R $ by which we mean that at each node $ e $, the edges adjacent to $e$ correspond to the tuple associated to the edge $ e $ by the rule $ R $. The linear span of these $R$-conformal trees will be denoted by $ \CT $. 
In deriving the renormalised equation it turns out that one can disregard some trees and focus only on the saturated ones, see \cite[Sec. 2.2]{BGHZ} where some trees will give a zero cancellation. Indeed, if one has  one thick edge at a node (odd number of thick edges over the decorated trees), 
the renormalisation constants are zero due to antisymmetry (odd number
of derivatives $ \partial_x $), see \cite[Lem. 2.5]{BGHZ}. For the same reason decorated trees containing a polynomial decoration are disregarded (see \cite[Sec. 2.1]{BGHZ}). If one has two nodes with one thick edge each, then the decorated trees is of positive degree when the regularity of the noise is in the full subcritical regime. Therefore, we consider trees generated by the ``saturated rule'' $ R^\sat $ given by
\begin{equ}
R^\sat(\<thin>) = \{(\<thin>^k,\<generic>_i), (\<thick>^2,\<thin>^k)\,:\, k \ge 0 \}\;,
\end{equ}
This rule guarantees that all thick edges come in pairs.
Given a noise $\xi$ such that its space-time regularity is strictly greater then $-2$, we denote by $ \SS_{\xi} $ the decorated trees with negative degree generated by the saturated rule. For the space-time white noise, we denote this space by $ \SS_{\<generic>} $.

We write $\CS_{\<generic>}$ \label{CS page ref} for the real vector space generated by $\SS_{\<generic>} = \SS_{\<generic>}^{(2)} \sqcup  \SS_{\<generic>}^{(4)}$ where $ \SS_{\<generic>}^{(k)} $ are $  R^\sat $ conformal trees with a vanishing vertex decoration $\mfn$ and with $ k $ noises.
\begin{equs} \label{listPage}
\begin{tabular}{rll} \toprule
& Elements of $ \SS_{\<generic>}  $\\
\midrule
 $ \SS_{\<generic>}^{(2)} $   &  \<Xi2> \,, \<I1Xitwo> 
\\ 
$ \SS_{\<generic>}^{(4)} $ &  \<Xi4>\,, \<I1Xi4a> \,,  \<I1Xi4b> \,, \<I1Xi4c> \,, \<I1Xi4ab> \,, \<I1Xi4bc> \,, \<I1Xi4ac> \,,  \<Xitwo> \,, \<2I1Xi4> \,, \<2I1Xi4b> \,,  \<2I1Xi4c> \,, \\  & \<Xi4b>\,, \<Xi4ba>\,, \<Xi4c>\,,   \<Xi4cb>\,, \<Xi4ca>\,, \<Xi4cab>\,, \<Xi4e> \,, \<Xi4ea> \,, \<Xi4eabis> \,,  \<Xi4eb> \,, \<Xi4eab> \,, \<Xi4eabbis> 
\\
\bottomrule
\end{tabular}
\end{equs}

We also associate to a decorated tree $\tau$, a symmetry factor $ S(\tau) $ that is ostensibly a count of its tree isomorphisms. An explicit recursive definition of $S(\tau)$ is possible by grouping the same planted trees $\CI_\alpha(\tau_i)$ in \eqref{eq:treedecomp} in such a manner so that $ \tau = \mathbf{X}^{k} \Big( \prod_{j=1}^m \mathcal{I}_{\alpha_j}[\tau_j]^{\beta_j}\Big) \Xi_{\mfl}\;, $ and $ \CI_{\alpha_i}(\tau_i) \neq \CI_{\alpha_j}(\tau_j)$ for $i \neq j$. With this representation of $\tau$ we can set
\begin{align}
S(\Xi_{\mfl}) = 1, \quad S(\tau)
=
k!
\Big(
\prod_{j=1}^{m}
S(\tau_{j})^{\beta_{j}}
\beta_{j}!
\Big)\;.
\end{align}
This furthermore defines an inner product on the space of decorated trees $T$, by:
\begin{equs}\label{eq:innerproduct}
\langle\sigma,\tau\rangle\coloneqq S(\tau)\delta(\sigma,\tau).
\end{equs}\vspace{-1cm}
\begin{example}
For monomials, one has
\begin{equs}\label{eq:monprod}
\inprod{\mbX^{k}}{\mbX^{\bar k}}=\delta(k,\bar k)k!,
\end{equs}
which implies the following for trees of the form $\Xi_{\mfl}\mbX^{k}$
$$\inprod{\Xi_{\mfl}\mbX^{k}}{\Xi_{\bar\mfl}\mbX^{\bar k}}=\delta(\mfl,\bar\mfl)\,\delta(k,\bar k)\,k!.$$
\end{example}
\begin{example}\label{ex:plantedtree} For a planted tree, $\tau=\CI_\alpha(\bar\tau)$, one notices that $S(\tau) = S(\bar\tau)$, and hence
\begin{equs}
\inprod{\CI_{\alpha_1}(\tau_1)}{\CI_{\alpha_2}(\tau_2)}=S(\tau_1)\delta(\tau_1,\tau_2)\delta(\alpha_1,\alpha_2)=\inprod{\tau_1}{\tau_2}\delta(\alpha_1,\alpha_2),
\end{equs}
so that one has in particular
\begin{equs}\label{eq:plantedtree}
\inprod{\CI_\alpha(\tau_1)}{\CI_\alpha(\tau_2)} = \inprod{\tau_1}{\tau_2}
\end{equs}
\end{example}
\begin{example}
For $U = \sum_{\tau}u_{\tau}\tau \in \CT$ and $\sigma\in\ T$, one has:
\begin{equs}
\inprod{U}{\sigma} = \inprod{\sum_{\tau}u_{\tau}\tau}{\sigma}=\begin{cases}
u_{\sigma}\inprod{\sigma}{\sigma} & \text{if }\sigma\text{ is amongst the }\tau \\
0 & \text{otherwise}
\end{cases}
\end{equs}
which implies that the coefficient $u_{\sigma}$ can be recovered via our inner product as follows:
\begin{equs}
 u_{\sigma}=\frac{\inprod{U}{\sigma}}{\inprod{\sigma}{\sigma}},
\end{equs}
which in turn means we are able to represent any $U\in\CT$ as
\begin{equs}
U = \sum_\tau \frac{\inprod{U}{\tau}}{\inprod{\tau}{\tau}}\tau.
\end{equs}
\end{example}
In addition to the functional derivatives we introduced before, we introduce a couple of abstract derivatives $\{\CD_i\}_{i\in\{0,1\}}$ on $\CT$, by requiring
\begin{equs}
\CD_i\mbX_j=\delta(i,j)\mathbf{1},\quad \CD_i\mathbf{1} = 0,\quad \CD_{i}\CI_\alpha(\tau) = \CI_{\alpha+e_i}(\tau),
\end{equs}
and extending to all of $\CT$ via Leibniz rule. We also set for $p=(p_0,p_1)\in\mathbb{N}^{2}$, $\CD^{p}=\CD_0^{p_0}\CD_1^{p_1}$, which is well defined because the derivatives commute.

Now following the authors' lead in \cite{BB21a} we introduce a number of important bilinear operators on $T$. The first one is a derivation denoted by $\uparrow^{k}_{n}$ that adds $k\in\mbbN^2$ to the node decoration at $v\in\N_\tau$. The second one, is the grafting operator $\graft{\alpha}$ defined for $\sigma,\,\tau\in T$ and indexed by $\alpha\in\mbbN^2$ by:
\begin{equs} \label{grafting_product}
\sigma\graft{\alpha}\tau\coloneqq\sum_{v\in N_\tau}\sum_{\beta\in\mbbN^2}\binom{\mbn_v}{\beta}\sigma\graft{\alpha-\beta}^v(\uparrow^{-m}_{v}\tau)
\end{equs}
where $\mbn_v$ is the decoration at the node $v$ and $\sigma\graft{\alpha - \beta}^{v}\tau$ the decorated tree constructed by attaching $\sigma$ to $\tau$ at $v$ with the edge $\CI_{\alpha -\beta}$.  
For planted trees $\CI_\alpha(\sigma)$, our grafting operator leads to a natural product $\graft{}:\CI(T)\times T\mapsto T$ given by:
\begin{equs}
\CI_\alpha(\sigma)\graft{}\tau\coloneqq\sigma\graft{\alpha}\tau,
\end{equs}
which is then extended to products of the form $\prod_{i}^{m}\CI_{\alpha_i}(\sigma_i)$ by grafting each tree $\sigma_i$ on $\tau$ according to $a_i$ independently of the others. We also enlarge the scope of the previously defined derivation to $B\subseteq N_\tau$, by setting:
\begin{equs}
\uparrow^{k}_B\tau=\sum_{\sum_{v\in B}k_v=k}\prod_{v\in B}\tau
\end{equs}
and set $\uparrow^{k}\tau\coloneqq\uparrow^{k}_{N_\tau}$.
The last of our operators is the $\star$ product which we define for all $\sigma = \mbX^{k}\prod_{i}\CI_{a_i}(\sigma_i),\,\tau\in T$, by:
\begin{equs} \label{star_product}
\sigma\star\tau\coloneqq\uparrow^{k}\left(\prod_{i}\CI_{a_i}(\sigma_i)\graft{}\tau\right).
\end{equs}

With all these definitions, we are able to describe the renormalisation occurring when $ a $ is not constant.
For non-constant $a$ taking values in $[\lambda,\lambda^{-1}]$ for some $\lambda>0$,
we use an augmented regularity structure as in \cite{MH}, which in our situation is set up as follows.
Fix a parameter value $c_0\in[\lambda,\lambda^{-1}]$. We consider $ T_p $ the planar version of $ T $ and we set on $T_p$ the function $[\cdot]$ denoting the number of integration edges. Finally, we define the set
\begin{equs} \label{paremetrised_trees}
\hat T=\{\d^{i_1}\otimes\cdots\otimes\d^{i_{[\tau]}}\otimes\tau:\,\tau\in T_p ;\,i_1,\ldots,i_{[\tau]}\in\mathbb{N}\}.
\end{equs}
The vector space generated by $\hat T$ is denoted by $\hat\mcT$. each
Here the derivatives are purely formal and should be understood to act on the parameter.
More precisely, if $\hat\alpha\in\mathbb{N}$ and $\bar \alpha\in \mathbb{N}^2$, then for $\alpha=(\hat\alpha,\bar\alpha)$ we have the abstract integration operator $\CI_{\alpha}$ acting on $\hat T$ by
\begin{equ}\label{eq:CI parameter}
\CI_{\alpha}\big(\d^{i_1}\otimes\cdots\otimes\d^{i_{[\tau]}}\otimes\tau\big)=\d^{\hat\alpha}\otimes\d^{i_1}\otimes\cdots\otimes\d^{i_{[\tau]}}\otimes\CI_{\bar\alpha}\tau,
\end{equ}
which is further extended linearly to all of $\bar\CT$.
\begin{remark}
We concede here the possibility of confusion betwixt this operation $\CI_\alpha$ for $\alpha\in\mbbN\times \mathbb{N}^2$ on parametrised trees - i.e. elements of $\hat T$ - and $\CI_{\alpha}$ for $\alpha\in\mbbN^{2}$ on unparametrised trees - elements of $T$. In defining the former we also apply $\CI_{\bar{\alpha}}$ on unparametrised trees which is also incongruous to our construction before. Nonetheless, we persist with this abuse of notation, with the hope that the readers are able to distinguish between these operators from context.
\end{remark}
Similarly, we extend the abstract derivative operator to parametrised trees $\tau\in\hat{T}$ by setting
\begin{equs}
\CD\tau = \partial^{i_1}\otimes\cdots\otimes\partial^{i_{[\tau]}}\otimes\CD\tau
\end{equs}
which is furthermore extended to $\hat{\CT}$ by linearity.
From \eqref{eq:CI parameter} and the product rule
\begin{equs}
\big(\d^{i_1}\otimes\cdots\otimes\d^{i_{[\tau]}}\otimes\tau\big)
\big(&\d^{\bar i_1}\otimes\cdots\otimes\d^{\bar i_{[\hat\tau]}}\otimes\bar\tau\big)
= \\
&\d^{i_1}\otimes\cdots\otimes\d^{i_{[\tau]}}\otimes \d^{\bar i_1}\otimes\cdots\otimes\d^{\bar i_{[\bar\tau]}}\otimes\tau\bar\tau,
\end{equs}
where now $ \tau\bar\tau $ is the planar tree product that respects the order of the edges at the root. 
All elements of $\hat T$ can be built from repeated use of integrations and multiplications. It should not be difficult to see that much as for $\tau\in T$, one has the following decomposition for $\hat\tau\in\hat{T}$
\begin{equs}\label{eq:augdecomp}
\hat\tau = \mathbf{X}^{k} \Big( \prod_{i=1}^m \mathcal{I}_{\alpha_i}[\hat\tau_i]\Big)\Xi_\mathfrak{l},
\end{equs}
where $\hat\tau_i\in\bar{T}$, $\alpha_i=(\hat{\alpha}_i,\bar{\alpha_i})\in\mbbN\times\mathbb{N}^2$.  
In the sequel, we will use the notation given by \eqref{eq:augdecomp} and we will not work with complete planar trees. Indeed, we allow edges with no parameter decoration to commute between themselves.
This means that if all of the first decorations - $ \hat{\alpha} $ - are zero, the decorated tree is non-planar and corresponds to the definition set in the semi-linear setting.  If two edges carry the same decoration, then they commute. This is also a natural property as these edges are interpreted as derivatives that commute. With this quotient, one has
\begin{equs}
	\CI_{(0,k)}(\tau) \CI_{(0,\ell)}(\sigma) =
	\CI_{(0,\ell)}(\sigma) \CI_{(0,k)}(\tau),
\end{equs}
but one has
\begin{equs}
	\CI_{(0,k)}(\tau) \CI_{(1,\ell)}(\sigma) \neq
	\CI_{(1,\ell)}(\sigma) \CI_{(0,k)}(\tau). 
	\end{equs}
On this new set of decorated trees, we extend naturally the definition of the grafting product  $ \graft{\alpha} $ by grafting on the right most location on each node. One then obtains an extension for the product $ \star $. The definition of $ \uparrow^{k} $ is not modified. 
For the symmetry factor, we also consider a similar definition by looking at the decomposition of the form $ \tau = \mathbf{X}^{k} \Big( \prod_{j=1}^m \mathcal{I}_{\alpha_j}[\tau_j]^{\beta_j}\Big) \Xi_{\mfl}\;, $ where now there is an order in the product $ \prod_i $. Then, one has:
\begin{equs}
	\label{not_equal}
	 \CI_{\alpha_i}(\tau_i) \neq \CI_{\alpha_j}(\tau_j), \quad \text{for } i \neq j. 
	 \end{equs}
	or if   $ \CI_{\alpha_i}(\tau_i)   = \CI_{\alpha_j}(\tau_j) $, then they are separated by an edge for which they do not commute which corresponds to the following situation:
\begin{equs} \label{edges_not_commuting}
	\begin{aligned}
	\tau = \mathbf{X}^{k} \Big( \prod_{\ell=1}^i \mathcal{I}_{\alpha_\ell}[\tau_\ell]^{\beta_\ell} \Big) \sigma_1 \mathcal{I}_{\alpha_r}[\tau_r]^{\beta_r} \sigma_2 \mathcal{I}_{\alpha_j}[\tau_j]^{\beta_j} \sigma_3 \Xi_{\mfl}, 
	\\ \mathcal{I}_{\alpha_r}[\tau_r]\mathcal{I}_{\alpha_j}[\tau_j] \neq \mathcal{I}_{\alpha_j}[\tau_j]\mathcal{I}_{\alpha_r}[\tau_r],
	\end{aligned}
\end{equs}
where the $ \sigma_i $ could be potentially empty and we have assumed that $i \leq j$ without loss of generality.
 With this representation of $\tau$ we can set
\begin{align}
	S(\Xi_{\mfl}) = 1, \quad S(\tau)
	=
	k!
	\Big(
	\prod_{j=1}^{m}
	S(\tau_{j})^{\beta_{j}}
	\beta_{j}!
	\Big)\;.
\end{align}
One recovers the classical definition of the combinatorial factor in the non-planar setting.
In the sequel, when we write $\prod_{j=1}^m \mathcal{I}_{\alpha_j}[\tau_j]^{\beta_j}$ we suppose the previous two conditions satisfied that are \eqref{not_equal} and  \eqref{edges_not_commuting}.

\begin{remark}Note that we work with a single (but arbitrary) value of $c_0$, which is slightly simpler than in \cite{MH}: the reason is that the abstract solution of \eqref{eq:transformed} is a modelled distribution, which at a given point $z$ takes values in $\hat\mcT$ with the choice $c_0=a(u(z))$.
Since we will calculate the renormalisation pointwise, this smaller structure will suffice for our purposes.
\end{remark}

We employ a similar graphical programme for parametrised trees as we did for the unparametrised trees, except that derivatives in the parameter are depicted alongside the edge. So for example one will have:
\begin{equs}
\Xi_\xi\CI_{(k,0)}(\Xi_\xi) = \begin{tikzpicture}[baseline = 0cm, scale=0.4]
\node[xi] at (0,0) (a) {}; 
\node[xi] at (1,1) (b) {};
\node at (0.2,0.9) {\tiny $k$};
\draw[color=blue] (a) to (b);
\end{tikzpicture},\qquad \CI_{(m,1)}(\Xi_\xi)\CI_{(l,1)}(\Xi_\xi)= \begin{tikzpicture}[baseline=0cm,scale=0.4]
\node[xi] at (0.9,1) (a) {};
\node[xi] at (-1,1) (b) {};
\draw[kernels2] (a) to (-0.1,0);
\draw[kernels2] (b) to (0,0);
\node at (0.5,0) {\tiny $l$};
\node at (-0.6,-0.1) {\tiny $m$};
\end{tikzpicture},
\end{equs}
with the caveat that we will eschew writing the parameters whenever they are all zero. This shorthand works because we can signal the second component by the thickness of the edge. When we wish to make a more general claim we will put aside the edges, the entire multi-index. To avoid confusion we will always use Greek alphabets for this purpose.
We denote by $\hat{\SS}_{\<generic>}$ the 
extension of ${\SS}_{\<generic>}$ with $ c $ derivatives in the sense, that one gets from $ \tau \in \hat{\SS}_{\<generic>}$ all the terms 
$ \left(\otimes_{e\in E_\tau}\partial^{\beta(e)}\right) \otimes \tau $ where $ \beta(e) $ is bounded by a certain $ N $. The degree for such terms is the same as the degree for $ \tau $ without $c$ derivatives.- Morever, we extend naturally the set  $T^{+}$, and $T_-$ by including the elements with the parameter derivatives. 

The renormalisation of the augmented regularity structure first assigns to each symbol $\tau\in\mfT$ a \emph{function} $C_\eps(\tau)(\cdot)$ in $[\tau]$ variables (as opposed to just a constant in the semilinear case).
The renormalisation constant assigned to $\hat\tau=\d^{i_1}\otimes\cdots\otimes\d^{i_{[\tau]}}\otimes\tau$ is then given by
\begin{equs} \label{c_derivative}
\hat C_\eps(\hat\tau)=\d^{i_1}_{c_1}\cdots\d^{i_{[\tau]}}_{c_{[\tau]}}
C_\eps(\tau)(c_1,\ldots,c_{[\tau]})|_{c_1=\cdots c_{[\tau]}=c_0}.
\end{equs}
Let us also introduce separate notation for the diagonal of the renormalisation function
\begin{equ}
C_\eps^c(\tau)=C_\eps(\tau)(c,\ldots,c).
\end{equ}

%But when $ a \neq 1 $, one need to use planar trees. \yvainText{Explain the structure of these trees.}

\subsection{Elementary differentials}
In the expansion of the solution of \eqref{eq:transformed}, one needs to define coefficients that are an extension of elementary differentials found in B-series. For that, we have to introduce new derivatives which are the cornerstone of our construction.
 \begin{definition} \label{definition_derivative}
 	One sets for every $ \alpha, \beta \in \mathbb{N} \times \mathbb{N}^2 $ and $ m \in \mathbb{N}^2 $:
\begin{equs}[eq:derivatives] 
	&\partial_x u = \frac{v_x}{q}, \qquad \partial_t u = \frac{v_t}{q} \qquad \partial_{v_{\alpha}}\partial^m u  = \delta(\alpha,(0,m))  \frac{1}{q}, \\
	&\d_{v_{\alpha}} v_{\beta}   = \delta(\alpha,(0,0)) \,  \frac{a'(u)v_{\beta + (1,0)}}{q}  + \delta(\alpha,\beta),
\end{equs}
where $\delta$ is the Kronecker symbol between multi-indices and $ \partial^m  = \partial_t^{m_1} \partial_x^{m_2} $ with $ m = (m_1,m_2) $.  We extend these derivatives via the chain rule.
\end{definition}

\begin{remark}
	These derivatives are not commuting in general. One checks for example that:
	\begin{equs}
		\partial_{v_c} \partial_{v} u  = \partial_{v_c} \left(\frac{1}{q} \right) = - \frac{\partial_{v_c} q}{q^{2}} =  \frac{a'(u)} {q^{2}}, \quad \partial_{v} \partial_{v_c} u = 0.
	\end{equs}
\end{remark}
The family of derivative $\{\partial^{m}\}_{m\in \mathbb{N} \times \mbbN^{2}}$ satisfies the chain rule:
\begin{equs}
	\partial^{k}F\coloneqq\sum_{\alpha\in \mathbb{N} \times \mathbb{N}^{2}}\partial^kv_\alpha \partial_{v_\alpha}F.
\end{equs}

\begin{definition}\label{def:upsilonhatF}
Given a non-linearity that admits the form $F =  F_{1} 1 +   F_{\xi} \xi  $ where $F$ depends on $ u, \partial_x u, v_c, v_{cc}, v_{xc} $, we define elementary differential operators $\Upsilon_{ F}(\tau)(u)$  such that 
\begin{equation} \label{eq:upsilonhatF}
\Upsilon_{ F} \left[\mathbf{X}^{k} \Big( \prod_{i=1}^m \mathcal{I}_{\alpha_i}[\tau_i]\Big) \Xi_{\mfl}  \right] = \partial^{k}\partial_{v_{\alpha_1}}\hdots\partial_{v_{\alpha_m}} F_{\mfl}\prod_{i=1}^{m}  \Upsilon_{ F}[\tau_i],
\end{equation}
where $ \mfl \in \lbrace 1,\xi  \rbrace $.
\end{definition}
\begin{example}
Accepting the convention that empty products are just one, we get that $\Upsilon_F[\Xi_\mfl]=F_\mfl$.
\end{example}
\begin{remark}
One may contrast this with the treatment of standard generalised KPZ equation in \cite{BCCH}, where one defines the elementary differential operators by
\begin{equs}\label{eq:upsilonF}
\Upsilon_F \left(\mathbf{X}^{k} \Big( \prod_{i=1}^m \mathcal{I}_{\alpha_i}[\tau_i]\Big) \Xi_{\mfl}  \right) = \prod_{i=1}^{m}  \Upsilon_{F}(\tau_i)  \partial^{k} \prod_{i=1}^m  \partial_{\partial_{\alpha_i}u}  F_{\mfl}.
\end{equs}
Although similarly defined, one needs to be careful that in Definition~\ref{def:upsilonhatF}, one is in the non-commutative setting which means that the order in the products $ \prod_{i=1}^m \mathcal{I}_{\alpha_i}[\tau_i]$ and  $ \prod_{i=1}^m  \partial_{v_{\alpha_i}} $
matters.
\end{remark}
We may now present our version of the morphism property in \cite[Cor. 4.15]{BCCH},
\begin{equs}\label{eq:morphism}
\Upsilon_{ F} \left( \tau_1 \graft{\alpha} \tau_2 \right) = \Upsilon_{ F} \left( \tau_1 \right) \partial_{v_{\alpha}} \Upsilon_{F} \left( \tau_2 \right),
\end{equs}
which is a consequence of the following proposition: 
\begin{proposition} \label{morphism star}
For $\tau=\mbX^{k}\prod_{i=1}^{n}\CI_{\alpha_i}(\tau_i)\in T$, $\sigma\in T$, one has
\begin{equs}
\Upsilon_{ F}\left(\left[\mbX^{k}\prod_{i=1}^{n}\CI_{a_i}(\tau_i)\right]\star\sigma\right)=\partial^{k}\partial_{v_{\alpha_1}}\cdots\,\partial_{v_{\alpha_n}}\Upsilon_{F}[\sigma]\prod_{i=1}^{n}\Upsilon_{F}[\tau_i].
\end{equs}
\end{proposition}
\begin{proof}
The proof of this proposition is exactly as in \cite[Proposition 2]{BB21a} but \cite[Equation 2.5]{BB21a} is replaced by
\begin{equs}
\sum_{\substack{l=(l_1,\hdots,l_n)\in(\mbbN^{2})^{n} \\ l_1+\hdots+l_n\le m}}\binom{m}{l_1,\hdots,l_n}\partial^{m-l_i}\partial_{v_{\alpha_i-(0,l_j)}}=\partial_{v_{\alpha_i}}\partial^m
\end{equs}
for $m\in\mbbN^{2}$, and
\begin{equs}
\binom{m}{l_1,\hdots,l_n} \coloneqq \frac{m!}{l_1!\cdots l_n!}
\end{equs}
The proof whereof is elementary.
\end{proof}
\subsection{Coherent Expansions}
In this section, we will formulate our counterpart to the notion of \textit{coherence} first given in \cite[Def. 3.20]{BCCH}. Naturally for the quasilinear case to be amenable under it, our definition needs to be more general. 

To formulate this definition, we remind the readers of the space of modelled distribution $D^{\gamma}_{P}$ (and $D^{\gamma,\eta}_{P}$), with $P=\{(0,x):x\in\mbbT\}$ as defined in \cite[Sec. 6]{reg} and the lift of a smooth function, as it was done in \cite[Eq. 4.11]{reg}. Then we may lift \eqref{eq:transformed} by setting
\begin{equs}\label{eq:nonlindef}
\hat{\mathbf{F}} = \left(1-A'(U)V_c\right)\mathbf{F}&+\left(A(U)\left(A'(U)\right)^{2}\right)V_{cc}(\CD U)^{2} \\ &+(AA'')(U)V_c(\CD U)^{2}+ 2(AA')(U)(\CD U)V_{cx}
\\&+A'(U)(\CD U)V_x,
\end{equs}
where $A(U),\,A'(U)\text{, and }A''(U)$ are the lift of $a(u),\,a'(u),\,a''(u)$ and we define $\mathbf{F}(U) = \mathbf{F}_1(U)(\CD U)^2 + \mathbf{F}_\xi(U)\Xi \coloneqq (F(U) - A'(U)) (\CD U)^{2} + G(U) \Xi $. In the interest of brevity for the upcoming arguments, we set $\mathbf{v}$ as the list of the arguments that $\hat{F}$ takes, namely $u, \partial_x u, v_c, v_{cc}, v_{cx}$, and $v_{x}$, and $\mathbf{V}$ correspondingly is the list of lifts that $\FF$ takes as argument. Moreover, to define $V_\alpha$, we take inspiration from \cite[Eq. 4.1]{MH} in setting
\begin{equ} \label{eq:frI}
V_\alpha(z)=\sum_{|\ell|\le N}\frac{(\tilde A(U)(z))^{ \ell}}{\ell!}
\left[ \mathcal{K}_{\alpha+(\ell,0)} \hat{\mathbf{F}}\right](z),
\end{equ}
where $\tilde A(U) = A(U) - \inprod{A(U)}{\one}\one$ and $\mcK_{\beta}$ is the abstract convolution built from $\CI_\beta$ and $\d^\beta K(c_0,\cdot)$, with the choice $c_0=a(u(z))$ (so in particular $\mathcal{K}_{\alpha+(\ell,0)} \hat{\mathbf{F}}=\mathcal{I}_{\alpha+(\ell,0)} \hat{\mathbf{F}}+v_{\alpha+(\ell,0)}\one+...$, where the ellipses contain polynomials that are not relevant to this article). We will also abide by the following convention
\begin{equs}\label{eq:defU}
U=V_{0,0}(z)=\sum_{\ell\le N}\frac{(\tilde A(U)(z))^{ \ell}}{\ell!}
( \mathcal{K}_{(\ell,0)} \hat{\mathbf{F}}(U))(z).
\end{equs}
In the usual regularity structures paradigm, it is the reconstruction of this, $\CR U$, which solves the (renormalised) equation. The same is true for us here, but the $U$ that \eqref{eq:defU} determines is different from the linear combination of monomials and planted trees in \cite{BCCH}. From \eqref{eq:defU}, it is obvious that $U$ admits the following decomposition:
\begin{equs}\label{eq:defU2}
U = u\one  + u_{\<PlantedNoise>}\<PlantedNoise> + ...,
\end{equs}
where the ellipses denote a sum in trees of the form $\tau = \prod_i\CI^{\zeta_i}_{\alpha_i}(\tau_i)$, with $\tau \neq \CI(\Xi)$. Indeed all of $V_\alpha$ admit a decomposition in products of planted trees. This leads to the following definition for $\left\{V_\alpha\right\}_\alpha$ of elementary differential operators:
\begin{equs}
    \Upsilon_{V_\alpha}\left[\prod_{i=1}^n\CI^{\zeta_i}_{\alpha_i}(\tau_i)\right] = \left(\prod_{i=1}^n\Upsilon_{\hat{F}}\left[\tau_i\right]\right)\partial^{\zeta_n}_{v_{\alpha_n}}\hdots\partial^{\zeta_1}_{v_{\alpha_1}}v_{\alpha}.
\end{equs}
\begin{example} Below, we present some examples of computation:
\begin{equs}
    \Upsilon_{U}\left[\<It0>\right] &= \Upsilon_{\hat{F}}\left[\tau\right]\partial_{v_{(0,0)}}u = \frac{\Upsilon_{\hat{F}}\left[\tau\right]}{q}
\end{equs}
and
\begin{equs}
    \Upsilon_{V_c}\left[\<0I1t11I1t2>\right] &= \Upsilon_{\hat{F}}\left[\tau_1\right]\Upsilon_{\hat{F}}\left[\tau_2\right]\partial_{v_c}\partial_vv_c\\ &= \Upsilon_{\hat{F}}\left[\tau_1\right]\Upsilon_{\hat{F}}\left[\tau_2\right]\partial_{v_c}\left[\frac{a'(u)v_{cc}}{q}\right] \\ &= \Upsilon_{\hat{F}}\left[\tau_1\right]\Upsilon_{\hat{F}}\left[\tau_2\right]\left(\frac{(a'(u))^2v_{cc}}{q^2}\right).
\end{equs}
\end{example}
Manifestly, these new differentials too obey a similar morphism property as in \eqref{eq:morphism}:
\begin{equs}\label{eq:morphism2}
    \Upsilon_{V_{\alpha}}\left[\tau\graft{\beta}\sigma\right] = \Upsilon_{\hat{F}}\left[\tau\right]\partial_{v_\beta}\Upsilon_{V_\alpha}\left[\sigma\right],
\end{equs}
where $\sigma$ necessarily takes the form $\prod_{i}\CI_{\alpha_i}\left(\tau_i\right)$.
\begin{example}
We check:
\begin{equs}
\Upsilon_{U}\left[\tau\graft{}\<PlantedNoise>\right] &= \Upsilon_{U}\left[\begin{tikzpicture}[scale = 0.2,baseline=0cm]
    \node[xi] at (-1,1) (b) {};
    \node[var] at (1,1) (c) {{\tiny $\tau$}};
    \draw[] (0,-0.75) to (b);
    \draw[] (0,-0.75) to (c);
\end{tikzpicture}
\right] + \Upsilon_{U}\left[\begin{tikzpicture}[scale = 0.2,baseline=0cm]
    \node[xi] at (0,0) (a) {};
    \node[var] at (0,2) (b) {{\tiny $\tau$}};
    \draw[] (a) to (b);
    \draw[] (a) to (0,-1.5);
\end{tikzpicture}\right] \\
&= \hat{g}\Upsilon_{\hat F}\left[\tau\right]\partial_{v}^{2}u + \Upsilon_{\hat F}\left[\begin{tikzpicture}[scale = 0.2,baseline=0cm]
    \node[xi] at (1,-0.5) (a) {};
    \node[var] at (0,1.5) (b) {{\tiny $\tau$}};
    \draw[] (a) to (b);
    \end{tikzpicture}\right]\partial_vu \\
    &= \frac{\hat{g}\Upsilon_{\hat F}[\tau]p_c}{q^2} + \frac{\Upsilon_{\hat F}[\tau]}{q^2}\left(g'q - gqp_c\right) \\ &= \frac{\Upsilon_{\hat{F}}[\tau]g'}{q}= \Upsilon_{\hat F}\left[\tau\right]\partial_vg = \Upsilon_{\hat F}\left[\tau\right]\partial_v\Upsilon_{U}\left[\<PlantedNoise>\right].
\end{equs}
\end{example}

\begin{remark}
The proof should mimic what happens for two derivatives. One gets the following recursion on the derivatives:
\begin{equs}
\partial_{v} \partial_v u = a'(u) v_c \partial_{v} \partial_v u + \left( a''(u) v_c + (a'(u))^2 v_{cc} \right) (\partial_{v} u)^{2}.
 \end{equs}
Knowing the value for $ \partial_v u $, one gets from the previous identity the value for $ \partial_{u} \partial_v u $ and this matches the definition of the derivative and the use of the chain and Leibniz rules. Then, one needs to check this in a more general setting when one wants some coherent properties as in \cite{BCCH}.
\end{remark}

We denote, henceforth, by $ T^{\textsf{sat}}_{v_\alpha} $ the set of trees that appear in the decomposition of $V_\alpha$ such that they are generated via the saturated rules. Let us highlight here that it is our restriction to saturated trees that allows us to discount monomials in our argument. This fact is central to the proof of the theorem that follows and also implies that in the expansions of $V_{\alpha}$, we must have that $T^{\,\textsf{sat}}_{v_\alpha}\cap\bar{T}=\{\one\}$ unlike \cite{BCCH}, where this would comprise all monomials. Nonetheless, $v_\alpha$ (the coefficient of $\one$ in $V_\alpha$) play the same role that $u^U_{\alpha}$  plays in \cite{BCCH}.
\begin{theorem}\label{prop:coherence}
    One has for all $\tau\in T^{\textsf{sat}}_{v}$ that
    \begin{equs}\label{eq:Vacoeff}
        \inprod{V_\alpha}{\tau} = \frac{\Upsilon_{V_\alpha}(\tau)}{S(\tau)},
    \end{equs}
 and
 \begin{equs}\label{eq:Fcoeff}
 \inprod{\hat{\mathbf{F}}}{\tau} = \frac{\Upsilon_{\hat{F}}(\tau)}{S(\tau)}.
 \end{equs}
\end{theorem}
\begin{proof}
To pin down the pattern, we begin with the lowest degree symbol in $\hat{\mathbf{F}}$, as follows:
\begin{equs}\label{eq:nonlin}
\hat{\mathbf{F}}&= (\one - A'(U)V_c)\mathbf{F}_{\xi}\,\<generic> + ... = (\one - (a'(u)\one + ...)(v_c\one + ...))\mathbf{F}_{\xi}\Xi \\
&= (\one - a'(u)v_c\one - ...)\left(g(u)\one + g'\tilde{U} + ...\right)\<generic> = qg\,\<generic> + ...,
\end{equs}
where the ellipses contain trees of a higher degree than \<generic>. Compare this to $\Upsilon_{\hat{F}}\left[\<generic>\right] = \hat{g} = gq$. This can now be used to compute the coefficients in the expansion of $U$, indeed from \eqref{eq:frI}, we have that:
\begin{equs}
U = V_{(0,0)} &= \sum_{|\ell|<N}\frac{(\tilde{A}(U))^{\ell}}{\ell!}\left[\CK_{(\ell,0)}\hat{\mathbf{F}}(U)\right](z) \\
&= \CK\hat{\mathbf{F}} + \tilde{A}(U)\CK_{(1,0)}\hat{\mathbf{F}} + ... \\
&= u\one + \CI(\mathbf{\hat F}) + A'(U)\tilde{U}(\CI_{(1,0)}(\mathbf{\hat F}) + v_{(1,0)}) + ... \\
&= u\one + \CI(qg\Xi) + v_{(1,0)}\left(a'(u)(U_{\<PlantedNoise>}\<PlantedNoise>+...)\right) + ... \\
&= u\one + qg\<PlantedNoise> + v_ca'(u)U_{\<PlantedNoise>}\<PlantedNoise> + ...,
\end{equs}
Comparing the coefficients gives us:
\begin{equs}\label{eq:ex1}
U_{\<PlantedNoise>}\<PlantedNoise> = qg\<PlantedNoise> + v_ca'(u)U_{\<PlantedNoise>}\<PlantedNoise> \\
\Leftrightarrow U_{\<PlantedNoise>}= g.
\end{equs}
This allows one to compute:
 \begin{equs}
 A(U) = \sum_{k}\frac{(a^{(k)})}{k!}\tilde{U}^{k} =a(u)\one + (a'g)(u)\<PlantedNoise> + ...,
 \end{equs}
 where $\tilde{U}^{k} = \overbrace{\tilde{U}\cdot\hdots\cdot\tilde{U}}^{k\text{ products in total}}$.
 The same line of reasoning yields the following equations:
 \begin{alignat}{2}\label{eq:expansions}
 &V_c &&= v_c\one + (a'g)(u)v_{cc}\<PlantedNoise> + (qg)(u)\!\!\<PlantedNoiseC>  + ... \nonumber \\
 &V_{cx} &&= v_{cx}\one + (a'g)(u)v_{ccx}\<PlantedNoise> + qg(u)\!\!\<PlantedNoiseCX> + ...\nonumber \\
 &V_{cc} &&= v_{cc}\one + (a'g)(u)v_{ccc}\<PlantedNoise> + qg(u)\!\!\<PlantedNoiseCC> + ... \nonumber \\
 &V_x &&= v_x\one + (a'g)(u)v_{cx}\<PlantedNoise> + qg(u)\<PlantedNoiseX> + ...
 \end{alignat}
One also checks that $\Upsilon_{V_\alpha}\left[\<PlantedNoise>\right]$ gives us exactly the coefficients as in \eqref{eq:ex1} and \eqref{eq:expansions}. This confirms the result for $\<PlantedNoise>$, which one can then use to get the coefficients of bigger trees in $\FF$ and from there the coefficients of trees in $\{V_\alpha\}_{\alpha}$. This interdependence suggests that an inductive proof of this proposition is possible which we now show to be the case. We show first that much like coefficients of smaller trees in $\FF$ can be used to compute coefficients of bigger trees in $V_\alpha$, \eqref{eq:Vacoeff} for smaller trees can be used to get \eqref{eq:Fcoeff} for larger ones. As the upcoming argument is not peculiar to our specific choice of $\FF$ (beyond the fact that it is possible to decompose the corresponding non-linearity $\hat{F}$ into $\hat{F}_{\one} + \hat{F}_{\Xi}\xi$), we pose it in terms of a general $\hat{\mathbf{F}}$ which takes as input $\mathbf{v}=\{\mathbf{v}_1,\hdots,\mathbf{v}_n\}$ and denote by $\mathbf{V}$ the vector of the corresponding lifts. We also introduce the following notation: $\CP_\tau$, which is the real-valued linear map that kills all trees that cannot be identified with $\tau$ (in the sense of tree products), and is $1$ on products of trees that can; $\Pi^{(n)}_{\tau}$, which is the set of all non-trivial decomposition of $\tau$ into $n$ trees, by which we mean that it comprises of all $n$-tuples $(\tau_1,\hdots,\tau_n)$, such that $\prod_{i=1}^{n}\tau_i = \tau$; and for a tree $\tau=\prod_{i=1}^{n}\CI_{\alpha_i}(\tau_i)$, the following map  $|\prod_{i=1}^n\CI(\tau_i)|_0 \coloneqq n$. A generalised variant of \cite[Eq 4.11]{reg} allows us to write:
\begin{equs}\label{eq:2.17p1}
\FF_{\tau} &= \CP_{\tau}\left[\sum_{\substack{k\neq 0 \\ k_1+\hdots +k_{n} = k}}\hspace{-6mm}\hat{F}^{(n)}\prod_{i}^{n}\left[\mathbf{V}_{i}-\mathbf{v}_i\right]^{k_i}\right] \\
&= \sum_{\substack{k\neq 0 \\ k_1+\hdots +k_{n} = k}}\hspace{-6mm}\hat{F}_{\mfl}^{(n)}\sum_{\substack{(\rho_1,\hdots,\rho_n)  \\ \in \Pi^{(n)}_{\tau}}}\left(\prod_{i=1}^{n}\left(\sum_{\substack{(\rho_{1,i},\hdots,\rho_{i,k_i})\\ \in \Pi^{(k_1)}_{\rho_i}}}\prod_{j=1}^{k_j}\frac{\Upsilon_{\mathbf{V}_i}[\rho_{i,j}]}{S(\rho_{(i,j)})}\right)\right),
\end{equs}
where $\hat{F}^{(n)}\coloneqq\frac{\partial^{k}\hat{F}}{\partial^{(k_1)}_{\mathbf{v}_1}\cdots\partial^{(k_n)}_{\mathbf{v}_n}}$ and the second inequality is due to the inductive hypothesis, that is we assume that for $\tau=\prod_{i=1}^{n}\CI^{\beta_i}_{\alpha_i}[\tau_i]$ (and trees smaller) $\eqref{eq:Vacoeff}$ holds. On the other hand, one has due to the Faa di Bruno formula:
\begin{equs}\label{eq:2.17p2}
\frac{\Upsilon_{\hat{F}}[\tau]}{S(\tau)} &= \frac{1}{S(\tau)}\left(\prod_{i}^n\Upsilon_{\hat F}[\tau_i]\right)\left(\partial_{\mathbf{v}_1}\cdots\partial_{\mathbf{v}_n}{F}_\mathfrak{l}\right) \\
&= \frac{1}{S(\tau)}\left(\prod_{i}^n\Upsilon_{\hat F}[\tau_i]\right)\left(\sum_{\pi\in\Pi}F_\mfl^{(|\pi|)}(\mathbf{v})\prod_{B\in\pi}\Big(\prod_{j\in B}\partial_{\mathbf{v}_{j}}\Bigr)\mathbf{v}\right),
\end{equs}
where $\pi$ runs through the set of all partitions, $\Pi$, of the set $\{1,\hdots,n\}$, "$B\in\pi$" means the variable $B$ runs through the list of all of the "blocks" of the partition $\pi$, and $|\pi|$ and $|B|$ as is usual represent the cardinality of the sets. We need to argue that \eqref{eq:2.17p1} and \eqref{eq:2.17p2} are the same, but this is easy to see. In fact, notice that all the $\rho_{i,j}$ are of the form $\prod_{k=1}^{|\rho_{i,j}|_0}\CI_{\alpha_k}\left(\rho_{(i,j,k)}\right)$, where $\CI_{\alpha_k}\left(\rho_{(i,j,k)}\right)\in\{\CI_{\alpha_1}(\tau_{1}),\hdots,\CI_{\alpha_n}(\tau_{n})\}$, which means that when $\Upsilon_{\mathbf{V}_i}$ is applied to $\rho_{(i,j)}$ we will get some product of terms like $\Upsilon_{\hat{F}}(\tau_i)$ multiplied to some derivatives. As these are ultimately decompositions of $\tau$, one is able to factor out $\prod_{i}^n\Upsilon_{\hat F}[\tau_i]$ from \eqref{eq:2.17p1}, and in the sum we are left with all the various combinations of the derivatives. It is not difficult to convince one's self that it is the same sum of derivatives in either equation. It only remains to see that the $S(\tau)$ in \eqref{eq:2.17p2} can be written as a product of the $S(\rho_{(i,j)})$. Recall that for $\tau=\prod_i^n\CI_{\alpha_i}^{\beta_i}(\tau_i)$, $S(\tau)=\prod_{i}^{n}S(\tau_i)^{\beta_i}\beta_i!$ and fix some decomposition of $\tau$. If in this decomposition $\CI_{\alpha_i}^{\beta_i}(\tau_i)$ ends up completely in some disjoint $\rho_{(i,j)}$, one directly gets the needed $S(\tau_i)^{\beta_i}\beta_i!$ factor. Alternatively, if $\CI_{\alpha_i}^{\beta_i}(\tau_i)$ ends up decomposed, which is to say for each $\ell$, $\CI_{\alpha_i}^{\beta_{i,\ell}}(\tau_i)$ ends up in distinct $\rho_{(i,j)}$, where $\prod_{\ell}\CI_{\alpha_i}^{\beta_{i,\ell}}(\tau_i) = \CI_{\alpha_i}^{\beta_{i}}(\tau_i)$ and $\sum_{\ell}\beta_{i,\ell}= \beta_i$, means that in \eqref{eq:2.17p1}, one sees the term $\prod_{\ell}S(\tau_{i})^{\beta_{i,\ell}}\beta_{i,\ell}!=S(\tau_i)^{\beta_i}\prod_{\ell}\beta_{i,\ell}!$. As this decomposition can be made in $\binom{\beta_i}{\prod_\ell\beta_{i,\ell}}$ ways, we see that this term is repeated the same number of times, wherefrom we get the correct combinatorial factors. With this we have shown how to go from $\Upsilon_{V_\alpha}$ to $\Upsilon_{\hat{F}}$. It remains to contend with the other direction.
To this end, recall the morphism property \eqref{eq:morphism2} for a fixed $\tau$:
\begin{equs}\label{eq:base1}
\Upsilon_{U}(\tau\!\graft{\alpha}\!\sigma) = \Upsilon_{\hat F}(\tau)\partial_{v_\alpha}\Upsilon_{U}(\sigma).
\end{equs}
Hence to prove $U_{\tau\graft{\alpha}\sigma} = \Upsilon_{U}\left[\tau\graft{\alpha}\sigma\right]/S(\tau\graft{\alpha}\sigma)$, it is enough to prove:
\begin{equs}
    S(\tau\graft{\alpha}\sigma)U_{\tau\graft{\alpha}\sigma} = \Upsilon_{\hat{F}}\left[\tau\right]\partial_{v_\alpha}\Upsilon_{U}\left[\sigma\right],
\end{equs}
where we are abusing notation $\tau\graft{\alpha}\sigma$, as it refers to a sum of trees and the equation above is to be read individually for every summand in it. Seeing as $U$ only consists of products of planted trees, we may set $\sigma = \prod_i^n\CI_{\alpha_i}(\sigma_i)$ and fix $\tau$, for which one is able to compute:
\begin{equs}\label{eq:base2}
\tau\graft{}\sigma &= \underbrace{ 
\begin{tikzpicture}[scale=0.2,baseline=1]
\node[var] at (-3,1) (a) {\tiny{$\sigma_{\!1}$}}; 
\node[var] at (3,1) (b) {\tiny{$\tau$}};
\node[var] at (3.5,3) (d) {{\tiny $\sigma_{\!n}$}};
\node[var] at (-3.5,3) (c) {\tiny{$\sigma_{\!2}$}};
\node at (0,1) {{\tiny ...}};
\draw (0,0) -- (a);
\draw (0,0) -- (b);
\draw (0,0) -- (c);
\draw (0,0) -- (d);
\node at (-1.5,-0.15) {{\tiny $\alpha_1$}};
\node at (-1.25,2) {\tiny{$\alpha_2$}};
\node at (1.25,2) {\tiny{$\alpha_n$}};
\end{tikzpicture}}_{\coloneqq\tau_0} + \underbrace{\begin{tikzpicture}[scale=0.2,baseline=0]
\node[sqvar] at (-4.5,2) (a) {{\tiny $\tau\!\!\graft{}\!\!\sigma_{\!1}$}};
\node[var] at (0,4) (b) {{\tiny $\sigma_{\!2}$}};
\node[var] at (3,2) (c) {{\tiny $\sigma_{\!n}$}};
\draw (0,0) -- (a);
\draw (0,0) -- (b);
\draw (0,0) -- (c);
\node at (-2,0.1) {\tiny{$\alpha_1$}};
\node at (-0.9,2) {\tiny{$\alpha_2$}};
\node[rotate=315] at (0.75,1.5) {\tiny{...}};
\node at (2,0.15) {\tiny{$\alpha_n$}};
\end{tikzpicture}}_{\coloneqq\tau_1} \\ &+ \underbrace{\begin{tikzpicture}[scale=0.2,baseline=2]
\node[var] at (-3,2) (a) {{\tiny $\sigma_{\!1}$}};
\node[sqvar] at (0,4) (b) {{\tiny $\tau\!\!\graft{}\!\!\sigma_{\!2}$}};
\node[var] at (3,2) (c) {{\tiny $\sigma_{\!n}$}};
\draw (0,0) -- (a);
\draw (0,0) -- (b);
\draw (0,0) -- (c);
\node at (-2,0.5) {\tiny{$\alpha_1$}};
\node at (-0.9,2) {\tiny{$\alpha_2$}};
\node[rotate=315] at (0.75,1.5) {\tiny{...}};
\node at (2,0.15) {\tiny{$\alpha_n$}};
\end{tikzpicture}}_{\coloneqq\tau_2} + ... + \underbrace{\begin{tikzpicture}[scale=0.2,baseline=2]
\node[var] at (-3,2) (a) {{\tiny $\sigma_{\!1}$}};
\node[var] at (0,4) (b) {{\tiny $\sigma_{\!2}$}};
\node[sqvar] at (4.5,2) (c) {{\tiny $\tau\!\!\graft{}\!\!\sigma_{\!n}$}};
\draw (0,0) -- (a);
\draw (0,0) -- (b);
\draw (0,0) -- (c);
\node at (-2,0.5) {\tiny{$\alpha_1$}};
\node at (-0.9,2) {\tiny{$\alpha_2$}};
\node[rotate=315] at (0.75,1.5) {\tiny{...}};
\node at (2,0.15) {\tiny{$\alpha_{n}$}};
\end{tikzpicture}}_{\coloneqq\tau_n},
\end{equs}
where the ellipses in the trees encode $\CI_{\alpha_i}(\sigma_i)$, $i\in\{3,\,4,\hdots,n-1\}$. We caution here that none of the $\tau_i$ are actual trees but rather sums of trees. The trees that have $\begin{tikzpicture}
    \node[sqvar] at (0,0) {};
\end{tikzpicture}$ are the sum of all the trees that are generated by $\tau\graft{}\sigma_i$ as one continues to graft back. 
Recall that $U$ takes the following form:
\begin{equs}\label{eq:proof1}
U = u\one + \sum_{\tau}U_\tau\tau,
\end{equs}
where $\tau$ runs over products of planted trees, i.e. $\tau = \prod_{i}\CI_{\alpha_i}(\tau_i)$. We recall from the definition that one has:
\begin{equs}\label{eq:proof2}
U = \sum_{\ell} \frac{(\tilde{A}(U))^{\ell}}{\ell!}\left[\CK_{(\ell,0)}\hat{\mathbf{F}}\right]
\end{equs}
where we know:
\begin{equs}
\tilde{A}(U) = \sum_{k \neq 0}\frac{D^{k}[a(u)]}{k!}(U - u\one)^{k}
\end{equs}
By substituting \eqref{eq:proof1} into this, one gets:
\begin{equs}
\tilde{A}(U) = \sum_{k \neq 0}\frac{D^{k}[a(u)]}{k!}\left[\sum_{\tau} U_\tau\tau\right]^{k}
\end{equs}
which in turn can be put into \eqref{eq:proof2}, yielding:
\begin{equs}
U = \sum_{\ell}\frac{1}{\ell!}\left(\sum_{k\neq 0}\frac{D^{k}[a(u)]}{k!}\left[\sum_{\tau} U_\tau\tau\right]^{k}\right)^{\ell}\left[\CK_{(\ell,0)}\hat{\mathbf{F}}\right] 
\end{equs}
where now the product is not the tree product but some forest product, in the sense that $ \tau_1 \tau_2 $ is not the joint root product.
Notice that any arbitrary $\sigma$ in the expansion of $U$ is also seen on the right-hand side (without identifying it with products of smaller trees) when $\ell=1,\,k=1$. By rearranging the expansion one gets:
\begin{equs}
    \left(U_{\sigma} - a'(u)v_cU_{\sigma}\right) &= \tilde{\CP}_\sigma\left[\sum_{\ell}\frac{1}{\ell!}\left(\sum_{k\neq 0}\frac{D^{k}[a(u)]}{k!}\left[\sum_{\tau\neq\sigma} U_\tau\tau\right]^{k}\right)^{\ell}\left[\CK_{(\ell,0)}\hat{\mathbf{F}}\right]\right] \\
    \Leftrightarrow U_\sigma &= \tilde{\CP}_\sigma\left[\frac{1}{q}\sum_{\ell}\frac{1}{\ell!}\left(\sum_{k\neq 0}\frac{D^{k}[a(u)]}{k!}\left[\sum_{\tau\neq\sigma} U_\tau\tau\right]^{k}\right)^{\ell}\left[\CK_{(\ell,0)}\hat{\mathbf{F}}\right]\right].
\end{equs}
where the projection $ \tilde{\CP}_{\sigma} $ extracts the coefficients of all the forests that preserve the order of the edges attached to the root of $ \sigma $. We can perform this identification as these elements have the same analytical interpretation.
  
Notice that the only forests that survive with $\tilde{\CP}_\sigma$ contain trees which are strictly smaller than $\sigma$, which means that due to the induction step, we can reformulate the equation above:
\begin{equs}\label{eq:DecompUpsilonU}
\frac{\Upsilon_{U}(\sigma)}{S(\sigma)} = \tilde{\CP}_\sigma\Biggl[\frac{1}{q}\sum_{\ell}\frac{1}{\ell!}\left(\sum_{k\neq 0}\frac{D^{k}[a(u)]}{k!}\left[\sum_{\tau\neq\sigma} \frac{\Upsilon_U(\tau)}{S(\tau)}\tau\right]^{k}\right)^{\ell}\\
\times\left[\CK_{(\ell,0)}\hat{\mathbf{F}}\right]\Biggr]
\end{equs}
The equation above tells us that the coefficient of a particular tree can be calculated as a linear combination of its non-trivial decompositions, and we use this to motivate the following decomposition.
\begin{align*}
&\frac{\Upsilon_{U}[\sigma]}{S(\sigma)} = \frac{1}{q}\left[\sum_{\ell=1}^{|\sigma|_0}\frac{1}{\ell!}\sum_{\substack{(\rho_1,\hdots,\rho_\ell) \\ \hphantom{aaaa}\in\Pi^{(\ell)}_\sigma}}\left[v_{(\ell,0)}\prod_{i=1}^{\ell}\left(\sum_{k=1}^{|\rho_i|_0}\frac{a^{(k)}}{k!}\left(\sum_{\substack{(\nu_1,\hdots,\nu_k) \\ \hphantom{aaaa}\in\CP^{(k)}_{\rho_i}}}\prod_{j=1}^k\frac{\Upsilon_U[\nu_j]}{S(\nu_j)}\right)\right)\right]\right. \\
&+ \left.\sum_{\ell=1}^{|\sigma|_0}\frac{1}{\ell!}\hspace{-5mm}\sum_{\substack{(\rho_1,\hdots,\rho_{\ell}, \\ \hphantom{aaaa}\CI_{\alpha_n}(\sigma_n))\in\Pi^{(\ell+1)}_\sigma}}\left[\frac{\Upsilon_{\hat{F}}[\sigma_n]}{S\left(\sigma_n\right)}\prod_{i=1}^{\ell}\left(\sum_{k=1}^{|\rho_i|_0}\frac{a^{(k)}}{k!}\left(\sum_{\substack{(\nu_1,\hdots,\nu_k) \\ \hphantom{aaaa}\in\Pi^{(k)}_{\rho_i}}}\prod_{j=1}^k\frac{\Upsilon_U[\nu_j]}{S(\nu_j)}\right)\right)\right]\right]
\end{align*}
We first make the argument for $\alpha=(0,0)$, for which we need to differentiate as follows:
\begin{align*}
&\frac{\partial_v\Upsilon_{U}[\sigma]}{S(\sigma)}=\frac{p_c\Upsilon_{U}(\sigma)}{q^2S(\sigma)} + \frac{1}{q}\sum_{\ell=1}^{|\sigma|_0}\frac{1}{\ell!}\hspace{-2mm}\sum_{\substack{(\rho_1,\hdots,\rho_\ell) \\ \hphantom{aaaa}\in\Pi^{(\ell)}_\sigma}}\left[\frac{a'v_{(\ell+1,0)}}{q}\prod_{i=1}^{\ell}\left(\sum_{k=1}^{|\rho_i|_0}\frac{a^{(k)}}{k!}\left(\sum_{\substack{(\nu_1,\hdots,\nu_k) \\ \hphantom{aaaa}\in\Pi^{(k)}_{\rho_i}}}\prod_{j=1}^k\frac{\Upsilon_U[\nu_j]}{S(\nu_j)}\right)\right)\right.\\
&+ v_{(\ell,0)}\sum_{m=1}^{\ell}\prod_{\substack{i=1 \\ i\neq m}}^\ell\left[\sum_{k=1}^{|\rho_i|_0}\frac{a^{(k)}}{k!}\left(\sum_{\substack{(\nu_1,\hdots,\nu_k) \\ \hphantom{aaaa}\in\Pi^{(k)}_{\rho_i}}}\prod_{j=1}^k\frac{\Upsilon_U[\nu_j]}{S(\nu_j)}\right)\right]\left(\sum_{k=1}^{|\rho_m|_0}\frac{a^{(k+1)}}{k!q}\left(\sum_{\substack{(\nu_1,\hdots,\nu_k) \\ \hphantom{aaaa}\in\Pi^{(k)}_{\rho_m}}}\prod_{j=1}^k\frac{\Upsilon_U[\nu_j]}{S(\nu_j)}\right)\right. \\
&+\left.\sum_{k=1}^{|\rho_m|_0}\frac{a^{(k)}}{k!}\left(\sum_{\substack{(\nu_1,\hdots,\nu_k) \\ \hphantom{aaaa}\in\Pi^{(k)}_{\rho_m}}}\sum_{n=1}^k\frac{\partial_{v}\Upsilon_U[\nu_n]}{S(\nu_n)}\prod_{\substack{j=1 \\ j\neq n}}^k\frac{\Upsilon_U[\nu_j]}{S(\nu_j)}\right)\right) + \sum_{\ell=1}^{|\sigma|_0}\frac{1}{\ell!}\hspace{-5mm}\sum_{\substack{(\rho_1,\hdots,\rho_{\ell}, \\ \hphantom{aaaa}\CI(\sigma_n))\in\Pi^{(\ell+1)}_\sigma}}\hspace{-2mm}\frac{\partial_v\Upsilon_{\hat{F}}[\sigma_n]}{S\left(\sigma_n\right)}
\end{align*}
\begin{align*}
&\prod_{i=1}^{\ell}\left(\sum_{k=1}^{|\rho_i|_0}\frac{a^{(k)}}{k!}\left(\sum_{\substack{(\nu_1,\hdots,\nu_k) \\ \hphantom{aaaa}\in\Pi^{(k)}_{\rho_i}}}\prod_{j=1}^k\frac{\Upsilon_U[\nu_j]}{S(\nu_j)}\right)\right) + \frac{\Upsilon_{\hat{F}}[\sigma_n]}{S\left(\sigma_n\right)}\sum_{m=1}^{\ell}\prod_{\substack{i=1 \\ i\neq m}}^\ell\left[\sum_{k=1}^{|\rho_i|_0}\frac{a^{(k)}}{k!}\left(\sum_{\substack{(\nu_1,\hdots,\nu_k) \\ \hphantom{aaaa}\in\Pi^{(k)}_{\rho_i}}}\prod_{j=1}^k\frac{\Upsilon_U[\nu_j]}{S(\nu_j)}\right)\right] \\
&\left(\sum_{k=1}^{|\rho_m|_0}\frac{a^{(k+1)}}{k!q}\left(\sum_{\substack{(\nu_1,\hdots,\nu_k) \\ \hphantom{aaaa}\in\Pi^{(k)}_{\rho_m}}}\prod_{j=1}^k\frac{\Upsilon_U[\nu_j]}{S(\nu_j)}\right)\right.
+\left.\sum_{k=1}^{|\rho_m|_0}\frac{a^{(k)}}{k!}\left(\sum_{\substack{(\nu_1,\hdots,\nu_k) \\ \hphantom{aaaa}\in\Pi^{(k)}_{\rho_m}}}\sum_{n=1}^k\frac{\partial_{v}\Upsilon_U[\nu_n]}{S(\nu_n)}\prod_{\substack{j=1 \\ j\neq n}}^k\frac{\Upsilon_U[\nu_j]}{S(\nu_j)}\right)\right)
\end{align*}
One computes then:
\begin{equs}
\Upsilon_{\hat{F}}[\tau]\partial_v\Upsilon_{U}[\sigma] &= \Upsilon_{\hat{F}}[\tau]\left[\frac{p_c\Upsilon_{U}[\sigma]}{q^{2}}+\hdots\right] \\
&=\frac{p_c\Upsilon_{\hat{F}}[\tau]\Upsilon_{U}[\sigma]}{q^{2}} + \hdots \\
&=\frac{p_c\Upsilon_{U}\left[\CI(\tau)\right]\Upsilon_{U}[\sigma]}{q} + \hdots,
\end{equs}
where for the third inequality we have used that $\Upsilon_{U}\left[\CI(\tau)\right]=\Upsilon_{\hat{F}}[\tau]/q$. Naturally, one might suspect that the above term can be found in $\Upsilon_U\left[\tau_0\right]$, and indeed by considering $\ell=1$ (which forces $k=2$) with the decomposition $\left(\sigma,\CI(\tau)\right)\in\Pi^{(2)}_{\tau_0}$ and also $\ell=2$ (which forces $k=1$) with the same decomposition, and one gets:
\begin{equs}
\Upsilon_U\left[\tau_0\right] &= \frac{S(\tau_0)}{q}\left[\frac{\left(a'(u)\right)^2v_{cc} + a''(u)v_{c}}{q}\left(\frac{  m(\tau_0,\tau)\Upsilon_U\left[\CI(\tau)\right]\Upsilon_U[\sigma]}{qS\left(\CI(\tau)\right)S\left(\sigma\right)}\right)\right] + \hdots \\
&= \frac{p_c\Upsilon_{U}\left[\CI(\tau)\right]\Upsilon_{U}[\sigma]}{q} + \hdots
\end{equs}
where $ m(\tau_0,\tau) $ is the number of times $ \CI(\tau) $ appear on the right-most location in a commutative block of edges. Then, the combinatorial factors cancel out and we get the identity we wanted to prove.
Consider now the term:
\begin{align*}
&\Upsilon_{\hat{F}}[\tau]\left[\frac{S(\sigma) }{q }\sum_{\ell=1}^{|\sigma|_0}\frac{1}{\ell!}\hspace{-2mm}\sum_{\substack{(\rho_1,\hdots,\rho_\ell) \\ \hphantom{aaaa}\in\Pi^{(\ell)}_\sigma}}\frac{a'v_{(\ell+1,0)}}{q}\prod_{i=1}^{\ell}\left(\sum_{k=1}^{|\rho_i|_0}\frac{a^{(k)}}{k!}\left(\sum_{\substack{(\nu_1,\hdots,\nu_k) \\ \hphantom{aaaa}\in\Pi^{(k)}_{\rho_i}}}\prod_{j=1}^k\frac{\Upsilon_U[\nu_j]}{S(\nu_j)}\right)\right)\right] \\
&= \frac{1}{q}\sum_{\ell=1}^{|\sigma|_0}\frac{S(\tau)S(\sigma)}{\ell!}\hspace{-2mm}\sum_{\substack{(\rho_1,\hdots,\rho_\ell,\CI(\tau)) \\ \hphantom{aaaa}\in\Pi^{(\ell+1)}_{\tau_0}}}v_{(\ell+1,0)}\prod_{i=1}^{\ell}\left(\sum_{k=1}^{|\rho_i|_0}\frac{a^{(k)}}{k!}\left(\sum_{\substack{(\nu_1,\hdots,\nu_k) \\ \hphantom{aaaa}\in\Pi^{(k)}_{\rho_i}}}\prod_{j=1}^k\frac{\Upsilon_U[\nu_j]}{S(\nu_j)}\right)\right) \\
&\hphantom{aaaaaaaaa}\hspace{7cm}\Biggl(a^{(1)}\frac{\Upsilon_{U}\left[\CI(\tau)\right]}{S(\CI(\tau))}\Biggr)
\end{align*}
We can compare this to $\Upsilon_{U}[\tau_0]$, with the partition $\{\rho_1,\hdots,\rho_\ell,\CI(\tau)\}\in\Pi^{(\ell+1)}_{\tau_0}$. It remains to match the factor $S(\tau_0)$ to $S(\tau)S(\sigma)$ which is dealt with the same way as before.
The next term we consider is
\begin{align*}
&S(\sigma)\Upsilon_{\hat{F}}[\tau]v_{(\ell,0)}\sum_{m=1}^{\ell}\prod_{\substack{i=1 \\ i\neq m}}^\ell\left[\sum_{k=1}^{|\rho_i|_0}\frac{a^{(k)}}{k!}\left(\sum_{\substack{(\nu_1,\hdots,\nu_k) \\ \hphantom{aaaa}\in\Pi^{(k)}_{\rho_i}}}\prod_{j=1}^k\frac{\Upsilon_U[\nu_j]}{S(\nu_j)}\right)\right]
\end{align*}
\begin{align*}
&\hspace{4cm}\left[\sum_{k=1}^{|\rho_m|_0}\frac{a^{(k+1)}}{k!q}\left(\sum_{\substack{(\nu_1,\hdots,\nu_k) \\ \hphantom{aaaa}\in\Pi^{(k)}_{\rho_m}}}\prod_{j=1}^k\frac{\Upsilon_U[\nu_j]}{S(\nu_j)}\right)\right]\\
&=S(\sigma)S(\tau)v_{(\ell,0)}\sum_{m=1}^{\ell}\prod_{\substack{i=1 \\ i\neq m}}^\ell\left[\sum_{k=1}^{|\rho_i|_0}\frac{a^{(k)}}{k!}\left(\sum_{\substack{(\nu_1,\hdots,\nu_k) \\ \hphantom{aaaa}\in\Pi^{(k)}_{\rho_i}}}\prod_{j=1}^k\frac{\Upsilon_U[\nu_j]}{S(\nu_j)}\right)\right] \\
&\hspace{4cm}\left[\sum_{k=1}^{|\rho_m|_0}\frac{a^{(k+1)}}{(k+1)!}\left(\frac{\Upsilon_U[\CI(\tau)]}{S(\CI\left(\tau)\right)}\hspace{-2mm}\sum_{\substack{(\nu_1,\hdots,\nu_k,\CI(\tau)) \\ \hphantom{aaaa}\in\Pi^{(k+1)}_{\rho_m\CI(\tau)}}}\prod_{j=1}^k\frac{\Upsilon_U[\nu_j]}{S(\nu_j)}\right)\right]
\end{align*}
We again compare with $\Upsilon_{U}[\tau_0]$ and the argument is the same as before.
\begin{align*}
&S(\sigma)\Upsilon_{\hat{F}}[\tau]v_{(\ell,0)}\sum_{m=1}^{\ell}\prod_{\substack{i=1 \\ i\neq m}}^\ell\left[\sum_{k=1}^{|\rho_i|_0}\frac{a^{(k)}}{k!}\left(\sum_{\substack{(\nu_1,\hdots,\nu_k) \\ \hphantom{aaaa}\in\Pi^{(k)}_{\rho_i}}}\prod_{j=1}^k\frac{\Upsilon_U[\nu_j]}{S(\nu_j)}\right)\right] \\
&\hspace{4cm}\sum_{k=1}^{|\rho_m|_0}\frac{a^{(k)}}{k!}\left(\sum_{\substack{(\nu_1,\hdots,\nu_k) \\ \hphantom{aaaa}\in\Pi^{(k)}_{\rho_m}}}\sum_{n=1}^k\frac{\partial_{v}\Upsilon_U[\nu_n]}{S(\nu_n)}\prod_{\substack{j=1 \\ j\neq n}}^k\frac{\Upsilon_U[\nu_j]}{S(\nu_j)}\right) \\
&=S(\sigma)v_{(\ell,0)}\sum_{m=1}^{\ell}\prod_{\substack{i=1 \\ i\neq m}}^\ell\left[\sum_{k=1}^{|\rho_i|_0}\frac{a^{(k)}}{k!}\left(\sum_{\substack{(\nu_1,\hdots,\nu_k) \\ \hphantom{aaaa}\in\Pi^{(k)}_{\rho_i}}}\prod_{j=1}^k\frac{\Upsilon_U[\nu_j]}{S(\nu_j)}\right)\right] \\
&\hspace{4cm}\sum_{k=1}^{|\rho_m|_0}\frac{a^{(k)}}{k!}\left(\sum_{\substack{(\nu_1,\hdots,\nu_k) \\ \hphantom{aaaa}\in\Pi^{(k)}_{\rho_m}}}\sum_{n=1}^k\frac{\Upsilon_U[\tau\graft{}\nu_n]}{S(\nu_n)}\prod_{\substack{j=1 \\ j\neq n}}^k\frac{\Upsilon_U[\nu_j]}{S(\nu_j)}\right)
\end{align*}
One has to be careful that $\tau\graft{}\nu_n$ is not one tree, but finitely many. The first of these terms - i.e. the root grafting - is to be found in $\Upsilon_U[\tau_0]$, while for the rest, one has to pick the relevant $\Upsilon_U[\tau_i]$, $i\neq 0$. The combinatorial factors are matched as before. Similarly, we take care of:
\begin{align*}
&S(\sigma)\Upsilon_{\hat{F}}[\tau]\sum_{\ell=1}^{|\sigma|_0}\frac{1}{\ell!}\hspace{-5mm}\sum_{\substack{(\rho_1,\hdots,\rho_{\ell}, \\ \hphantom{aaaa}\CI(\sigma_n))\in\Pi^{(\ell+1)}_\sigma}}\hspace{-3mm}\frac{\partial_v\Upsilon_{\hat{F}}[\sigma_n]}{S\left(\sigma_n\right)}\prod_{i=1}^{\ell}\left(\sum_{k=1}^{|\rho_i|_0}\frac{a^{(k)}}{k!}\left(\sum_{\substack{(\nu_1,\hdots,\nu_k) \\ \hphantom{aaaa}\in\Pi^{(k)}_{\rho_i}}}\prod_{j=1}^k\frac{\Upsilon_U[\nu_j]}{S(\nu_j)}\right)\right)\\
&=S(\sigma)\sum_{\ell=1}^{|\sigma|_0}\frac{1}{\ell!}\hspace{-5mm}\sum_{\substack{(\rho_1,\hdots,\rho_{\ell}, \\ \hphantom{aaaa}\CI(\sigma_n))\in\Pi^{(\ell)}_\sigma}}\hspace{-2mm}\frac{\Upsilon_{\hat{F}}[\tau\graft{}\sigma_n]}{S\left(\sigma_n\right)}\prod_{i=1}^{\ell}\left(\sum_{k=1}^{|\rho_i|_0}\frac{a^{(k)}}{k!}\left(\sum_{\substack{(\nu_1,\hdots,\nu_k) \\ \hphantom{aaaa}\in\Pi^{(k)}_{\rho_i}}}\prod_{j=1}^k\frac{\Upsilon_U[\nu_j]}{S(\nu_j)}\right)\right)
\end{align*}
Next, consider:
\begin{align*}
&S(\sigma)\Upsilon_{\hat{F}}[\tau]\frac{\Upsilon_{\hat{F}}[\sigma_n]}{S\left(\sigma_n\right)}\sum_{m=1}^{\ell}\prod_{\substack{i=1 \\ i\neq m}}^\ell\left[\sum_{k=1}^{|\rho_i|_0}\frac{a^{(k)}}{k!}\left(\sum_{\substack{(\nu_1,\hdots,\nu_k) \\ \hphantom{aaaa}\in\Pi^{(k)}_{\rho_i}}}\prod_{j=1}^k\frac{\Upsilon_U[\nu_j]}{S(\nu_j)}\right)\right]\\
&\hspace{4cm}\left(\sum_{k=1}^{|\rho_m|_0}\frac{a^{(k+1)}}{k!q}\left(\sum_{\substack{(\nu_1,\hdots,\nu_k) \\ \hphantom{aaaa}\in\Pi^{(k)}_{\rho_m}}}\prod_{j=1}^k\frac{\Upsilon_U[\nu_j]}{S(\nu_j)}\right)\right)\\
&=S(\sigma)S(\tau)\frac{\Upsilon_{\hat{F}}[\sigma_n]}{S\left(\sigma_n\right)}\sum_{m=1}^{\ell}\prod_{\substack{i=1 \\ i\neq m}}^\ell\left[\sum_{k=1}^{|\rho_i|_0}\frac{a^{(k)}}{k!}\left(\sum_{\substack{(\nu_1,\hdots,\nu_k) \\ \hphantom{aaaa}\in\Pi^{(k)}_{\rho_i}}}\prod_{j=1}^k\frac{\Upsilon_U[\nu_j]}{S(\nu_j)}\right)\right]
\end{align*}
\begin{align*}
&\hspace{3.5cm}\left(\sum_{k=1}^{|\rho_m|_0}\frac{a^{(k+1)}}{(k+1)!}\left(\frac{\Upsilon_{U}[\CI(\tau)]}{S(\CI(\tau))}\sum_{\substack{(\nu_1,\hdots,\nu_k) \\ \hphantom{aaaa}\in\Pi^{(k)}_{\rho_m}}}\prod_{j=1}^k\frac{\Upsilon_U[\nu_j]}{S(\nu_j)}\right)\right)
\end{align*}
The reasoning for this is the same as before. Finally one has:
\begin{align*}
&S(\sigma)\Upsilon_{\hat{F}}[\tau]\frac{\Upsilon_{\hat{F}}[\sigma_n]}{S\left(\sigma_n\right)}\sum_{m=1}^{\ell}\prod_{\substack{i=1 \\ i\neq m}}^\ell\left[\sum_{k=1}^{|\rho_i|_0}\frac{a^{(k)}}{k!}\left(\sum_{\substack{(\nu_1,\hdots,\nu_k) \\ \hphantom{aaaa}\in\Pi^{(k)}_{\rho_i}}}\prod_{j=1}^k\frac{\Upsilon_U[\nu_j]}{S(\nu_j)}\right)\right]
\end{align*}
\begin{align*}
&\hspace{4cm}\left.\sum_{k=1}^{|\rho_m|_0}\frac{a^{(k)}}{k!}\left(\sum_{\substack{(\nu_1,\hdots,\nu_k) \\ \hphantom{aaaa}\in\CP^{(k)}_{\rho_m}}}\sum_{n=1}^k\frac{\partial_{v}\Upsilon_U[\nu_n]}{S(\nu_n)}\prod_{\substack{j=1 \\ j\neq n}}^k\frac{\Upsilon_U[\nu_j]}{S(\nu_j)}\right)\right)\\
&=\frac{\Upsilon_{\hat{F}}[\sigma_n]}{S\left(\sigma_n\right)}\sum_{m=1}^{\ell}\prod_{\substack{i=1 \\ i\neq m}}^\ell\left[\sum_{k=1}^{|\rho_i|_0}\frac{a^{(k)}}{k!}\left(\sum_{\substack{(\nu_1,\hdots,\nu_k) \\ \hphantom{aaaa}\in\CP^{(k)}_{\rho_i}}}\prod_{j=1}^k\frac{\Upsilon_U[\nu_j]}{S(\nu_j)}\right)\right] \\
&\hspace{4cm}\left.\sum_{k=1}^{|\rho_m|_0}\frac{a^{(k)}}{k!}\left(\sum_{\substack{(\nu_1,\hdots,\nu_k) \\ \hphantom{aaaa}\in\CP^{(k)}_{\rho_m}}}\sum_{n=1}^k\frac{\Upsilon_U[\tau\graft{}\nu_n]}{S(\nu_n)}\prod_{\substack{j=1 \\ j\neq n}}^k\frac{\Upsilon_U[\nu_j]}{S(\nu_j)}\right)\right)
\end{align*}
where we repeat the previous argument. With this we have exhausted all the trees in $\Upsilon_{U}\left[\tau\graft{}\sigma\right]$, completing the argument. We now proceed with the induction for other possible grafts. The $\alpha=(1,0)$ turns out to be simpler than $\alpha=(0,0)$ because as it only hits $v_c$ and $q$, one computes:
\begin{align*}
&\frac{\partial_{v_c}\Upsilon_{U}(\sigma)}{S(\sigma)} = \frac{a^{(1)}\Upsilon_{U}(\sigma)}{q^2S(\sigma)} + \frac{1}{q}\left[\left(\sum_{k=1}^{|\sigma|_0}\frac{a^{(k)}}{k!}\left(\sum_{\substack{(\nu_1,\hdots,\nu_k) \\ \hphantom{aaaa}\in\CP^{(k)}_{\sigma}}}\prod_{j=1}^k\frac{\Upsilon_U[\nu_j]}{S(\nu_j)}\right)\right)\right. \\
&+v_c\left(\sum_{k=1}^{|\sigma|_0}\frac{a^{(k)}}{k!}\left(\sum_{\substack{(\nu_1,\hdots,\nu_k) \\ \hphantom{aaaa}\in\CP^{(k)}_{\sigma}}}\sum_{n=1}^k\frac{\partial_{v_c}\Upsilon_{U}[\nu_n]}{S(\nu_n)}\prod_{\substack{j=1 \\ j\neq n}}^k\frac{\Upsilon_U[\nu_j]}{S(\nu_j)}\right)\right) \\
&+ \sum_{\ell=2}^{|\sigma|_0}\frac{1}{\ell!}\sum_{\substack{(\rho_1,\hdots,\rho_\ell) \\ \hphantom{aaaa}\in\CP^{(\ell)}_\sigma}}\left[v_{(\ell,0)}\sum_{m=1}^\ell\prod_{\substack{i=1 \\ i\neq m}}^{\ell}\left(\sum_{k=1}^{|\rho_i|_0}\frac{a^{(k)}}{k!}\left(\sum_{\substack{(\nu_1,\hdots,\nu_k) \\ \hphantom{aaaa}\in\CP^{(k)}_{\rho_i}}}\prod_{j=1}^k\frac{\Upsilon_U[\nu_j]}{S(\nu_j)}\right)\right)\right.\\
&\left.\left(\sum_{k=1}^{|\rho_i|_0}\frac{a^{(k)}}{k!}\left(\sum_{\substack{(\nu_1,\hdots,\nu_k) \\ \hphantom{aaaa}\in\CP^{(k)}_{\rho_m}}}\sum_{n=1}^k\frac{\partial_{v_c}\Upsilon_{U}[\nu_n]}{S(\nu_n)}\prod_{\substack{j=1 \\ j\neq n}}^k\frac{\Upsilon_U[\nu_j]}{S(\nu_j)}\right)\right)\right]\\
&+\sum_{\ell=1}^{|\sigma|_0}\frac{1}{\ell!}\hspace{-5mm}\sum_{\substack{(\rho_1,\hdots,\rho_{\ell}, \\ \hphantom{aaaa}\CI_{\alpha_n}(\sigma_n))\in\CP^{(\ell+1)}_\sigma}}\left[\frac{\partial_{v_c}\Upsilon_{\hat{F}}[\sigma_n]}{S\left(\sigma_n\right)}\prod_{i=1}^{\ell-1}\left(\sum_{k=1}^{|\rho_i|_0}\frac{a^{(k)}}{k!}\left(\sum_{\substack{(\nu_1,\hdots,\nu_k) \\ \hphantom{aaaa}\in\CP^{(k)}_{\rho_i}}}\prod_{j=1}^k\frac{\Upsilon_U[\nu_j]}{S(\nu_j)}\right)\right)\right] \end{align*}
\begin{align*}
&+ \sum_{\ell=1}^{|\sigma|_0}\frac{1}{\ell!}\hspace{-5mm}\sum_{\substack{(\rho_1,\hdots,\rho_{\ell}, \\ \hphantom{aaaa}\CI_{\alpha_n}(\sigma_n))\in\CP^{(\ell+1)}_\sigma}}\left[\frac{\Upsilon_{\hat{F}}[\sigma_n]}{S\left(\sigma_n\right)}\sum_{m=1}^{\ell-1}\prod_{\substack{i=1 \\ i\neq m}}^{\ell-1}\left(\sum_{k=1}^{|\rho_i|_0}\frac{a^{(k)}}{k!}\left(\sum_{\substack{(\nu_1,\hdots,\nu_k) \\ \hphantom{aaaa}\in\CP^{(k)}_{\rho_i}}}\prod_{j=1}^k\frac{\Upsilon_U[\nu_j]}{S(\nu_j)}\right)\right)\right. \\
&\hspace{5cm}\left(\sum_{k=1}^{|\rho_m|_0}\frac{a^{(k)}}{k!}\left(\sum_{\substack{(\nu_1,\hdots,\nu_k) \\ \hphantom{aaaa}\in\CP^{(k)}_{\rho_m}}}\sum_{n}^{k}\frac{\Upsilon_U[\nu_n]}{S(\nu_n)}\prod_{\substack{j=1 \\ j\neq n}}^k\frac{\Upsilon_U[\nu_j]}{S(\nu_j)}\right)\right)
\end{align*}
The first term above is the only new term one needs to contend with.
\begin{equs}
    \Upsilon_{\hat F}(\tau)\,\partial_{v_c}\!\Upsilon_{U}(\sigma) &= \Upsilon_{\hat F}(\tau)\left(\frac{a^{(1)}\Upsilon_{U}(\sigma)}{q^2} + ...\right. \\
    &= \frac{a^{(1)}\Upsilon_{\U}(\tau)\Upsilon_{U}(\sigma)}{q} + ...
\end{equs}
This term is found in $\Upsilon_{U}(\tau_0)$. The arguments for the rest of the terms are in the same vein as for $\alpha=(0,0)$, and as such are omitted for brevity. The same argument can now be made for all $\alpha\in\mathbb{Z}_+^2\setminus\{(0,0),(1,0)\}$ and the same conclusion drawn.
\end{proof}
\begin{proposition}
The coefficient of $ \CI_{\alpha}(a(u), \hat F) $ in the expansion $ \mathcal{D}^{(0,m)} U  $ (resp. $ V_{\beta} $) is given by $ \partial_{v_{\alpha}} \partial^{m}_{x} u $
(resp. $ \partial_{v_{\alpha}} v_{\beta} $) defined in \ref{eq:derivatives}.\end{proposition}
\begin{proof}
Using \eqref{eq:defU} and the fact that planted trees are only seen when $\ell=0,\,1$, we get that:
\begin{equs}
\langle \mathcal{D}^{(0,m)} U , \CI_{\alpha} \hat F \rangle & =\left\langle \mathcal{D}^{(0,m)} \left[\sum_{|\ell|\le N}\frac{(\bar A(U))^{ \ell}}{\ell!}
(  \mathcal{K}_{(\ell,0)} \hat F)\right], \CI_{\alpha} \hat F \right\rangle  \\
& =  \langle
(  \mathcal{D}^{(0,m)} \mathcal{K} \hat F), \CI_{\alpha} \hat F \rangle + \langle \mathcal{D}^{(0,m)} (\bar A(U)\mathcal{K}_{(1,0)} \hat F), \CI_{\alpha} \hat F \rangle\\
&=\langle
(  \mathcal{D}^{(0,m)} \mathcal{K} \hat F), \CI_{\alpha} \hat F \rangle + \langle (\mathcal{D}^{(0,m)} \bar A(U))\mathcal{K}_{(1,0)} \hat F), \CI_{\alpha} \hat F \rangle
%& = a'(u) v_c \langle \mathcal{D}^{(0,m)} U , \CI_{\alpha} \hat F \rangle + \delta_{\alpha,(0,m)}.
\end{equs}
One has to be careful in the above computation because the term $$\langle ( \bar A(U))(  \mathcal{D}^{(0,m)}\mathcal{K}_{(1, 0)} \hat F), \CI_{\alpha} \hat F \rangle,$$ that one sees when one applies Leibniz rule to the second term in the second equality above is not in general non-zero. On the level of regularity of KPZ, it can be disregarded, however. To compute the two terms in the last equality above, we first recall that we have by definition $\mathcal{K} \hat F=\mathcal{I}\hat F+v_{(0,0)}\one$. An application of $\CD^{(0,m)}$ then gives us that:
\begin{equs}
\langle
(  \mathcal{D}^{(0,m)} \mathcal{K} \hat F), \CI_{\alpha} \hat F \rangle &= \langle
(  \CI_{(0,m)} \hat F), \CI_{\alpha} \hat F \rangle \\
&= \delta((0,m),\alpha)
\end{equs}
A similar computation as before gives:
\begin{equs}
\langle (\mathcal{D}^{(0,m)} \bar A(U))\mathcal{K}_{(1,0)} \hat F, \CI_{\alpha} \hat F \rangle = a'(u) v_c \langle \mathcal{D}^{(0,m)} U , \CI_{\alpha} \hat F \rangle
\end{equs}
Inserting these into the first equation and rearranging yields:
\begin{equs}\label{eq:U coeff}
\langle \mathcal{D}^{(0,m)} U &, \CI_{\alpha}\hat F \rangle = \delta_{\alpha,(0,m)}  \frac{1}{q}\,,
\end{equs}
as claimed. For $ V_{\beta} $, we get:
\begin{equs}
\langle V_{\beta} , \CI_{\alpha} \hat F \rangle  &=\left\langle  \sum_{|\ell|\le N}\frac{(\bar A(U))^{ \ell}}{\ell!}
(  \mathcal{K}_{\beta+ (\ell,0)} \hat F), \CI_{\alpha} \hat F \right\rangle  \\
& = \langle
   \mathcal{K}_\beta \hat F, \CI_{\alpha} \hat F \rangle + \langle ( \bar A(U))
(  \mathcal{K}_{\beta+(1,0)} \hat F), \CI_{\alpha} \hat F\rangle  \\
& = \delta_{\alpha,\beta} + a'(u) v_{\beta+(1,0)} \langle  U , \CI_{\alpha} \hat F \rangle.
\end{equs}
Substituting in \eqref{eq:U coeff} with $m=0$, we get \eqref{eq:derivatives}.
\end{proof}
\section{Renormalised equation}

\label{sec::3}

\subsection{Admissible model via preparation maps}

The next component we introduce is that of a preparation map $R$. Such maps are fundamental integrants in the construction of renormalisation maps, and as such their utility here, should not come as a surprise. These were first introduced in \cite{BR18}.
\begin{definition}
	\label{def:PrepMap}
	A \textbf{preparation map} is a map $
	R : \CT \rightarrow \CT $ 
	that fixes polynomials, noises, planted trees, and such that 
	\begin{itemize}
		\item for each $ \tau \in \mathcal{T} $ there exist finitely many $\tau_i \in \mathcal{T}$ and  $\lambda_i \in \mathbb{R}$ on such that
		\begin{equation} \label{EqAnalytical}
			R(\tau) = \tau + \sum_i \lambda_i \tau_i, \quad\textrm{with}\quad \deg(\tau_i) \geq \deg(\tau) \quad\textrm{and}\quad |\tau_i|_{\Xi} < |\tau|_{\Xi},
		\end{equation} 
		where $|\tau|_{\Xi}$ refers to the number of the noises in $\tau$.
		\item one has
		\begin{equation} \label{EqCommutationRDelta}
			( R  \otimes \Id) \Delta = \Delta R.
		\end{equation}
		where $\Delta$ is the coaction defined in \cite[Eq. 8.8a \& 8.8b]{reg}.
	\end{itemize}
\end{definition}
It is known that the adjoint $R^{\star}$ of $R$ with respect to \eqref{eq:innerproduct} satisfies the identity
\begin{equs}\label{eq:identity1}
	R^{\star}(\sigma\star\tau) = \sigma\star(R^{\star}\tau),
\end{equs}
for $\sigma\in T^{+},\,\tau\in T$. Imposing this condition more generally leads to a desirable strengthening of the preparation map:
\begin{definition}\label{def:strongprepmap}
	A \textbf{strong preparation map} is a preparation map that satisfies \eqref{eq:identity1} for all $\sigma\in T,\,\tau\in T$.
\end{definition}

\begin{example} Denote by $\CB_{-}$, the canonical basis of $T_-$ and let $\ell:T_-\mapsto\mbbR$ be a character on the linear span of the forests of $T_-$. Then
	\begin{equs}\label{eq:exprep}
		R^{\star}_\ell(\tau)\coloneqq\sum_{\sigma\in\CB^{-}}\frac{\ell(\sigma)}{S(\sigma)}(\tau\star\sigma)
	\end{equs}
	is a strong preparation map. 
\end{example}
In the sequel, we consider a strong preparation map $ R_{\ell}^{\star} $ defined on $ \hat{\CT} $ where we use the product $ \star $ on $\hat{\CT}$ and now the subset of trees $ \CB^- $ is replaced by  $ \hat{\CB}^- $ by adding $c$ derivatives. One assumption on the character $ \ell(\sigma) $ is to be compatible with the $c$-derivatives in that it satisfies the property \eqref{c_derivative}. Morever, following \cite{BGHZ}, one has just to consider $ \ell $ such that it is supproted on $ \SS_{\<generic>} $ or on $ \hat{\SS}_{\<generic>} $ when working with parametrised trees. 

Let us fix a nonnegative symmetric (under the involution $x \mapsto -x$) 
smooth function $\rho$ supported in the unit ball and integrating to 1 and set $\rho_\varepsilon(t,x)=\eps^{-3}\rho(\eps^{-2} t,\eps^{-1}x)$, $\xi_\eps=\rho_\eps\ast\xi$.
Then for $\varepsilon > 0$, we denote by
\begin{equs}
\xi_{\varepsilon}\in C^{\infty}(\mathbb{R}\times\mbbT)
\end{equs}
the regularised version of the spacetime white noise $\xi$, and also set $\xi_1 = 1$. Now given a spacetime dependent strong preparation map $R$ and a parameter constant $c_0$ fixed, we define an admissible model with respect to $\Xi_i$, $i\in\{1,\xi\}$ by setting
\begin{equs}
\Pi_x^{R,c_0}\Xi_i = \hat\Pi_x^{R,c_0}\Xi_i = \xi_i,\qquad\mbox{for }i\in\{1,\xi\}
\end{equs}
and then we recursively extend $\hat{\Pi}_x^{R,c_0}$ multiplicatively
\begin{equs}
\left(\hat{\Pi}_x^{R,c_0}(\tau\bar\tau)\right)(y)=\left(\hat{\Pi}_x^{R,c_0}\tau\right)(z)\left(\hat{\Pi}_x^{R,c_0}\bar\tau\right)(y)
\end{equs} and under the operation $\CI_\alpha$ for $\alpha=(\hat a,\bar a)\in\mbbN\times\mathbb{N}^2$ by requiring:
\begin{equs}
\left(\hat\Pi_x^{R,c_0}\CI_\alpha(\tau)\right)(y) = \bigl(\d_c^{\hat\alpha}&\d^{\bar\alpha}K(c_0,\cdot)\ast\Pi_x^{R,c_0}\tau\bigr)(y) \\
&-\sum_{|k|_{\s}\le\gamma+\beta}\frac{(y-x)^{k}}{k!}\left(\d_c^{\hat\alpha} \d^{\bar \alpha}  K(c_0,\cdot)\ast\Pi_x^{R,c_0}\right)(x)
\end{equs}
and then finally
\begin{equs}
\left(\Pi^{R,c_0}_{x}\tau\right)(y) = \left(\hat\Pi^{R,c_0}_{x}\left(R \tau\right)\right)(y).
\end{equs}
In the sequel, we will use the short-hand notation $ \PPi^R $ or $ \PPi $ for denoting a model.
\subsection{Main results} 
In the build-up to our main result, let us introduce some notation and conventions. We fix some $\delta\in(0,1/2)$ and set $\Co^\delta_\star$ to be the set of parabolic H\"older-$\delta$ functions with possibly finite (but strictly positive) blow-up time, for a precise definition see \cite[Sec.~2]{BGHZ}. With these, we are able to state one of our main results:

\begin{theorem} \label{fisrt renormalisation}
The renormalised term of \eqref{eq:transformed} is given by:
\begin{equs}
\sum_{\tau \in \hat{\SS}_{\<generic>} } C_{\eps}^{a(u_\eps)}(\tau) \frac{\Upsilon_{\hat F}[\tau]}{qS(\tau)}(u_{\eps})
\end{equs}
and then the renormalised equation of \eqref{eq:quasi KPZ} is given by:
\begin{equs}[eq:renorm nonlocal]
\partial_t u_{\eps} - a(u_{\eps}) \partial_x^{2} u_{\eps} & = f(u_{\eps}) (\partial_x u_{\eps})^{2} + k(u_{\eps}) \partial_x u_{\eps} + h(u_{\eps}) + g(u_{\eps}) \xi_{\eps} \\ & + \sum_{\tau \in \hat{\SS}_{\<generic>} } C_{\eps}^{a(u_\eps)}(\tau) \frac{\Upsilon_{\hat F}[\tau](u_{\eps})}{qS(\tau)}\,.
\end{equs}
More precisely, the local solutions $u_\eps$ of \eqref{eq:renorm nonlocal} on  $\mathbb{T}$ endowed with an initial condition $u(0,\cdot)=\varphi\in\Co^{2\delta}(\mathbb{T})$, converge in probability in $\Co^\delta_\star$ to a nontrivial limit $u$.
\end{theorem}
\begin{proof}
We follow here very closely \cite[Th. 10]{BB21b}, which can be traced back to \cite[Th. 9]{BB21a}. The key concept there is the notion of coherence that comes from \cite{BCCH} which is analogous to Theorem~\ref{prop:coherence} in our setup. Recall that:
$$
\FF\coloneqq\sum_{\tau}\frac{\Upsilon_{\hat F}[\tau]}{S(\tau)}\tau\text{,       and       }{\sf R}f(x) = \left(\hat\Pi^{R,c_0}_{x}\left(R f\right)\right)(x),
$$
for $f$ a modelled distribution, $\CQ_{\gamma-2}$  a projection from $T$ onto the subset thereof that consists of trees of degree less than $\gamma-2$ and finally $\gamma = 3/2$. The function ${\sf v}_i$ is a modelled distribution of regularity $\gamma$ and explosion index $\eta$ (refer to \cite{reg}). One has by construction and proposition~\eqref{prop:coherence}:
$$
\langle R\mathsf{v},\tau\rangle = \langle \mathsf{v},\,R^*\tau\rangle = \Upsilon_{\hat F}[R^*\tau],
$$
which is ultimately just the right derivation property of the preparation map (refer to \cite[Proposition 3.3]{BB21b}). Recall that the non-linearity $\hat{F}$ is such that in the lift $\FF$ one only sees trees of the form: $\tau = \Xi_\mfl\prod_{i=1}^n\CI_{\alpha_i}(\tau_i)$. Using the definition of $\Upsilon_{\hat{F}}$ for this particular form of $\tau$, one can rewrite the equality
$$
R {\FF} = \sum_{\textrm{deg}(\tau)<\gamma-2} \frac{\Upsilon_{\hat F}[R^*\tau]}{S(\tau)}\,\tau
$$ 
as:
\begin{equation*}
\begin{aligned}
R {\FF} &= \sum_{\substack{\mfl\in\{\one,\xi\} }} \sum_{\alpha_1,\hdots,\alpha_n} \sum_{\tau_1,...,\tau_n} \left\{\frac{\prod_{i=1}^n S(\tau_i)}{S\big(\Xi_{\mfl}\prod_{i=1}^n\CI_{a_i}(\tau_i)\big)} \prod_{i=1}^n  \frac{\Upsilon_{\hat{F}}(\tau_i)}{S(\tau_i)}\, \CI_{\alpha_i}(\tau_i)\right.\\ &\left.\hspace{6.65cm}\prod_{i=1}^n \partial_{v_{a_i}}\Upsilon_{\hat F}(R^*\Xi_{\mfl}) \Xi_{\mfl}\right\},
\end{aligned}
\end{equation*}
By grouping together  $\alpha_i$'s and $ \tau_i $'s so that $ \CI_{a_i}(\tau_i) \neq \CI_{a_j}(\tau_j) $ for $ i \neq j $ and rearranging the index, we can rewrite $\tau$ as $\prod_{i=1}^n\CI_{\alpha_i}^{\beta_i}(\tau_i)$ with also the assumption \eqref{edges_not_commuting}, which in turn allows us to rewrite the equation above as:
\begin{equs}
R{\FF}= \sum_{\substack{\mfl\in\{\one,\xi\}}} \sum_{a_1,...,a_n} \sum_{\tau_1,...,\tau_n} \prod_{i=1}^n \frac{1}{\beta_i !} \left( \frac{\Upsilon_{\hat F}(\tau_i)}{S(\tau_i)}\,\CI_{a_i}(\tau_i) \right)^{\beta_i}\prod_{i=1}^n (\partial_{v_{\alpha_i}})^{\beta_{i}} \Upsilon_{\hat F} (R^*\Xi_{\mfl}) \Xi_{\mfl}.
\end{equs}
Here $\beta_i!$ comes from the fact that ratio of  $S(\tau_i)^{\beta_i}$ and $S(\CI^{\beta_i}_{\alpha_1}(\tau_i))$ is $1/\beta_i!$, as we have seen in the proof of proposition~\ref{prop:coherence}. Now, applying Fa\`a di Bruno formula as in \cite[Lemma A.1]{BCCH}, one gets:
%\begin{equation}
%\frac{\partial^{k}G}{k!} =  \sum_{b_1,...,b_m} \sum_{k = \sum_{j=1}^m \beta_j k_j} \prod_{j=1}^m \frac{1}{\beta_j !} \left(\frac{Z_{b_j + k_j}}{k_j!}\right)^{\beta_j} \prod_{j=1}^m (D_{b_j})^{\beta_j} G 
%\end{equation}
\begin{equs}
R{\FF} = \sum_{\substack{\mfl\in\{\one,\xi\} }} \sum_{\alpha_1,...,\alpha_n} \sum_{\beta_1,...,\beta_n } \prod_{\alpha}\prod_{i=1}^n \frac{1}{\beta_i !} \left(\left[V_\alpha\right]_{\alpha_i} - \left[v_{\alpha}\right]_{\alpha_i} \right)^{\beta_j} \prod_{i=1}^n (\partial_{v_{\alpha_i}})^{\beta_{i}} \Upsilon_{\hat F} (R^*\Xi_{\mfl}) \Xi_{\mfl},
\end{equs}
where $\left[V_\alpha\right]_{\alpha_i}\coloneqq D_{\alpha_i}V_{\alpha}$, and $\left[v_{\alpha}\right]_{\alpha_i}\coloneqq \inprod{V_{\alpha_i}}{\one}$. One can see in the summation above, the definition of the lift of the smooth function $\Upsilon_{\hat F}[R^*\Xi_\mfl]$
\begin{equs} \label{main_identity}
R\FF= \sum_{\mfl\in\{\one,\xi\}}\FF_{R^*\Xi_\mfl}\Xi_{\mfl},
\end{equs}
where as usual $\FF_{R^\star\Xi_\mfl} \coloneqq \inprod{\FF}{R^*\Xi_\mfl}$. Returning to the equation for a fixed spacetime point $x$, one gets:
 
\begin{equation*} \begin{split}
\big((\partial_t-a(u)\partial_{x}^2)u\big)(x) &= \big({\sf R}{\FF}\big)(x)  = \hat\Pi^{R,c_0}_{x}\big(R{\FF}(x)\big)(x)   \\
&= \sum_{\mfl\in\{\one,\xi\}} \hat\Pi^{R,c_0}_{x}\Big(\FF_{R^{*}\Xi_\mfl}\left(\mathbf{V}(x)\right)\Xi_\mfl\Big)(x)   \\
&= \sum_{\mfl\in\{\one,\xi\}} \FF_{R^*\Xi_\mfl}\left(\hat\Pi^{R,c_0}_{x}\mathbf{V}(x)\right)\hat\Pi^{R,c_0}_{x}\Xi_\mfl   \\
&=  \FF_{R^*\Xi_\one}(\mathbf{v}(x)) + \FF_{R^*\Xi_\xi}(\mathbf{v}(x))\xi,
\end{split}\end{equation*} 
where $\hat\Pi^{R,c_0}_x\mathbf{V}(x)$ codifies applying $\hat\Pi^{R,c_0}_x$ to the entries of $\mathbf{V}(x)$ which of course gives $\mathbf{v}$. This reveals how the renormalisation is affected at the level of the noises. In particular, we see that as the preparation map leaves $\Xi_\xi$ invariant, the renormalisation occurs in the term $\FF_{R^* \Xi_1}$. From here we need to make an appropriate choice of preparation maps and the BPHZ form \eqref{eq:exprep} works for us, yielding:
\begin{equs}
\sum_{\tau \in \hat{\SS}_{\<generic>} } C_{\eps}^{a(u_\eps)}(\tau) \frac{\Upsilon_{\hat F}[\tau]}{S(\tau)}(u_{\eps})
\end{equs}
Finally notice that $\hat{F} = qF + ...$, which explains the presence of the $1/q$ factor when one backtracks from $\hat{F}$ to $F$.
\end{proof}
\begin{remark}
The above result cannot preclude the possibility of some nonlocal term appearing on the right-hand side of the equation. We need some constraints on the $ C_{\eps}(\tau) $ in order to make these terms disappear.
\end{remark}

We introduce now the so-called covariant derivatives that are known from \cite{BGHZ} to generate all the trees in our analysis:
\begin{equs}
	\Nabla_{\tau_2} \tau_1 &=  a
	\begin{tikzpicture}[scale=0.2,baseline=2]
		\draw[symbols]  (-.5,2.5) -- (0,0) ;
		\draw[tinydots] (0,0)  -- (0,-1.3);
		\node[var] (root) at (0,-0.1) {\tiny{$ \tau_1 $ }};
		\node[var] (diff) at (-0.5,2.5) {\tiny{$ \tau_2 $ }};
		\node[blank] at (0.3,1.25) {\tiny{$0$}};
	\end{tikzpicture} + \frac{1}{2}\;
	\begin{tikzpicture}[scale=0.2,baseline=2]
		\coordinate (root) at (0,0);
		\coordinate (t1) at (-1,2);
		\coordinate (t2) at (1,2);
		\draw[tinydots] (root)  -- +(0,-0.8);
		\draw[kernels2] (t1) -- (root);
		\draw[kernels2] (t2) -- (root);
		\node[not] (rootnode) at (root) {};
		\node[var] (t1) at (t1) {\tiny{$ \tau_1 $}};
		\node[var] (t1) at (t2) {\tiny{$ \tau_2 $}};
		\node[blank] at (-1.2,0.6) {\tiny{$0$}};
		\node[blank] at (1.2,0.6) {\tiny{$0$}};
	\end{tikzpicture} \\
	\Nabla_{\tau_2}^1 \tau_1 &=  \partial(a\cdot)
	\begin{tikzpicture}[scale=0.2,baseline=2]
		\draw[symbols]  (-.5,2.5) -- (0,0) ;
		\draw[tinydots] (0,0)  -- (0,-1.3);
		\node[var] (root) at (0,-0.1) {\tiny{$ \tau_1 $ }};
		\node[var] (diff) at (-0.5,2.5) {\tiny{$ \tau_2 $ }};
		\node[blank] at (0.3,1.25) {\tiny{$1$}};
	\end{tikzpicture} + \frac{1}{2}\; (\partial \cdot )  \left( 
	\begin{tikzpicture}[scale=0.2,baseline=2]
		\coordinate (root) at (0,0);
		\coordinate (t1) at (-1,2);
		\coordinate (t2) at (1,2);
		\draw[tinydots] (root)  -- +(0,-0.8);
		\draw[kernels2] (t1) -- (root);
		\draw[kernels2] (t2) -- (root);
		\node[not] (rootnode) at (root) {};
		\node[var] (t1) at (t1) {\tiny{$ \tau_1 $}};
		\node[var] (t1) at (t2) {\tiny{$ \tau_2 $}};
		\node[blank] at (-1.2,0.6) {\tiny{$1$}};
		\node[blank] at (1.2,0.6) {\tiny{$0$}};
	\end{tikzpicture}  + 
	\begin{tikzpicture}[scale=0.2,baseline=2]
		\coordinate (root) at (0,0);
		\coordinate (t1) at (-1,2);
		\coordinate (t2) at (1,2);
		\draw[tinydots] (root)  -- +(0,-0.8);
		\draw[kernels2] (t1) -- (root);
		\draw[kernels2] (t2) -- (root);
		\node[not] (rootnode) at (root) {};
		\node[var] (t1) at (t1) {\tiny{$ \tau_1 $}};
		\node[var] (t1) at (t2) {\tiny{$ \tau_2 $}};
		\node[blank] at (-1.2,0.6) {\tiny{$0$}};
		\node[blank] at (1.2,0.6) {\tiny{$1$}};
	\end{tikzpicture} \right)
\end{equs}
The second covariant derivative could be seen as the derivative of the first one. Here we have used short-hand notation. In fact for a renormalisation constant
\begin{equs} \label{c_derivative}
\hat{C}^a_{\varepsilon} = 	 C_\eps(c_1 \cdots , c_{[\tau_1]}, \bar{c}_1, \bar{c}_2 , \tilde{c}_1,\cdots , \tilde{c}_{[\tau_2]})|_{c_i= \bar{c}_j = \tilde{c}_{\ell} = a}.
\end{equs}
One has:
\begin{equs}
	\partial(a \hat{C}^a_{\varepsilon}) = \hat{C}^a_{\varepsilon} + \left(   (\partial_{\bar{c}_1} + \partial_{ \bar{c}_2})C_\eps(c_1 \cdots , c_{[\tau_1]}, \bar{c}_1, \bar{c}_2 , \tilde{c}_1,\cdots , \tilde{c}_{[\tau_2]}) \right)|_{c_i= \bar{c}_j = \tilde{c}_{\ell} = a}.
\end{equs} 
More generally, one can set
\begin{equs} \label{covariant_derivative}
		\Nabla_{\tau_2}^m \tau_1 &=  \partial^m(a\cdot)
	\begin{tikzpicture}[scale=0.2,baseline=2]
		\draw[symbols]  (-.5,2.5) -- (0,0) ;
		\draw[tinydots] (0,0)  -- (0,-1.3);
		\node[var] (root) at (0,-0.1) {\tiny{$ \tau_1 $ }};
		\node[var] (diff) at (-0.5,2.5) {\tiny{$ \tau_2 $ }};
		\node[blank] at (0.45,1.25) {\tiny{$m$}};
	\end{tikzpicture} + \frac{1}{2}\; (\partial^m \cdot )  \left(  \sum_{k+\ell = m} \frac{1}{\ell!k!} 
	\begin{tikzpicture}[scale=0.2,baseline=2]
		\coordinate (root) at (0,0);
		\coordinate (t1) at (-1,2);
		\coordinate (t2) at (1,2);
		\draw[tinydots] (root)  -- +(0,-0.8);
		\draw[kernels2] (t1) -- (root);
		\draw[kernels2] (t2) -- (root);
		\node[not] (rootnode) at (root) {};
		\node[var] (t1) at (t1) {\tiny{$ \tau_1 $}};
		\node[var] (t1) at (t2) {\tiny{$ \tau_2 $}};
		\node[blank] at (-1.3,0.6) {\tiny{$k$}};
		\node[blank] at (1.3,0.6) {\tiny{$\ell$}};
	\end{tikzpicture}
	\right).
\end{equs}
In general, in expansion of the solution, the number of parameter derivative is bounded, although in the interest of brevity, we will write sums over the number of parameter derivative  $ \sum_m $ without indicating the bound.
\begin{assumption}\label{chain rule}
	We assume that the renormalisation constants $ C_{\varepsilon}^{a(u_\eps)}(\tau) $ have been chosen such that they satisfy the chain rule in the sense that the covariant derivatives generate the counter-terms. It means that the term $ \sum_{\tau \in \hat{\SS}_{\<generic>} } C_{\eps}^{a(u_\eps)}(\tau) \frac{\Upsilon_{\hat F}[\tau]}{S(\tau)}(u_{\eps}) $ can be written as a linear combination of the form $ \Upsilon_{\hat F}[ \sum_m \Nabla_{\tau_2}^m \tau_1] $ with $ \tau_1, \tau_2 $ also satisfying the chain rule. 
\end{assumption}

\begin{remark}
	In fact, the terms of the above assumption are generated by the $ \Nabla_{}^m $ and the noise $ \<generic> $. In the context of space-time white noise, they contain at most four covariant derivatives and four noises $ \<generic> $.
	\end{remark}

\begin{remark}
	The previous assumption is very natural. Indeed, when the sum boils down to just $m=0 $, one recovers the chain rule property stated in \cite{BGHZ} which is equivalent to the integration by parts identities used in \cite{Mate19}.  The $m\neq 0$ case corresponds to taking parameter derivatives of the identities obtained in \cite{BGHZ}. It is shown in  \cite[Prop. 3.9]{BGHZ} that if one performs the BPHZ renormalisation and looks at the projection onto the orthogonal of the geometric terms, the renormalisation constants converge to a limit. If one takes parameter derivatives of them, one also gets a finite limit.
 This corresponds to the parameter derivation of the integration by parts formulae found in \cite[Lem. 2.4]{Mate19}.
\end{remark}
 
The next theorem is the main result of this paper.

\begin{theorem}\label{main result 2}
The renormalised equation of \eqref{eq:quasi KPZ} is given by:
\begin{equs}[eq:renorm local]
\partial_t u_{\eps} - a(u_{\eps}) \partial_x^{2} u_{\eps} & = f(u_{\eps}) (\partial_x u_{\eps})^{2} + k(u_{\eps}) \partial_x u_{\eps} + h(u_{\eps}) + g(u_{\eps}) \xi_{\eps} \\ & + \sum_{\tau \in \SS_{\<generic>} } C_{\eps}^{a(u_\eps)}(\tau) \frac{\Upsilon_{F}[\tau](u_{\eps})}{S(\tau)}
\end{equs}
where the $ C^{a(u_\eps)}_{\eps}(\tau) $ satisfy the Assumption~\ref{chain rule} on the chain rule.
More precisely, the solutions $u_\eps$ of \eqref{eq:renorm local} converge in the same sense as in Theorem \ref{fisrt renormalisation}.
\end{theorem}
\begin{proof} From Theorem \ref{fisrt renormalisation}, one gets the following counter-terms:
	\begin{equs}
		\sum_{\tau \in \hat{\SS}_{\<generic>} } C_{\eps}^{a(u_\eps)}(\tau) \frac{\Upsilon_{\hat F}[\tau](u_{\eps})}{q S(\tau)}\,.
	\end{equs}
According to the Assumption \ref{chain rule}, it can be rewritten as a linear combination of terms of the form
$ \sum_{m} \Upsilon_{\hat F} \left[ \Nabla_{\tau_2}^m \tau_1 \right]$. In the previous sum, the terms for $ m \geq 2 $ are equal to zero due Proposition \ref{prop:nullprop}, and the first two terms ($m = 0,\,1$) as we will see in Theorem \ref{th:covloc}, satisfy:
\begin{equs}
	\Upsilon_{\hat F} \left[ \Nabla_{\tau_2} \tau_1 \right] + \Upsilon_{\hat F} \left[ \Nabla_{\tau_2}^1 \tau_1 \right]= q\Upsilon_{F}\brsq{ \mathcal{P}\nabla_{\tau_2}\tau_1}
	\end{equs}
where $ \mathcal{P} $ is the projection from parametrised to unparametrised trees.
So in the end, one obtains a counter-term 
\begin{equs}
\sum_{\tau \in \SS_{\<generic>} } C_{\eps}^{a(u_\eps)}(\tau) \frac{\Upsilon_{F}[\tau](u_{\eps})}{S(\tau)}
\end{equs}
satisfying Assumption \ref{chain rule} without any paremeter derivative ($ m=0 $).
\end{proof}

\begin{remark}
One can check as an exercise that the renormalised equations given in \cite{MH,Mate19} are covered by this theorem. As in \cite{OW-div}, we see also that there is not a need for renormalisation when the equation is in divergence form
and the right-hand side only contains the noise.
\end{remark}

\begin{remark}
This result has the rather interesting feature that the chain rule transforms a non-commutative structure into a commutative one. From Theorem~\ref{fisrt renormalisation} to Theorem~\ref{main result 2}, we remove the augmented regularity structure and its planar trees. Indeed, the renormalisation looks the same as in the case where $ a $ is constant.
\end{remark}

\begin{remark}
The dependence of $C_\eps^c(\tau)$ on $c$ has no explicit form in our setting. However, if one instead considers a higher dimensional variant of \eqref{eq:quasi KPZ} with noise that is white in only space, then the kernels that play a role in the regularity structure are the Green's functions $\bar P(c,\cdot)$ of the operators $-c\Delta$.
Their form is of course even simpler than \eqref{eq:kernel scaling}: $\bar P(c,\cdot)=c^{-1}\bar P(1,\cdot)$, and therefore $C_\eps^c(\tau)=c^{-[\tau]}C_\eps^1(\tau)$.
\end{remark}

\subsection{Chain rule: A warm up example}\label{sec:computation}

In the sections preceding this, we have performed a transformation which introduces $v_\alpha$ into \eqref{eq:transformed} and solved the resulting system of equations in the Regularity Structures paradigm before transforming the solution back. Naturally one expects to see counter-terms in both $U$ and $V$ when solving the transformed equation, which can be potentially non-local functions of the solution $u$. It is imperative therefore that the counter-terms in $V$ vanish when transforming back for this line of reasoning to be well-founded, as we have seen in Theorem~\ref{main result 2}. To ground this ``magical disappearance'' we illustrate this process on the trees: $ \<Xi2> \,, \<I1Xitwo>  $. Indeed, due to the recursive definitions \eqref{eq:upsilonF} and \eqref{eq:upsilonhatF}, we have for $ \<Xi2>  $:
\begin{equs}\label{eq:eq1}
\Upsilon_{F}[\<Xi2>] & = \Upsilon_{F}[\<generic>]\, \partial_u \Upsilon_{F}[\<generic>] = g g', \,\,\, \Upsilon_{\hat F}[\hspace{-1mm}\<pXi2>]  = \Upsilon_{\hat F}[\<generic>]  \, \partial_{v_{c}} \Upsilon_{\hat F}[\<generic>] = - q a' g^2, \\
\Upsilon_{\hat F}[\<Xi2>] & = \Upsilon_{\hat F}[\<generic>]  \, \partial_{v} \Upsilon_{\hat F}[\<generic>] = \hat g \partial_v \hat g = q g g'  +  (- a'' v_c - (a')^2 v_{cc})  g^2.
\end{equs} 
Then for $ \<I1Xitwo>  $, we get
\begin{equs}
 \frac{1}{2} \Upsilon_{F}[\<I1Xitwo>]&  = \frac{1}{2} \Upsilon_{ F}[\<generic>]^{2} \partial_{\partial_x u}^{2} (f-a') (\partial_{x} u)^2 = (f-a') g^2, \\
\frac{1}{2} \Upsilon_{\hat F}[\<I1Xitwo>]
& = \frac{1}{2}\Upsilon_{\hat F}[\<generic>]^{2} \partial_{v_x}^{2} \hat F_{1}
\\ & = \frac{1}{q^2} \hat f \hat g^2 +  \frac{1}{q} a' \hat g^2 = q (f-a') g^2  + (a (a')^{2} v_{cc} + a a'' v_c) g^2 + q a' g^2 \\
\Upsilon_{\hat F}[\hspace{-1mm}\<I1Xitwou>\hspace{-1mm}] & = \Upsilon_{\hat F}[\<generic>]^{2} \partial_{v_x}\partial_{v_{cx}} \hat F_{1} =  2 q a a' g^2. 
\\ 
\Upsilon_{\hat F}[\hspace{-1mm}\<I1Xitwoub>\hspace{-1mm}] & = \Upsilon_{\hat F}[\<generic>]^{2}\partial_{v_{cx}} \partial_{v_x} \hat F_{1} =  2 q a a' g^2.
\end{equs}
These computations reveal the following relationship:
\begin{equs}\label{eq:eq1}
a \Upsilon_{\hat F}\brsq{\<Xi2>} + \frac{1}{2}\Upsilon_{\hat F}\brsq{\<I1Xitwo>}&+\partial (a \cdot ) \Upsilon_{\hat F}\brsq{\!\!\<pXi2>} \\ &+ (\partial \cdot)\frac{1}{2} \left( \Upsilon_{\hat F}\brsq{\!\<I1Xitwou>\!} + \Upsilon_{\hat F}\brsq{\!\<I1Xitwoub>\!} \right) = q a \Upsilon_{ F}\brsq{\<Xi2>} + q \frac{1}{2} \Upsilon_{F}\brsq{\<I1Xitwo>}
\end{equs}
Now using the easily verifiable facts that:
\begin{equs}
\Upsilon_{\hat F}\brsq{\hspace{-1mm}\<puXi2>}  = 0, \quad m > 1,  \quad
\Upsilon_{\hat F}\brsq{\hspace{-1mm}\<I1Xitwoup>\hspace{-1mm}}=0 , \quad k+ \ell > 1
\end{equs}
and $ S(\<Xi2>) = 1, S(\<I1Xitwo>) = 2 $, one has the following identity:
\begin{equs}
& \sum_{\tau \in \hat{\SS}_{\<generic>}^{(2)}}  C_{\eps}(\tau) \frac{\Upsilon_{\hat F}[ \tau]}{qS(\tau)}  = \frac{1}{q}\left(C_{\epsilon}(\<Xi2>)\Upsilon_{\hat F}\brsq{\<Xi2>}+ \frac{1}{2} C_{\epsilon}(\<I1Xitwo>  )\Upsilon_{\hat F}\brsq{\<I1Xitwo>}\right. \\
&\left. + 
 C_{\epsilon}(\!\!\<pXi2>) \Upsilon_{\hat F}\brsq{\!\!\<pXi2>} + \frac{1}{2}  C_{\eps}(\!\<I1Xitwou>\!)  \Upsilon_{\hat F}\brsq{\!\<I1Xitwou>\!} +  \frac{1}{2}  C_{\eps}(\!\<I1Xitwoub>\!)  \Upsilon_{\hat F}\brsq{\!\<I1Xitwoub>\!}\right)
\end{equs}
\noindent Wanting to leverage \eqref{eq:eq1}, we now adjust the preparation map to enforce
\begin{equs} \label{identity constants}
  C_{\varepsilon}(\<Xi2>) =  a  C_{\varepsilon}(\<I1Xitwo>),
\end{equs}
which yields to an application of $\partial$:
\begin{equs}
 C_{\varepsilon}(\!\!\<pXi2>) = C_{\varepsilon}(\<I1Xitwo>) + 2 a  \left( C_{\eps}(\!\<I1Xitwou>\!) + C_{\eps}(\!\<I1Xitwoub>\!)  \right).
\end{equs}
Therefore:
\begin{equs}
\sum_{\tau \in \SS_{\<generic>}^{(2)}}  \hspace{-2mm}C_{\eps}(\tau) \frac{\Upsilon_{\hat F}[\tau]}{qS(\tau)}  & =  \frac{1}{q}C_{\varepsilon}(\<I1Xitwo>) \Bigl( a \Upsilon_{\hat F}\brsq{\<Xi2>} + \frac{1}{2}\Upsilon_{\hat F}\brsq{\<I1Xitwo>} \\
&\hspace{2cm}+  \partial (a \cdot ) \Upsilon_{\hat F}\brsq{\!\!\<pXi2>} + (\partial \cdot)\frac{1}{2}\Upsilon_{\hat F}\brsq{\!\<I1Xitwou>\!} \Bigr) \\
& =  C_{\varepsilon}(\<I1Xitwo>) \left( a \Upsilon_{ F}\brsq{\<Xi2>} + \frac{1}{2} \Upsilon_{ F}\brsq{\<I1Xitwo>} \right).
\end{equs}
where we have used the fact that $ \Upsilon_{\hat F}(\hspace{-1mm}\<I1Xitwou>\hspace{-1mm}) =  \Upsilon_{\hat F}(\hspace{-1mm}\<I1Xitwoub>\hspace{-1mm}) $.
Hence we see the locality of the constants for these particular trees. It turns out that this is true more generally as we will see in the next section.

\subsection{Chain rule: The general case}
With the view of proving that the calculation in the previous section is not a peculiarity of that particular tree and is indeed true for all the trees we see for the generalised KPZ, we would like to introduce a notion of locality on $T$. Pursuant to this, we define a projection $\CP:\hat{T}\mapsto T$ that canonically sends every parameterised tree with zero parameter derivatives to its unparametrised tree, and kills everything else.
\begin{example}
Consider trees of the form $\begin{tikzpicture}[scale=0.2,baseline=2]
\draw[symbols]  (-.5,2.5) -- (0,0) ;
\draw[tinydots] (0,0)  -- (0,-1.3);
\node[var] (root) at (0,-0.1) {\tiny{$ \tau_1 $ }};
\node[var] (diff) at (-0.5,2.5) {\tiny{$ \tau_2 $ }};
\node[blank] at (0.3,1.25) {\tiny{$k$}};
\end{tikzpicture}$ for $k\in\mbbN$ and $\begin{tikzpicture}[scale=0.2,baseline=2]
\coordinate (root) at (0,0);
\coordinate (t1) at (-1,2);
\coordinate (t2) at (1,2);
\draw[tinydots] (root)  -- +(0,-0.8);
\draw[kernels2] (t1) -- (root);
\draw[kernels2] (t2) -- (root);
\node[not] (rootnode) at (root) {};
\node[var] (t1) at (t1) {\tiny{$ \tau_1 $}};
\node[var] (t1) at (t2) {\tiny{$ \tau_2 $}};
\node[blank] at (-1.2,0.6) {\tiny{$k$}};
\node[blank] at (1.2,0.6) {\tiny{$l$}};
\end{tikzpicture}$ for $k,\,l\in\mbbN$. Then one has
\begin{equs}
 \CP\left[\begin{tikzpicture}[scale=0.2,baseline=2]
\draw[symbols]  (-.5,2.5) -- (0,0) ;
\draw[tinydots] (0,0)  -- (0,-1.3);
\node[var] (root) at (0,-0.1) {\tiny{$ \tau_1 $ }};
\node[var] (diff) at (-0.5,2.5) {\tiny{$ \tau_2 $ }};
\node[blank] at (0.3,1.25) {\tiny{$k$}};
\end{tikzpicture}\right] = \begin{cases}
\begin{tikzpicture}[scale=0.2,baseline=2]
\draw[symbols]  (-.5,2.5) -- (0,0) ;
\draw[tinydots] (0,0)  -- (0,-1.3);
\node[var] (root) at (0,-0.1) {\tiny{$ \bar\tau_1 $ }};
\node[var] (diff) at (-0.5,2.5) {\tiny{$ \bar\tau_2 $ }};
\end{tikzpicture} & \mbox{for }k=0 \\
0 & \mbox{otherwise}
\end{cases}
 \end{equs} 
 \begin{equs}
\CP\left[\begin{tikzpicture}[scale=0.2,baseline=2]
\coordinate (root) at (0,0);
\coordinate (t1) at (-1,2);
\coordinate (t2) at (1,2);
\draw[tinydots] (root)  -- +(0,-0.8);
\draw[kernels2] (t1) -- (root);
\draw[kernels2] (t2) -- (root);
\node[not] (rootnode) at (root) {};
\node[var] (t1) at (t1) {\tiny{$ \tau_1 $}};
\node[var] (t1) at (t2) {\tiny{$ \tau_2 $}};
\node[blank] at (-1.2,0.6) {\tiny{$k$}};
\node[blank] at (1.2,0.6) {\tiny{$l$}};
\end{tikzpicture}\right] = \begin{cases}
\begin{tikzpicture}[scale=0.2,baseline=2]
\coordinate (root) at (0,0);
\coordinate (t1) at (-1,2);
\coordinate (t2) at (1,2);
\draw[tinydots] (root)  -- +(0,-0.8);
\draw[kernels2] (t1) -- (root);
\draw[kernels2] (t2) -- (root);
\node[not] (rootnode) at (root) {};
\node[var] (t1) at (t1) {\tiny{$ \bar\tau_1 $}};
\node[var] (t1) at (t2) {\tiny{$ \bar\tau_2 $}};
\end{tikzpicture} & \mbox{for }k=l=0 \\
0 & \mbox{otherwise}
\end{cases}
\end{equs}
\end{example}
With this definition at hand, we define the following subset of local trees in $\hat{T}$: 
\begin{equ}
\hat T^{\loc}=\{\tau\in\hat T:\,\Upsilon_{\hat F}(\tau)=q\Upsilon_{F}(\mcP\tau)\}.
\end{equ}
In particular, if a (paramterised) tree $\tau$ has at least one nonzero parameter derivative, the condition $\tau\in\hat T^{\loc}$ will imply that $\Upsilon_{\hat F}\brsq{\tau}$ must necessarily vanish because $\Upsilon_F\brsq{\CP\tau}$ will.
\begin{proposition}\label{prop:nullprop}
One has for $\tau_1,\tau_2\in\hat{T}^{\loc}$ that
$\begin{tikzpicture}[scale=0.2,baseline=2]
\draw[symbols]  (-.5,2.5) -- (0,0) ;
\draw[tinydots] (0,0)  -- (0,-1.3);
\node[var] (root) at (0,-0.1) {\tiny{$ \tau_1 $ }};
\node[var] (diff) at (-0.5,2.5) {\tiny{$ \tau_2 $ }};
\node[blank] at (0.3,1.25) {\tiny{$k$}};
\end{tikzpicture}\;\in\hat T^{\loc}$ for all $k>1$, and $\begin{tikzpicture}[scale=0.2,baseline=2]
\coordinate (root) at (0,0);
\coordinate (t1) at (-1,2);
\coordinate (t2) at (1,2);
\draw[tinydots] (root)  -- +(0,-0.8);
\draw[kernels2] (t1) -- (root);
\draw[kernels2] (t2) -- (root);
\node[not] (rootnode) at (root) {};
\node[var] (t1) at (t1) {\tiny{$ \tau_1 $}};
\node[var] (t1) at (t2) {\tiny{$ \tau_2 $}};
\node[blank] at (-1.2,0.6) {\tiny{$k$}};
\node[blank] at (1.2,0.6) {\tiny{$l$}};
\end{tikzpicture}$ for all $k + l > 1$.
\end{proposition}
\begin{proof}
Given the hypothesis we may assume $\Upsilon_{\hat F}(\tau_i) = \Upsilon_{F}(\CP\tau_i)$ for $i\in\{1,2\}$. Then one calculates
\begin{equs}
\Upsilon_{\hat{F}}\left[\begin{tikzpicture}[scale=0.2,baseline=2]
\draw[symbols]  (-.5,2.5) -- (0,0) ;
\draw[tinydots] (0,0)  -- (0,-1.3);
\node[var] (root) at (0,-0.1) {\tiny{$ \tau_1 $ }};
\node[var] (diff) at (-0.5,2.5) {\tiny{$ \tau_2 $ }};
\node[blank] at (0.3,1.25) {\tiny{$k$}};
\end{tikzpicture}\right] &= \Upsilon_{\hat{F}}[\tau_1]\partial_{v_{(k,0)}}\Upsilon_{\hat{F}}[\tau_2] \\
&= q\Upsilon_{{F}}[\CP\tau_1]\partial_{v_{(k,0)}}q\Upsilon_{{F}}[\CP\tau_2] \\
&= 0,
\end{equs}
where the last equality is due to that fact $\Upsilon_F$ and $q\,(\,=1-a'(u)v_c) $ is independent of $v_{(k,0)}$. On the other hand, one has
\begin{equs}
q\Upsilon_{{F}}\left[\CP\left[\begin{tikzpicture}[scale=0.2,baseline=2]
\draw[symbols]  (-.5,2.5) -- (0,0) ;
\draw[tinydots] (0,0)  -- (0,-1.3);
\node[var] (root) at (0,-0.1) {\tiny{$ \tau_1 $ }};
\node[var] (diff) at (-0.5,2.5) {\tiny{$ \tau_2 $ }};
\node[blank] at (0.3,1.25) {\tiny{$k$}};
\end{tikzpicture}\right]\right] = 0
\end{equs}
One can argue similarly that
\begin{equs}
\Upsilon_{\hat F}\brsq{\begin{tikzpicture}[scale=0.2,baseline=2]
\coordinate (root) at (0,0);
\coordinate (t1) at (-1,2);
\coordinate (t2) at (1,2);
\draw[tinydots] (root)  -- +(0,-0.8);
\draw[kernels2] (t1) -- (root);
\draw[kernels2] (t2) -- (root);
\node[not] (rootnode) at (root) {};
\node[var] (t1) at (t1) {\tiny{$ \tau_1 $}};
\node[var] (t1) at (t2) {\tiny{$ \tau_2 $}};
\node[blank] at (-1.2,0.6) {\tiny{$k$}};
\node[blank] at (1.2,0.6) {\tiny{$l$}};
\end{tikzpicture}} = 0 = q\Upsilon_F\brsq{\CP\left[\begin{tikzpicture}[scale=0.2,baseline=2]
\coordinate (root) at (0,0);
\coordinate (t1) at (-1,2);
\coordinate (t2) at (1,2);
\draw[tinydots] (root)  -- +(0,-0.8);
\draw[kernels2] (t1) -- (root);
\draw[kernels2] (t2) -- (root);
\node[not] (rootnode) at (root) {};
\node[var] (t1) at (t1) {\tiny{$ \tau_1 $}};
\node[var] (t1) at (t2) {\tiny{$ \tau_2 $}};
\node[blank] at (-1.2,0.6) {\tiny{$k$}};
\node[blank] at (1.2,0.6) {\tiny{$l$}};
\end{tikzpicture}\right]}
\end{equs}
\end{proof}

The crux of our argument now lies in the following theorem:
\begin{theorem}\label{th:covloc}
One has for all $\tau_1,\,\tau_2\in \hat{T}^{\text{loc}}$ that $\nabla_{\tau_2}\tau_1 + \nabla^{1}_{\tau_2}\tau_1 \in \hat{T}^{\text{loc}}$. 
\end{theorem}
\begin{proof}
Assuming that $\Upsilon_{\hat F}\brsq{\tau_i}=\Upsilon_{{F}}\brsq{\CP\tau_i}$, $i\in\{1,2\}$, we need to prove that
\begin{equs}
\Upsilon_{\hat F}\brsq{\nabla_{\tau_2}\tau_1 + \nabla^{1}_{\tau_2}\tau_1}=q\Upsilon_{F}\brsq{\CP\brsq{\nabla_{\tau_2}\tau_1 + \nabla^{1}_{\tau_2}\tau_1}}.
\end{equs}
Due to the linearity of $\Upsilon$ and the definition of $\CP$, we find that 
\begin{equs}\label{eq:localeq1}
	\begin{aligned}
	aq \Upsilon_{F}\brsq{\begin{tikzpicture}[scale=0.2,baseline=2]
			\draw[symbols]  (-.5,2.5) -- (0,0) ;
			\draw[tinydots] (0,0)  -- (0,-1.3);
			\node[var] (root) at (0,-0.1) {\tiny{$ \bar\tau_1 $ }};
			\node[var] (diff) at (-0.5,2.5) {\tiny{$ \bar\tau_2 $ }};
	\end{tikzpicture}} + q\frac{1}{2}\Upsilon_{{F}}\brsq{\begin{tikzpicture}[scale=0.2,baseline=2]
			\coordinate (root) at (0,0);
			\coordinate (t1) at (-1,2);
			\coordinate (t2) at (1,2);
			\draw[tinydots] (root)  -- +(0,-0.8);
			\draw[kernels2] (t1) -- (root);
			\draw[kernels2] (t2) -- (root);
			\node[not] (rootnode) at (root) {};
			\node[var] (t1) at (t1) {\tiny{$ \bar\tau_1 $}};
			\node[var] (t1) at (t2) {\tiny{$ \bar\tau_2 $}};
	\end{tikzpicture}} & = 
a\Upsilon_{\hat F}\brsq{\begin{tikzpicture}[scale=0.2,baseline=2]
\draw[symbols]  (-.5,2.5) -- (0,0) ;
\draw[tinydots] (0,0)  -- (0,-1.3);
\node[var] (root) at (0,-0.1) {\tiny{$ \tau_1 $ }};
\node[var] (diff) at (-0.5,2.5) {\tiny{$ \tau_2 $ }};
\node[blank] at (0.3,1.25) {\tiny{$0$}};
\end{tikzpicture}} + \frac{1}{2}\Upsilon_{\hat{F}}\brsq{\begin{tikzpicture}[scale=0.2,baseline=2]
\coordinate (root) at (0,0);
\coordinate (t1) at (-1,2);
\coordinate (t2) at (1,2);
\draw[tinydots] (root)  -- +(0,-0.8);
\draw[kernels2] (t1) -- (root);
\draw[kernels2] (t2) -- (root);
\node[not] (rootnode) at (root) {};
\node[var] (t1) at (t1) {\tiny{$ \tau_1 $}};
\node[var] (t1) at (t2) {\tiny{$ \tau_2 $}};
\node[blank] at (-1.2,0.6) {\tiny{$0$}};
\node[blank] at (1.2,0.6) {\tiny{$0$}};
\end{tikzpicture}}+\partial(a\cdot)\Upsilon_{\hat{F}}\brsq{\begin{tikzpicture}[scale=0.2,baseline=2]
\draw[symbols]  (-.5,2.5) -- (0,0) ;
\draw[tinydots] (0,0)  -- (0,-1.3);
\node[var] (root) at (0,-0.1) {\tiny{$ \tau_1 $ }};
\node[var] (diff) at (-0.5,2.5) {\tiny{$ \tau_2 $ }};
\node[blank] at (0.3,1.25) {\tiny{$1$}};
\end{tikzpicture}} \\ & + \frac{1}{2}\; (\partial \cdot )  \left( 
\begin{tikzpicture}[scale=0.2,baseline=2]
\coordinate (root) at (0,0);
\coordinate (t1) at (-1,2);
\coordinate (t2) at (1,2);
\draw[tinydots] (root)  -- +(0,-0.8);
\draw[kernels2] (t1) -- (root);
\draw[kernels2] (t2) -- (root);
\node[not] (rootnode) at (root) {};
\node[var] (t1) at (t1) {\tiny{$ \tau_1 $}};
\node[var] (t1) at (t2) {\tiny{$ \tau_2 $}};
\node[blank] at (-1.2,0.6) {\tiny{$1$}};
\node[blank] at (1.2,0.6) {\tiny{$0$}};
\end{tikzpicture}  + 
\begin{tikzpicture}[scale=0.2,baseline=2]
\coordinate (root) at (0,0);
\coordinate (t1) at (-1,2);
\coordinate (t2) at (1,2);
\draw[tinydots] (root)  -- +(0,-0.8);
\draw[kernels2] (t1) -- (root);
\draw[kernels2] (t2) -- (root);
\node[not] (rootnode) at (root) {};
\node[var] (t1) at (t1) {\tiny{$ \tau_1 $}};
\node[var] (t1) at (t2) {\tiny{$ \tau_2 $}};
\node[blank] at (-1.2,0.6) {\tiny{$0$}};
\node[blank] at (1.2,0.6) {\tiny{$1$}};
\end{tikzpicture} \right) 
\end{aligned}
\end{equs}
To prove this we compute each summand above in turn.
\begin{equs}
\Upsilon_{\hat{F}}\brsq{
\begin{tikzpicture}[scale=0.2,baseline=2]
\draw[symbols]  (-.5,2.5) -- (0,0) ;
\draw[tinydots] (0,0)  -- (0,-1.3);
\node[var] (root) at (0,-0.1) {\tiny{$\tau_1 $ }};
\node[var] (diff) at (-0.5,2.5) {\tiny{$ \tau_2 $ }};
\node[blank] at (0.3,1.25) {\tiny{$0$}};
\end{tikzpicture}}&=\Upsilon_{\hat{F}}\brsq{\tau_1}\partial_v\Upsilon_{\hat{F}}\brsq{\tau_2} \\
&=q\Upsilon_{F}\brsq{\CP\tau_1}\partial_v\left[q\Upsilon_{F}\brsq{\CP\tau_2}\right] \\
&=q\Upsilon_{F}\brsq{\CP\tau_1}\left[q\partial_v\left(\Upsilon_{F}\brsq{\CP\tau_2}\right)+\Upsilon_{F}\brsq{\CP\tau_2}\partial_v(q)\right] \\
&= q\Upsilon_{F}\brsq{\CP\tau_1}\Bigl[q\partial_u\left(\Upsilon_{F}\brsq{\CP\tau_2}\right)\partial_vu \\
&\hspace{2cm}- \Upsilon_{F}\brsq{\CP\tau_2}\frac{a''v_c+(a')^2v_{cc}}{q}\Bigr] \\
&= q\Upsilon_{F}\brsq{\CP\tau_1}\partial_u\left(\Upsilon_{F}\brsq{\CP\tau_2}\right) \\
&\hspace{2cm}-\Upsilon_{F}\brsq{\CP\tau_1}\Upsilon_{F}\brsq{\CP\tau_2}(a''v_c+(a')^2v_{cc}),
\end{equs}
We compute then:
\begin{equs}
\frac{1}{2}\Upsilon_{\hat{F}}\brsq{\begin{tikzpicture}[scale=0.2,baseline=2]
\coordinate (root) at (0,0);
\coordinate (t1) at (-1,2);
\coordinate (t2) at (1,2);
\draw[tinydots] (root)  -- +(0,-0.8);
\draw[kernels2] (t1) -- (root);
\draw[kernels2] (t2) -- (root);
\node[not] (rootnode) at (root) {};
\node[var] (t1) at (t1) {\tiny{$ \tau_1 $}};
\node[var] (t1) at (t2) {\tiny{$ \tau_2 $}};
\node[blank] at (-1.2,0.6) {\tiny{$0$}};
\node[blank] at (1.2,0.6) {\tiny{$0$}};
\end{tikzpicture}}&=\frac{1}{2}\Upsilon_{\hat{F}}\brsq{\tau_1}\Upsilon_{\hat{F}}\brsq{\tau_2}\partial^{2}_{v_x}\hat{F}_1 \\
&=\frac{1}{2}q^{2}\Upsilon_F\brsq{\CP\tau_1}\Upsilon_{{F}}\brsq{\CP\tau_2}\partial^{2}_{v_x}\Bigl[(\partial_xu)^{2}\hat{f} \\
&\hspace{1cm}+2(aa')(\partial_xu)v_{cx} + a'(u)(\partial_xu)v_x\Bigr] \\
&=\frac{1}{2}q^{2}\Upsilon_F\brsq{\CP\tau_1}\Upsilon_{{F}}\brsq{\CP\tau_2}\partial_{v_x}\Bigl[\frac{2}{q}\hat{f}\partial_xu \\
&\hspace{1cm}+\frac{2}{q}(aa')v_{cx}+a'\left(\partial_xu + \frac{v_x}{q}\right)\Bigr] \\
&=\frac{q^2}{2}\Upsilon_F\brsq{\CP\tau_1}\Upsilon_{F}\brsq{\CP\tau_2}\frac{2}{q^2}(\hat{f}+qa') \\
&=(q(f-a')+a(a')^{2}v_{cc} \\
&\hspace{2cm}+aa''v_c+qa')\Upsilon_F\brsq{\CP\tau_2}\Upsilon_F\brsq{\CP\tau_1}
\end{equs}
The next term is handled thusly:
\begin{equs}
(\partial a \cdot  ) \Upsilon_{\hat{F}}\brsq{
\begin{tikzpicture}[scale=0.2,baseline=2]
\draw[symbols]  (-.5,2.5) -- (0,0) ;
\draw[tinydots] (0,0)  -- (0,-1.3);
\node[var] (root) at (0,-0.1) {\tiny{$ \tau_1 $ }};
\node[var] (diff) at (-0.5,2.5) {\tiny{$ \tau_2 $ }};
\node[blank] at (0.3,1.25) {\tiny{$1$}};
\end{tikzpicture}}&= (\partial a \cdot  )\Upsilon_{\hat{F}}\brsq{\tau_1}\partial_{v_c}\Upsilon_{\hat{F}}\brsq{\tau_2} \\
&= (\partial a \cdot  ) q\Upsilon_{F}\brsq{\CP\tau_1}\partial_{v_c}\left(q\Upsilon_{{F}}\brsq{\CP\tau_2}\right) \\
&= (\partial a \cdot  ) -qa'\Upsilon_{F}\brsq{\CP\tau_1}\Upsilon_{F}\brsq{\CP\tau_2}
\\ & =  -qa'\Upsilon_{F}\brsq{\CP\tau_1}\Upsilon_{F}\brsq{\CP\tau_2}    - \left( \partial \cdot \right)a  a' q \Upsilon_{F}\brsq{\CP\tau_1}\Upsilon_{F}\brsq{\CP\tau_2}
\end{equs}
and then the last term on the left hand side:
\begin{equs}
\Upsilon_{\hat{F}}\brsq{\begin{tikzpicture}[scale=0.2,baseline=2]
\coordinate (root) at (0,0);
\coordinate (t1) at (-1,2);
\coordinate (t2) at (1,2);
\draw[tinydots] (root)  -- +(0,-0.8);
\draw[kernels2] (t1) -- (root);
\draw[kernels2] (t2) -- (root);
\node[not] (rootnode) at (root) {};
\node[var] (t1) at (t1) {\tiny{$ \tau_1 $}};
\node[var] (t1) at (t2) {\tiny{$ \tau_2 $}};
\node[blank] at (-1.2,0.6) {\tiny{$0$}};
\node[blank] at (1.2,0.6) {\tiny{$1$}};
\end{tikzpicture}}&=\frac{1}{2}\Upsilon_{\hat{F}}\brsq{\tau_1}\Upsilon_{\hat{F}}\brsq{\tau_2}\partial_{v_{cx}}\partial_{v_x}\hat{F}_1 \\
&=\frac{1}{2}q^{2}\Upsilon_F\brsq{\CP\tau_1}\Upsilon_{{F}}\brsq{\CP\tau_2}\partial_{v_{cx}}\partial_{v_x}\Bigl[(\partial_xu)^{2}\hat{f} \\
&\hspace{2cm}+2(aa')(\partial_xu)v_cx + a'(u)(\partial_xu)v_x\Bigr] \\
&=\frac{1}{2}q^{2}\Upsilon_F\brsq{\CP\tau_1}\Upsilon_{{F}}\brsq{\CP\tau_2}\partial_{v_{cx}}\Bigl[\frac{2}{q}\hat{f}\partial_xu \\
&\hspace{2cm}+\frac{2}{q}(aa')v_{cx}+a'\left(\partial_xu \right)\Bigr] \\
&=aa'q\Upsilon_{F}\brsq{\CP\tau_1}\Upsilon_{{F}}\brsq{\CP\tau_2}
\end{equs}
On the other hand, one has:
On the other hand it is easy to see that
\begin{equs}
\Upsilon_{{F}}\brsq{
\begin{tikzpicture}[scale=0.2,baseline=2]
\draw[symbols]  (-.5,2.5) -- (0,0) ;
\draw[tinydots] (0,0)  -- (0,-1.3);
\node[var] (root) at (0,-0.1) {\tiny{$ \bar\tau_1 $ }};
\node[var] (diff) at (-0.5,2.5) {\tiny{$ \bar\tau_2 $ }};
\end{tikzpicture}}&=\Upsilon_{F}\brsq{\CP\tau_1}\partial_{u}\Upsilon_{F}\brsq{\CP\tau_2}
\end{equs}
and from \eqref{eq:transformed}
\begin{equs}
\frac{1}{2}\Upsilon_{{F}}\brsq{\begin{tikzpicture}[scale=0.2,baseline=2]
\coordinate (root) at (0,0);
\coordinate (t1) at (-1,2);
\coordinate (t2) at (1,2);
\draw[tinydots] (root)  -- +(0,-0.8);
\draw[kernels2] (t1) -- (root);
\draw[kernels2] (t2) -- (root);
\node[not] (rootnode) at (root) {};
\node[var] (t1) at (t1) {\tiny{$ \bar\tau_1 $}};
\node[var] (t1) at (t2) {\tiny{$ \bar\tau_2 $}};
\end{tikzpicture}}&=\frac{1}{2}\Upsilon_F\brsq{\CP\tau_1}\Upsilon_F\brsq{\CP\tau_2}\partial^{2}_{\partial_{x}u}F \\
&=\Upsilon_F\brsq{\CP\tau_1}\Upsilon_F\brsq{\CP\tau_2}(f-a')
\end{equs}
 One has also:
\begin{equs}
\Upsilon_{\hat{F}}\brsq{\begin{tikzpicture}[scale=0.2,baseline=2]
\coordinate (root) at (0,0);
\coordinate (t1) at (-1,2);
\coordinate (t2) at (1,2);
\draw[tinydots] (root)  -- +(0,-0.8);
\draw[kernels2] (t1) -- (root);
\draw[kernels2] (t2) -- (root);
\node[not] (rootnode) at (root) {};
\node[var] (t1) at (t1) {\tiny{$ \tau_1 $}};
\node[var] (t1) at (t2) {\tiny{$ \tau_2 $}};
\node[blank] at (-1.2,0.6) {\tiny{$1$}};
\node[blank] at (1.2,0.6) {\tiny{$0$}};
\end{tikzpicture}} = \Upsilon_{\hat{F}}\brsq{\begin{tikzpicture}[scale=0.2,baseline=2]
\coordinate (root) at (0,0);
\coordinate (t1) at (-1,2);
\coordinate (t2) at (1,2);
\draw[tinydots] (root)  -- +(0,-0.8);
\draw[kernels2] (t1) -- (root);
\draw[kernels2] (t2) -- (root);
\node[not] (rootnode) at (root) {};
\node[var] (t1) at (t1) {\tiny{$ \tau_1 $}};
\node[var] (t1) at (t2) {\tiny{$ \tau_2 $}};
\node[blank] at (-1.2,0.6) {\tiny{$0$}};
\node[blank] at (1.2,0.6) {\tiny{$1$}};
\end{tikzpicture}},
\end{equs}
Summarising the various computations, one gets:
\begin{equs}
	aq \Upsilon_{F}\brsq{\begin{tikzpicture}[scale=0.2,baseline=2]
		\draw[symbols]  (-.5,2.5) -- (0,0) ;
		\draw[tinydots] (0,0)  -- (0,-1.3);
		\node[var] (root) at (0,-0.1) {\tiny{$ \bar\tau_1 $ }};
		\node[var] (diff) at (-0.5,2.5) {\tiny{$ \bar\tau_2 $ }};
\end{tikzpicture}} + q\frac{1}{2}\Upsilon_{{F}}\brsq{\begin{tikzpicture}[scale=0.2,baseline=2]
		\coordinate (root) at (0,0);
		\coordinate (t1) at (-1,2);
		\coordinate (t2) at (1,2);
		\draw[tinydots] (root)  -- +(0,-0.8);
		\draw[kernels2] (t1) -- (root);
		\draw[kernels2] (t2) -- (root);
		\node[not] (rootnode) at (root) {};
		\node[var] (t1) at (t1) {\tiny{$ \bar\tau_1 $}};
		\node[var] (t1) at (t2) {\tiny{$ \bar\tau_2 $}};
\end{tikzpicture}} = q (f-a'+a) \Upsilon_F\brsq{\CP\tau_1}\Upsilon_F\brsq{\CP\tau_2}.
\end{equs}
On the other hand, one has
\begin{equs}
	a & \Upsilon_{\hat F}\brsq{\begin{tikzpicture}[scale=0.2,baseline=2]
			\draw[symbols]  (-.5,2.5) -- (0,0) ;
			\draw[tinydots] (0,0)  -- (0,-1.3);
			\node[var] (root) at (0,-0.1) {\tiny{$ \tau_1 $ }};
			\node[var] (diff) at (-0.5,2.5) {\tiny{$ \tau_2 $ }};
			\node[blank] at (0.3,1.25) {\tiny{$0$}};
	\end{tikzpicture}} + \frac{1}{2}\Upsilon_{\hat{F}}\brsq{\begin{tikzpicture}[scale=0.2,baseline=2]
			\coordinate (root) at (0,0);
			\coordinate (t1) at (-1,2);
			\coordinate (t2) at (1,2);
			\draw[tinydots] (root)  -- +(0,-0.8);
			\draw[kernels2] (t1) -- (root);
			\draw[kernels2] (t2) -- (root);
			\node[not] (rootnode) at (root) {};
			\node[var] (t1) at (t1) {\tiny{$ \tau_1 $}};
			\node[var] (t1) at (t2) {\tiny{$ \tau_2 $}};
			\node[blank] at (-1.2,0.6) {\tiny{$0$}};
			\node[blank] at (1.2,0.6) {\tiny{$0$}};
	\end{tikzpicture}}+\partial(a\cdot)\Upsilon_{\hat{F}}\brsq{\begin{tikzpicture}[scale=0.2,baseline=2]
			\draw[symbols]  (-.5,2.5) -- (0,0) ;
			\draw[tinydots] (0,0)  -- (0,-1.3);
			\node[var] (root) at (0,-0.1) {\tiny{$ \tau_1 $ }};
			\node[var] (diff) at (-0.5,2.5) {\tiny{$ \tau_2 $ }};
			\node[blank] at (0.3,1.25) {\tiny{$1$}};
	\end{tikzpicture}}    \\ & + \frac{1}{2}\; (\partial \cdot ) \left(  	\Upsilon_{\hat{F}}\brsq{
\begin{tikzpicture}[scale=0.2,baseline=2]
\coordinate (root) at (0,0);
\coordinate (t1) at (-1,2);
\coordinate (t2) at (1,2);
\draw[tinydots] (root)  -- +(0,-0.8);
\draw[kernels2] (t1) -- (root);
\draw[kernels2] (t2) -- (root);
\node[not] (rootnode) at (root) {};
\node[var] (t1) at (t1) {\tiny{$ \tau_1 $}};
\node[var] (t1) at (t2) {\tiny{$ \tau_2 $}};
\node[blank] at (-1.2,0.6) {\tiny{$1$}};
\node[blank] at (1.2,0.6) {\tiny{$0$}};
\end{tikzpicture}}  + 
	\Upsilon_{\hat{F}}\brsq{\begin{tikzpicture}[scale=0.2,baseline=2]
\coordinate (root) at (0,0);
\coordinate (t1) at (-1,2);
\coordinate (t2) at (1,2);
\draw[tinydots] (root)  -- +(0,-0.8);
\draw[kernels2] (t1) -- (root);
\draw[kernels2] (t2) -- (root);
\node[not] (rootnode) at (root) {};
\node[var] (t1) at (t1) {\tiny{$ \tau_1 $}};
\node[var] (t1) at (t2) {\tiny{$ \tau_2 $}};
\node[blank] at (-1.2,0.6) {\tiny{$0$}};
\node[blank] at (1.2,0.6) {\tiny{$1$}};
\end{tikzpicture} } \right) \\ & = a q  \Upsilon_{F}\brsq{\CP\tau_1}\partial_u\left(\Upsilon_{F}\brsq{\CP\tau_2}\right) -(a a''v_c+a (a')^2v_{cc}) \Upsilon_{F}\brsq{\CP\tau_1}\Upsilon_{F}\brsq{\CP\tau_2}
\\ & + (q(f-a')+a(a')^{2}v_{cc} +aa''v_c+qa')\Upsilon_F\brsq{\CP\tau_2}\Upsilon_F\brsq{\CP\tau_1}
\\ &  -qa'\Upsilon_{F}\brsq{\CP\tau_1}\Upsilon_{F}\brsq{\CP\tau_2}  - \left( \partial \cdot \right)a  a' q \Upsilon_{F}\brsq{\CP\tau_1}\Upsilon_{F}\brsq{\CP\tau_2}
\\ & +  \left( \partial \cdot \right) aa'q\Upsilon_{F}\brsq{\CP\tau_1}\Upsilon_{{F}}\brsq{\CP\tau_2}
\\ &= q (f-a'+a) \Upsilon_F\brsq{\CP\tau_1}\Upsilon_F\brsq{\CP\tau_2}.
\end{equs}
which allows us to conclude.
\end{proof}

\section{Global-in-time solutions}
\label{sec::4}

When the driving noise $\xi$ is white in space-time, the flexibility to treat all KPZ-like right-hand sides simultaneously (and, by the chain rule, consistently) allows to derive global-in-time well-posedness of the equation. 
The proof relies on the fact that the chain rule allows to see large families of equations as equivalent, as long as any two can be formally transformed into one another by a diffeomorphism.
This makes rigorous some previously formally used change of variables such as the Cole-Hopf transform for the classical KPZ equation or the transformation of quasilinear singular SPDEs mentioned (without proof) in \cite{Bailleul2019}.
Moreover, we crucially use that among all equivalent equations, there is a distinguished It\^o one, for which many stochastic analytic tools are available to deduce its global-in-time existence. In particular, for vector-valued equations like the geometric stochastic heat equation in \cite{BGHZ} the extension of this argument does not seem straightforward.

\begin{theorem}\label{thm:GWP}
	Let $a\in\mathcal{C}^6$, $f,g\in \mathcal{C}^{5}$ such that $a$ takes values in $[\lambda,\lambda^{-1}]$ for some $\lambda>0$, and that $g,g',f,\int f,a'$ are bounded.
	%$g'+\tfrac{(f-a')g}{a}$ is bounded.
	%\begin{equ}
	%\sup_{x}|\int_0^xf(r)\,dr|<\infty
	%\end{equ}
	Let $u_0\in\CC^\alpha(\mbbT)$ for some $\alpha>0$ and let $\xi$ be the space-time white noise.
	Assume that a model $\PPi$ above $\xi$ is a limit of smooth models satisfying both Assumption~\ref{chain rule} and Assumption~\ref{Assumption_Ito}.
	Then the solution of the equation
	\begin{equ}\label{eq:1}
		\d_t u-a(u)\d_x^2 u=f(u)(\d_x u)^2+g(u)\xi
	\end{equ}
	with initial condition $u_0$,
	exists for all times $t\in[0,1]$.
\end{theorem}
\begin{remark}\label{rem:GWP}
	Since the conditions are invariant under the change $f\leftrightarrow f-a'$, the statement for equations with divergence form second order operator $\d_x(a(u)\d_x u)$ is equivalent.
\end{remark}

\begin{proof}
	In light of Remark \ref{rem:GWP}, we will prove the statement for divergence form equations.
	It suffices to show that with some $\alpha'>0$, with probability $1$, $\|u_t\|_{\mathcal{C}^{\alpha'}(\mbbT)}$ remains finite for $t\in[0,1]$.
	
	By the chain rule, for any sufficiently regular diffeomorphism $\psi$,   $v=\psi(u)$ solves the equation
	\begin{equ}\label{eq:ito0}
		\d_t v-\d_x(\hat a(v)\d_x v)=\hat f(v)(\d_x v)^2+\hat g(v)\xi,%\qquad v(0,\cdot)=\psi(u_0(\cdot))
	\end{equ}
	with initial condition $v_0(\cdot)=u_0(\psi(\cdot))$,
	where the nonlinearities $\hat a$, $\hat f$, and $\hat g$ are given by
	\begin{equ}
		\hat a=a\circ\varphi,\qquad\hat f=\frac{(a\circ\varphi)\varphi''+(f\circ\varphi)(\varphi')^2}{\varphi'},\qquad \hat g=\frac{g\circ\varphi}{\varphi'},
	\end{equ}
	and $\varphi=\psi^{-1}$.
	%\begin{equs}[eq:chain rule]
	%\varphi'(v)\d_t v & =\d_t u  =\d_x\big(a(u)\d_x u\big)+f(u)(\d_x u)^2+g(u)\xi
	%\\
	%&=\d_x\big((a\circ\varphi)(v)\varphi'(v)\d_x v\big)+(f\circ\varphi)(v)(\varphi'(v))^2(\d_x v)^2+(g\circ\varphi)(v)\xi.
	%\\
	%&=\d_x\big((a\circ\varphi)(v)\d_x v\big)+\big(\varphi''(v)+(f\circ\varphi)(v)(\varphi'(v))^2\big)(\d_x v)^2+(g\circ\varphi)(v)\xi.
	%\end{equs}
	Let us set $F(x)=\int_0^x\tfrac{f(r)}{a(r)}\,dr$ and $\psi(x)=\int_0^xe^{F(r)}\,dr$.
	By our assumptions, both $\psi$ and its inverse $\varphi$ have derivatives bounded away from $0$, so it is a diffeomorphism from $\mbbR$ to itself. One then has
	\begin{equ}\label{eq:phi prime}
		\varphi'=\frac{1}{\psi'\circ\varphi}=e^{-F\circ\varphi},
	\end{equ}
	and so taking logarithms and derivatives on both sides, one gets
	\begin{equ}
		\frac{\varphi''}{\varphi'}=-\frac{f\circ\varphi}{a\circ\varphi}\varphi'.
	\end{equ}
	Hence, $\hat f=0$. In particular, in \eqref{eq:ito0} the only remaining ill-defined product is $\hat g(v)\xi$ (since $\hat a(v)\d_x v$ can be written as a total derivative),
	which one can interpret as an It\^o integral.
	Moreover, by \eqref{eq:phi prime} one has
	\begin{equ}
		\hat g'=\big((ge^F)\circ\varphi\big)'=g'+\frac{fg}{a}, 
	\end{equ}
	which in particular implies that $\hat a,\hat g\in \mathcal{C}^{5}$, and $\hat a, \hat g,\hat g'$ are bounded.
	
	At this point one would like to wrap up the proof by saying that the solution of the It\^o equation \eqref{eq:ito0} is known to exist for all times \cite{Pardoux, KR}, and therefore so does $u$.
	There are two issues with this argument. First, it is not at all obvious that the It\^o solution and the regularity structure solution coincide.
Second, the It\^o theory in its most classical form guarantees the non-blow-up of the solution only in the rather weak space $H^{-1}(\mbbT)$, which is far from the non-blow-up in $\cC^{\alpha'}(\mbbT)$ that the regularity structures solution requires. 
	%By refining the energy estimates this can be improved to $H^{-1/2-}(\mbbT)$ \cite[Rem.~2.7]{DGG_STWN}, but this is still far from the non-blow-up in $\cC^{\alpha'}(\mbbT)$ that the regularity structures solution requires. 
	
	The resolutions of these issues are of independent interest and are stated separately below. By Theorem \ref{thm:consistency}, under our assumptions on the model $v$ indeed equals to the It\^o solution of \eqref{eq:ito0}, up to its maximal existence time. 
	By Theorem \ref{thm:Ito-regularity}, for small enough $\alpha'>0$ one has $\E\|v\|_{\mathcal{C}^{\alpha'}([0,1]\times \mbbT)}^2<\infty$. Therefore $\|v_t\|_{\mathcal{C}^{\alpha'}(\mbbT)}$ remains finite almost surely, which implies that so does $\|u_t\|_{\mathcal{C}^{\alpha'}(\mbbT)}$, and the proof is finished.
\end{proof}
\begin{remark}
	The conditions on the coefficients in Theorem \ref{thm:GWP} are likely possible to improve,
	% but, particularly concerning $f$, are likely nontrivial to substantially improve.
	for example, the reader might notice that the usual KPZ equation, $a=f=g=1$, is not covered. The reason is that the corresponding transformation $\psi(x)=e^x$ is a diffeomorphism between $\mathbb{R}$ and $(0,+\infty)$, and therefore one needs to show that the solution of the corresponding It\^o equation not only does not blow up but does not reach $0$ either.
	In the case of the KPZ equation $\psi(u)$ solves the stochastic heat equation with linear multiplicative noise for which such a positivity result is known from \cite{Mueller}.
	For general quasilinear It\^o equations we are not aware of analogous results.
\end{remark}

\subsection{It\^o consistency}
An intermediate step towards global well-posedness, and a question of interest in its own right, is whether solutions of SPDEs in the It\^o sense and the regularity structure sense coincide.
Of course, first, we have to restrict our attention to a class where both theories apply, which is the case for nondegenerate quasilinear equations in divergence form
\begin{equ}\label{eq:Ito-div}
	\d_t u-\d_x(a(u)\d_x u)=g(u)\xi.
\end{equ}
Let us mention that without a semigroup/heat kernel available, the It\^o SPDE approaches based on mild formulations \cite{DPZ, Walsh} cannot handle \eqref{eq:Ito-div}.
However, the monotone operator approach of \cite{Pardoux, KR} can accommodate quasilinear operators in divergence form and therefore can be used to define (analytically) weak solutions in the It\^o sense  (see Definition \ref{def:Ito} below) and prove their well-posedness (see Theorem \ref{thm:Ito-WP} below).
The consistency of the two solution theories depends on the choice of the model,
the required properties are summarised as follows.
 $\rho$ is a fixed $C^\infty$ function on $\mbbR$ that is nonnegative, even, has support contained in the unit ball, and integrates to $1$. We use the notations $\varrho^+(x)=\rho(x+1)$, $\rho_\eps(x)=\eps^{-1}\rho(\eps^{-1}x)$, $\rho_\eps^+(x)=\eps^{-1}\rho^+(\eps^{-1}x)$, $\varrho_{\bar\eps,\eps}(t,x)=\rho_{\bar\eps}^+(t)\rho_\eps(x)$, and $\xi_{\bar\eps,\eps}=\varrho_{\bar \eps,\eps}\ast\xi$ (the purpose of the shifting in time is that $\xi_{\bar\eps,\eps}$ is adapted). 
\begin{assumption} \label{Assumption_Ito}
The renormalisation functions $C_\eps^c(\tau)$ coincide with the BPHZ renormalisation functions for $\tau\in T_{\Itoo}:=\{\<3>\;,\<AAA>\;,\<AAM>\;,\<AMA>\;,\<AMM>\;\}$.
%For $\eps,\beps>0$, we say that a model $\PPi^{\bar\eps,\eps}$ that is a lift of $\xi_{\bar\eps,\eps}$ is a smooth It\^o model, if the renormalisation map sending the canonical model to $\PPi^{\bar\eps,\eps}$ agrees with the BPHZ renormalisation map when restricted to the set $T_{\Itoo}:=\{\<3>\;,\<AAA>\;,\<AAM>\;,\<AMA>\;,\<AMM>\;\}$.

%Further, we say that a model $\PPi$ that is a lift of $\xi$ is an It\^o model if it is a limit of smooth It\^o models $\PPi^{\bar\eps,\eps}$ with $0<\bar\eps\leq\eps^5$.
\end{assumption}

\begin{proposition} \label{model_both} There exists a choice of renormalisation constants such that the renormalised model $\PPi^{\bar\eps,\eps}$ satisfies both Assumption~\ref{chain rule} and Assumption~\ref{Assumption_Ito}.
	\end{proposition}
\begin{proof}
	This is just a consequence of \cite[Prop. 6.13]{BGHZ} where one shows that $ T_{\Itoo} $ is linearly independent of the terms satisfying the chain rule.
	\end{proof}

Similarly to \cite{wong}, one can say more about the counterterms in an equation generated by smooth It\^o models.
\begin{proposition}\label{prop:Ito}
Let $\eps\to 0$ and $\bar \eps\leq\eps^5$ and consider approximating \eqref{eq:Ito-div} with the approximate equations driven by $\xi_{\beps,\eps}$. Then the counterterm can be chosen to be
\begin{equ}\label{eq:Itocounter}
		\tfrac{1}{2}C_\eps^{\Itoo}g'(u_{\beps,\eps})g(u_{\beps,\eps}),
	\end{equ}
	where $C^{\Itoo}_\eps=\|\rho_{\eps}\|_{L^2}^2$.
\end{proposition}
Note that
	\begin{equ}
		C^{\Itoo}_\eps=\eps^{-1}C_1^{\Itoo}=\sum_{n\in\mbbN}\big(\varrho_\eps\ast e_n(x)\big)^2,
	\end{equ}
for any orthonormal basis $(e_n)_{n\in\mbbN}$ of $L^2(\mbbT)$ and any $x\in\mbbT$.
%\begin{assumption}\label{asn:model-ito}
%	Assume that for $\eps>0,\bar\eps>0$ with $\bar \eps\leq\eps^5$ there exist models $\PPi^{\bar\eps,\eps}$ satisfying: A) they are lifts of the noise $\xi^{\bar\eps,\eps}$; B) $\PPi^{\bar\eps,\eps}\to\PPi$; C) the counterterm of the equation for the model $\PPi^{\bar\eps,\eps}$ is
%	\begin{equ}
%		\tfrac{1}{2}C_\eps^{\Itoo}g'(u)g(u),
%	\end{equ}
%	where
%	\begin{equ}
%		C^{\Itoo}_\eps=\|\rho^{\eps}\|_{L^2}^2=\eps^{-1}\|\rho^{1}\|_{L^2}^2=\eps^{-1}C_1^{\Itoo}.
%	\end{equ}
%\end{assumption}
%\begin{remark}
%	{\color{blue}Mate: The above is an ``intermediate'' definition, a model with these properties should basically come from our general theorem. For example, what happens in the semilinear case is that the counterterm is of the form
%		\begin{equ}
%			c_{\bar \eps,\eps}(\<3>)g'(u)g(u)+\sum_\tau c_{\bar \eps,\eps}(\tau)\tilde\tau g_\tau(u)
%		\end{equ}
%		where the sum runs over $\tau\in\{\<AAA>\;,\<AAM>\;,\<AMA>\;,\<AMM>\;\}$ and $\tilde g_\tau$ are some smooth functions. It is not hard to see (see \cite[Sec~2.2]{wong}) that for fixed $\eps>0$, $c_{\bar\eps,\eps}(\<3>)\to \frac{1}{2}C^{\Itoo}_\eps$ and $c_{\bar\eps,\eps}(\tau)\to 0$ as $\bar\eps\to0$ for all $\tau$ as above. This is partially also detailed in the next remark.}
%\end{remark}
\begin{proof}
The statement is similar to the results of \cite[Sec.~2.3]{wong}.
By Proposition \ref{model_both} we can take models $\PPi_{\beps,\eps}$ above $\xi_{\beps,\eps}$ that satisfy both Assumptions \ref{chain rule} and \ref{Assumption_Ito}.
By Theorem \ref{main result 2} and the fact that for \eqref{eq:Ito-div} the nonlinearity $F$ is simply given by $g(u)\xi$, the counterterm is of the form
\begin{equ}\label{eq:Itocounter0}
\sum_{\tau\in T_{\Itoo}}C_{\beps,\eps}^{a(u_{\beps,\eps})}(\tau) \frac{\Upsilon_{F}[\tau](u_{\beps,\eps})}{S(\tau)}.
\end{equ}
%	It might be surprising that the counterterm does not depend on the coefficient $a$. This is an effect of the regime $\bar \eps\ll\eps^2$. Indeed, the leading order counterterm for quasilinear equations is
%	\begin{equ}
%		c_{\bar\eps,\eps}^{a(u)}(\<3>)g'(u)g(u).
%	\end{equ}
For the BPHZ renormalisation, one has
$C^c_{\beps,\eps}(\<AMA>)=C^c_{\beps,\eps}(\<AMM>)=0$
simply as a consequence of Gaussianity, see \cite[Sec.~4.3]{wong}.
Further, for arbitrary $\beta>0$, on the scale $\bar \eps\leq \eps^\beta$ one has
$C^c_{\beps,\eps}(\<AAA>),C^c_{\beps,\eps}(\<AAM>)\to 0$, see \cite[Sec.~2.3]{wong}.
The latter convergence is easily seen to be uniform in $c\in[\lambda,\lambda^{-1}]$, since all the relevant bounds on the kernels $K(c,\cdot)$ are uniform.

As for the only nonzero contribution to \eqref{eq:Itocounter0}, namely $\tau=\<3>$,
we give all details, 
partly to detail that the difference between \eqref{eq:Itocounter} and \eqref{eq:Itocounter0} goes to $0$ (\cite{wong} only discusses the difference being $o(\eps^{-1})$),
partly to emphasize the feature that the dependence on the quasilinear coefficient $a$ disappears.
Clearly, $\Upsilon_F[\<3>](u)=g'(u)g(u)$ and $S(\<3>)=1$.
Let $c\in[\lambda,\lambda^{-1}]$ be arbitrary. Using the scaling property $K(c,(\mu^2 t,\mu x))=\mu^{-1}K(c,(t,x))$ for $\mu\in(0,1]$ and $|t|^{1/2}+|x|\leq r$ for some $r>0$, one has for sufficiently small $\bar \eps,\eps$
	\begin{equ}
		C_{\bar\eps,\eps}^{c}(\<3>)=\varrho_{\bar\eps,\eps}\ast\bar\varrho_{\bar\eps,\eps}\ast K^c(0)%\to_{\bar\eps\to 0}\rho^\eps\ast\bar\rho^\eps\ast K(0)
		=\eps^{-1}\big(\varrho_{\bar\eps\eps^{-2},1}\ast\bar\varrho_{\bar\eps\eps^{-2},1}\ast K^c(0)\big),
	\end{equ}
	where $\bar\varrho_{\bar\eps,\eps}$ is the (space-time) reflection of $\varrho_{\bar\eps,\eps}$.
	% We see that the final expression does not depend on $c$.
	%Note that $K^c(t,x)\approx \bone_{t>0}\delta_0(x)$ as $t\to 0$, uniformly over $c\in[\lambda,\lambda^{-1}]$, therefore
	Let $\tilde \eps=\bar\eps\eps^{-2}$, and note $\tilde\eps\to 0$.
	After some elementary arrangement one has
	\begin{equs}
		\varrho_{\tilde \eps,1}\ast\bar\varrho_{\tilde \eps,1}\ast K^c(0)&=\int\rho^{+}_{\tilde\eps}\ast\bar\rho^{+}_{\tilde\eps}(t)\int\rho\ast\bar\rho(x)K^c(-t,-x)\bone_{t<0}\,dx\,dt,
		%\approx\rho\ast\bar\rho(0)\int\rho_{+}^{\tilde\eps}\ast\bar\rho_{+}^{\tilde\eps}(t)\bone_{t<0}\,dt,
		\\
		\tfrac{1}{2}C_1^{\Itoo}&=\tfrac{1}{2}\rho\ast\bar\rho(0)=\rho\ast\bar\rho(0)\int\rho^{+}_{\tilde\eps}\ast\bar\rho^{+}_{\tilde\eps}(t)\bone_{t<0}\,dt.
	\end{equs}
	Hence one has uniformly in $c\in[\lambda,\lambda^{-1}]$,
	\begin{equ}
		\big|\varrho_{\tilde \eps,1}\ast\bar\varrho_{\tilde \eps,1}\ast K^c(0)-\tfrac{1}{2}C_1^{\Itoo}\big|\lesssim\max_{|t|\lesssim \tilde \eps}\int |x| K^c(-t,-x)\,dx\lesssim 
		%\max_{|t|\lesssim \tilde \eps}t^{1/2}=
		\tilde\eps^{1/2}.
	\end{equ}
	Therefore if $\bar\eps\leq \eps^5$, we have
\begin{equ}
\big|C_{\bar\eps,\eps}^c(\<3>)-\tfrac{1}{2}C_{\eps}^{\Itoo}\big|\lesssim \eps^{1/2}
\end{equ}	
%	 $\eps^{-1}|C_{\bar\eps,\eps}^c(\<3>)-\tfrac{1}{2}C_{1}^{\Itoo}|\lesssim \eps^{1/2}$
uniformly in $c\in[\lambda,\lambda^{-1}]$, which completes the proof.
\end{proof}
The main result of the section is the following.
\begin{theorem}\label{thm:consistency}
	Assume $a,g\in \mathcal{C}^5$, $a\geq \lambda$ for some constant $\lambda>0$, $u_0\in \mathcal{C}^\alpha(\mbbT)$ for some $\alpha\in(0,1/2)$, and suppose that
	%$\PPi$ satisfies Assumption \ref{asn:model-ito}.
$\PPi$ is an model that is a limit of smooth models satisfying Assumptions \ref{chain rule} and \ref{Assumption_Ito}.
	Let $u^{\Itoo}$ be the solution of \eqref{eq:Ito-div} with initial condition $u_0$ in the sense of Theorem \ref{thm:Ito-WP} and
	let $u^{\mathrm{RS}}$ be the solution of \eqref{eq:Ito-div} with initial condition $u_0$ in the sense of Theorem \ref{main result 2}.
	Fix $L>\|u_0\|_{\mathcal{C}^\alpha(\mbbT)}$ and let define the random variable $T_L=\inf\{t:\,\|u_t^{\mathrm{RS}}(\cdot)\|_{\mathcal{C}^{\alpha/2}(\mbbT)}\geq L\}\wedge 1$.
	Then almost surely $u^{\Itoo}=u^{\mathrm{RS}}$ in $L^2([0,T_L]\times\mbbT)$.
\end{theorem}
In the semilinear case $a\equiv 1$ the statement of Theorem \ref{thm:consistency} is contained in the work of Hairer-Pardoux \cite{wong}. The proof in \cite{wong} makes essential use of the mild formulation of the equation, which is not available for quasilinear equations.
%Let us now sketch an alternative, although less self-contained proof. Let $\xi_{\bar\eps,\eps}$ be the space-time white noise $\xi$ mollified on scale $\bar\eps$ in time and $\eps$ in space. The regularity structure solution of the semilinear equation
%\begin{equ}
%\d_t u=\d_{x}^2 u+g(u)\xi
%\end{equ}
%enjoys good stability properties, and 
%So far we have only shifted the problem to the question of obtaining quasilinear variation of \cite{BMSS}, which does not seem a big improvement given that the mild formulation is also heavily used therein.
%The remainder of the section is therefore essentially devoted to another proof of \cite[Theorem 1.13]{BMSS} (with various technical differences)
Therefore we give another argument that relies on the stability of the It\^o equation with respect to spatial approximations of the noise. This gives a simpler alternative in the semilinear case, but more importantly for us, such spatial stability can be obtained also in the quasilinear case, using the weak formulation of the equation, see Section \ref{sec:Ito-stability} below.

\begin{proof}[of Theorem \ref{thm:consistency}]
	By Theorem \ref{thm:main} and Proposition \ref{prop:Ito}, we have that, $u_{\bar\eps,\eps}$, the solution of
	\begin{equ}
		\d_t u_{\bar\eps,\eps}-\d_x(a(u_{\bar\eps,\eps})\d_x u_{\bar\eps,\eps})=g(u_{\bar\eps,\eps})\xi_{\bar \eps,\eps}-\frac{1}{2}C_\eps^{\Itoo}g'(u_{\bar\eps,\eps})g(u_{\bar\eps,\eps})
	\end{equ}
	converges to $u^{\mathrm{RS}}$ as $\bar\eps,\eps\to 0$ with $\bar \eps\leq \eps^5$. Note that although the analytic arguments of \cite{reg, MH} would only yield convergence of local solutions, the arguments in \cite[Sec.~6]{GHM} yield convergence in probability in $\mathcal{C}^{\alpha/2}([0,T_L]\times\mbbT)$. 
By classical Wong-Zakai approximation results on quasilinear equations with spatially smooth noise (see \ref{sec:smoothWZ} below for details), for any fixed $\eps>0$, $u_{\bar\eps,\eps}$ converges to $u_\eps$ as $\bar \eps\to 0$ in probability in $C([0,1],L^2(\mbbT))$, where $u_\eps$ is the It\^o solution of
	\begin{equ}\label{eq:Ito-spatial}
		\d_t u_{\eps}-\d_x(a(u_{\eps})\d_x u_{\eps})=g(u_{\eps})\xi_{0,\eps}.
	\end{equ}
	%Recall furthermore the processes $u^\eps$ from Theorem \ref{thm:Ito-WP}.
	%Given the convergences
	%\begin{equ}
	%u^{\bar\eps,\eps}\underset{\bar \eps,\eps\to 0}{\longrightarrow} u^{\mathrm{RS}},\qquad u^{\eps}\underset{\eps\to 0}{\longrightarrow} u^{\Itoo},
	%\end{equ}
	The precise understanding of \eqref{eq:Ito-spatial} is given in Section \ref{sec:Ito-stability} below, as well as Theorem \ref{thm:Ito-WP} stating the convergence $u^\eps\to u^{\Itoo}$ in probability in $L^2([0,1]\times\mbbT)$ as $\eps\to 0$. This finishes the proof.
\end{proof}

\subsection{Stability of It\^o solutions}\label{sec:Ito-stability}
Here we recall the main relevant elements of the (analytically) weak It\^o solution theory of \eqref{eq:Ito-div} (for much more details and generality see \cite{Pardoux,KR, PR}) and prove their stability with respect to spatial approximations of the noise, which was crucial to the proof of Theorem \ref{thm:consistency} above.

Define the standard orthonormal basis of $L^2$ on $\mbbT$ by setting $e_0(x)\equiv 1$, $e_n(x)=\sqrt{2}\cos(\pi n x)$ for even $n\geq 1$, and $e_n(x)=\sqrt{2}\sin(\pi (n+1)x)$ for odd $n\geq 1$. 
Take a parameter $M\geq1$, and for $\alpha\in\mbbR$, $f,g\in \mathrm{span}\{e_{n}:\,n\geq 0\}$ define
\begin{equ}\label{eq:H alpha norm}
	\scal{f,g}_{H^\alpha_M}=\sum_{n\in\mbbN}(M+\pi^2\bar n^2)^{\alpha}\scal{f,e_n}_{L^2}\scal{g,e_n}_{L^2},
\end{equ}
where $\bar n=n$ for even $n$ and $\bar n=n+1$ for odd $n$.
Set $\|f\|_{H^\alpha_M}^2=\scal{f,f}_{H^\alpha_M}$,
and define the Sobolev space $H^\alpha_M$ as the completion of $ \mathrm{span}\{e_{n}:\,n\geq 0\}$ with respect to the norm $\|\cdot\|_{H^\alpha_M}$.
For $M=1$ we simply write $H^\alpha=H^\alpha_1$.
For any $\alpha\in\mbbR$, $M,M'\geq 1$, the norms $\|\cdot\|_{H^\alpha_M}$ and $\|\cdot\|_{H^\alpha_{M'}}$ are equivalent, so in particular as vector spaces all $H^\alpha_M$ coincide with $H^\alpha$.
For any $M\geq 1$, $H^0_M=L^2$.
Note that for any $\alpha,\beta\in\mbbR$, the operator $(M-\Delta)^{\beta/2}$ (defined in the natural way $\scal{(M-\Delta)^{\beta/2}f,e_n}_{L^2}=(M+\pi^2\bar n^2)^{\beta/2}\scal{f,e_n}_{L^2}$) is an isometry between $H^\alpha_M$ and $H^{\alpha-\beta}_M$ and that the inner product in $H^\alpha_M$ extends to a duality between $H^{\alpha-\beta}_M$ and $H^{\alpha+\beta}_M$.

Define the processes $(W^n_t)_{t\geq -1}$ for $n\geq 0$ as the continuous modifications of the stochastic process $\xi(\bone_{[0,t]}\otimes e_n)$, with the convention that for $a\geq 0$, $\bone_{[0,-a]}:=-\bone_{[-a,0]}$. It is standard that such continuous modifications exist and $(W^n)_{n\in\mbbN}$ is a sequence of mutually independent two-sided Brownian motions.
Let $\mathbf{F}=(\mathcal{F}_t)_{t\geq-1}$ be a complete filtration such that $(W^n)_{n\in\mbbN}$ is a $\mathbf{F}$-Brownian motion, and let $\mathcal{P}$ be the predictable $\sigma$-algebra constructed from $\mathbf{F}$. 
%For $\gamma\in\mbbR$ define 
%\begin{equ}
%\cH^\gamma_M=L^2(\Omega\times[0,1];\cP;H^{\gamma+1}_M)\cap L^2(\Omega;\cF;C([0,1],H^{\gamma}_M)).
%\end{equ}

We can then write \eqref{eq:Ito-div} and its spatial approximations \eqref{eq:Ito-spatial} in the form
\begin{equ}\label{eq:Ito-general}
	du_t^\eps=\Delta(A(u_t^\eps))\,dt+\sum_{n\geq 0}g(u_t^\eps)(\rho_\eps\ast e_n)\,dW^n_t,
\end{equ}
where $A(r)=\int_0^r a(s)\,ds$, and equip \eqref{eq:Ito-general} with the initial condition $u_0$.

\begin{definition}\label{def:Ito}
	Let $a$ be a bounded function and let $g$ be a function of linear growth. Fix $u_0\in L^2(\Omega,H^{-1})$ and $\eps\in[0,1]$. An It\^o solution of \eqref{eq:Ito-general} with initial condition $u_0$ is a process $u^\eps\in L^2(\Omega\times[0,1]\times \mbbT;\mathcal{P}\otimes \mathcal{B}(\mbbT))\cap L^2(\Omega;\mathcal{F}_1;C([0,1],H^{-1}))$ such that the equality
	\begin{equ}
		u_t^\eps=u_0+\int_0^t\Delta(A(u_s^\eps))\,ds+\sum_{n\geq 0}\int_0^t g(u_s^\eps)(\rho_\eps\ast e_n)\,dW^n_s
	\end{equ}
	holds in $H^{-2}$ almost surely for all $t\in[0,1]$.
\end{definition}
\begin{remark}
	It is clear that the deterministic integral is a continuous process in $H^{-2}$ under the stated conditions on the coefficients and $u^\eps$ itself.
	The fact that the same holds for the sum of the stochastic integrals (in fact, even in $H^{-1}$) is not immediately obvious but will follow from 
	the arguments
	below.
\end{remark}

The following is a summary of the relevant consequences of the monotone operator approach from \cite{KR}.
There is one non-standard aspect in the formulation: we also consider quasilinear equations outside their ``usual'' Gelfand triple $L^2\subset H^{-1}\subset H^{-2}$, in this aspect we borrow arguments from \cite{DGG_STWN}.
\begin{theorem}\label{thm:Ito-WP}
	Let $\eps\in[0,1]$.
	Let $a$ satisfy $\lambda\leq a\leq \lambda^{-1}$ for some constant $\lambda>0$.
	Let $g$ be Lipschitz continuous.
	Let $u_0\in L^2(\Omega,L^2(\mbbT))$.
	Then there exists a unique It\^o solution $u^{\eps}$ to \eqref{eq:Ito-general} with initial condition $u_0$.
	Moreover, for any $\gamma\in[-1,-1/2)$ one has the bound
	\begin{equ}\label{eq:Ito-apriori}
		\E \sup_{t\in[0,1]}\|u_t^\eps\|_{H^{\gamma}}^2+\E\int_0^1 \|u_t^\eps\|_{H^{\gamma+1}}^2\,dt\leq
		N\big(\E\|u_0\|_{H^{\gamma}}^2+g(0)^2\big),
	\end{equ}
	with  $N$ depending only on $\gamma,\lambda, \|g'\|_{L^\infty}$.
Henceforth we denote $u=u^0$.
Then for any $\gamma\in(-1,-1/2)$ one has the boun
	\begin{equ}\label{eq:Ito-error}
		%\E \sup_{t\in[0,1]}\|u_t-u_t^\eps\|_{H^{-1}}^2+
		\E\int_0^1 \|u_t-u_t^\eps\|_{L^2}^2\,dt
		\leq
		N \eps^{2(\gamma+1)},
	\end{equ}
	with $N$ depending only on $\gamma,\lambda,\E\|u_0\|_{H^\gamma}^2, g(0), \|g'\|_{L^\infty},\rho$.
	% $M\geq 1$, and $L<\infty$.
	%Suppose that $\cA:\mbbR\to\mbbR$ is differentiable with $\lambda\leq \cA'\leq L$.
	%\begin{enumerate}[(i)]
	%\item 
	%Suppose that $u_0\in H^{-1}_M$ and that $\cG:\Omega\times[0,T]\times L^2\to\ell^2(H_M^{-1})$ is $\cP\otimes\cB(L^2)$-measurable and is Lipschitz continuous with respect to $v\in L^2$, uniformly in $\omega,t$,  with Lipschitz constant $\lambda$.
	%Then there exists a unique It\^o solution $u$ to \eqref{eq:Ito-general} in $\cH^0_M$.
	%\item Suppose in addition that $u_0\in H^{\gamma}_M$ and that when restricted to $H^{\gamma+1}_M$, $\cG$ takes values in $H^\gamma_M$ and is of linear growth with growth constant $\lambda$.
	%Then, there exists a constant $N=N(\gamma,\lambda,L,M)$ such that one has the bound
	%\begin{equ}
	%\|u\|_{\cH^\gamma_M}\leq N\Big(\|u_0\|_{H^\gamma_M}+\Big(\E\int_0^1\sum_{n\in\mbbN}\|\cG^n_t(0)\|^2_{H^\gamma_M}\,dt\Big)^{1/2}\Big).
	%\end{equ}
	%\end{enumerate}
\end{theorem}

\begin{proof}[of Theorem \ref{thm:Ito-WP}]
	%It is clearly sufficient to show well-posedness and the estimate in $\cH^\gamma_M$ for some large $M$ depending only on $L$, since the norms are equivalent.
	
\emph{Step 1.} We start by showing the first claim, as well as \eqref{eq:Ito-apriori} with $\gamma=-1$. This can be deduced from an appropriate application \cite[Thms.~2.1-2.2-2.3]{KR}. We only need to verify the assumption $A_1)$ through $A_5)$ in \cite[Chap.~II]{KR}, which we do with $p=2$ and the Gelfand triple 
	\begin{equ}
		V\subset H\subset V^*\qquad\text{ with }V=L^2,\, H=H^{-1}_M,\, V^*=H^{-2}_M.
	\end{equ}
	The parameter $M=M(\lambda,\|g'\|_{L^\infty})$ is to be chosen later.
	%, which is ``allowed'' due to the equivalence of the norms.
	As a notational convenience, the inner product in $H$ as well as the duality between
	$V$ and $V^*$ is denoted by $\scal{\cdot,\cdot}$, which is not an essential abuse of notation as the two coincide for $\scal{v,h}$ with $v\in V$, $h\in H$. 
	
	Assumptions $A_1)$, $A_4)$, and $A_5)$ hold trivially. Next we verify $A_2)$ (``Monotonicity''). We use the shorthand $e_n^\eps=\rho_\eps\ast e_n$ and for later purposes, we look at the following quantity for any $u,v\in V$:
	\begin{equs}%[eq:monotone]
		D_M^\eps(u,v):&=2\big{\langle} u-v,\Delta(A(u)-A(v))\big{\rangle}+2\sum_{n\geq 0}\|\big(g(u)-g(v)\big)e^\eps_n\|_{H}^2
	\end{equs}
	This is a trivial upper bound of the left-hand side of the monotonicity condition (which is the same expression without the second factor $2$), so it certainly suffices to bound $D_M^\eps(u,v)$. We write
	\begin{equs}
		D_M^\eps(u,v)
		&
		=2\Big\langle u-v,\big(\Delta-M\big)(A(u)-A(v))+M(A(u)-A(v))\Big\rangle
		\\ & +2\sum_{n\geq 0}\|\big(g(u)-g(v)\big)e^\eps_n\|_{H}^2
		\\
		&= -2\scal{u-v,A(u)-A(v)}_{L_2}+2M\scal{u-v,A(u)-A(v)}
		\\ & +2\sum_{n\geq 0}\|\big(g(u)-g(v)\big)e^\eps_n\|_{H}^2
		\\
		&\leq -2\lambda\|u-v\|_{L_2}^2+2M\scal{u-v,A(u)-A(v)}
	\\ & 	+2\sum_{n\geq 0}\|\big(g(u)-g(v)\big)e^\eps_n\|_{H}^2.%\label{eq:monotone 1}
		%\\
		%&\quad\leq
		%-2\lambda\|u-v\|_{L_2}^2+\eps^{-1}M\|u-v\|_{H^{-1}}^2+\eps M\|A(u)-A(v)\|_{H^{-1}}^2+C(-1,M)L^2\|u-v\|_{L_2}^2
		%\\
		%&\quad\leq -(2\lambda-\eps \lambda^{-1}-C(-1,M)L^2)\|u-v\|_{L_2}^2+\eps^{-1}M\|u-v\|_{H^{-1}}^2,
	\end{equs}
	First, we note the elementary fact that with some $N=N(M,\lambda)$
	\begin{equs}
		2M\scal{u-v,A(u)-A(v)} &\leq N\|u-v\|_{H}^2+\lambda^2\|A(u)-A(v)\|_H^2 \\ &\leq N\|u-v\|_{H}^2+\lambda\|u-v\|_{L^2}^2, 
	\end{equs}
	and therefore we have
	\begin{equs}%[eq:monotone]
		%&2\big{\langle} u-v,\Delta(A(u)-A(v))\big{\rangle}+\sum_{n\geq 0}\|\big(g(u)-g(v)\big)e^\eps_n\|_{H}^2
		%\\
		%&\,\,\leq 
		\begin{aligned}
		D_M^\eps(u,v) & \leq-\lambda\|u-v\|_{L_2}^2+N\|u-v\|_{H}^2
		\\ & +2\sum_{n\geq 0}\|\big(g(u)-g(v)\big)e^\eps_n\|_{H}^2.\label{eq:monotone 2} \end{aligned}
	\end{equs}
	Next, we write, for any $w\in L^2$ and $\alpha<-1/2$
	\begin{equs}
	\label{eq:krylov-1}	\begin{aligned}
		\sum_{n\geq 0}\|we^\eps_n\|_{H^{\alpha}_M}^2
		&= \sum_{n\geq 0} \sum_{m\geq0}(M+\pi^2\bar m^2)^{\alpha}\scal{w e^\eps_n, e_m}_{L^2}^2
		\\
		&= \sum_{m\geq 0}(M+\pi^2\bar m^2)^{\alpha} \sum_{n\geq 0}\scal{\rho_\eps\ast(w e_m), e_n}_{L^2}^2
		\\
		&= \sum_{m\geq 0}(M+\pi^2\bar m^2)^{\alpha} \|\rho_\eps\ast(w  e_m)\|^2_{L^2} 
		\\
		&\leq 2 \|w\|^2_{L^2}\sum_{m\geq 0}(M+\pi^2\bar m^2)^{\alpha}
		\leq N(\alpha) \|w\|^2_{L^2} M^{\alpha+1/2}.
		\end{aligned}
	\end{equs}
	%using $\alpha<-1/2$ 
	Applying this with $w=g(u)-g(v)$ and $\alpha=-1$, we see that if $M$ is sufficiently large (depending only on $\|g'\|_{L^\infty}$ and $\lambda$),
	% so that $M^{-1/2}\|g'\|_{L^{\infty}}\leq \lambda$
	then from \eqref{eq:monotone 2} we get
	\begin{equ}\label{eq:monotone 3}
		%2\big{\langle} u-v,\Delta(A(u)-A(v))\big{\rangle}+\sum_{n\geq 0}\|\big(g(u)-g(v)\big)e^\eps_n\|_{H}^2\leq 
		D_M^\eps(u,v)\leq -(\lambda/2)\|u-v\|_{L^2}^2+ N\|u-v\|_{H}^2,
	\end{equ}
	which verifies the required monotonicity condition $A_2)$.
	
	It remains to verify $A_3)$ (``Coercivity''). For later purposes we again look at a more general quantity, for $u\in H^{\alpha+1}$:
	\begin{equ}\label{eq:coercivity 1}
		E^\alpha_M(u):=2_{H^{\alpha+1}_M}\scal{u,\Delta A(u)}_{H^{\alpha-1}_M}+\sum_{n\geq 0}\|g(u)e^\eps_n\|_{H^\alpha_M}^2,
	\end{equ}
	with $\alpha\in[-1,-1/2)$,
	which is the left-hand side of the coercivity condition when $\alpha=-1$. Using \eqref{eq:krylov-1} and elementary rearrangements as above, we get that for any $\kappa>0$ there exist constants $N=N(\alpha,M,\kappa)$ and $\bar N=\bar N(\alpha)$ such that
	\begin{equs}
		%&
		%2_{H^{\alpha+1}_M}\scal{u,\Delta A(u)}_{H^{\alpha-1}_M}+\sum_{n\geq 0}\|g(u)e^n_\eps\|_{H^\alpha_M}^2
		%\\
		E^\alpha_M(u)&\leq -2_{H^{\alpha+1}_M}\scal{u,(M-\Delta) A(u)}_{H^{\alpha-1}_M}
		+N\|u\|_{H^\alpha_M}\\ & +\kappa\|u\|_{L^2}^2+\bar N M^{\alpha+1/2}\|g(u)\|_{L^2}^2.
	\end{equs}
	Now note that 
	\begin{equ}
		\,_{H^{\alpha+1}_M}\scal{u,(M-\Delta) A(u)}_{H^{\alpha-1}_M}=\scal{u,(M-\Delta)^{1+\alpha} A(u)}_{L^2},
	\end{equ}
	which can be bounded from below by the Stroock-Varopoulos inequality, see Lemma \ref{lem:S-V}. With $\tilde A(r)=\int_0^r\sqrt{a(s)}\,ds$, we therefore get 
	%%%%
	%%%%%%%%   maybe the \alpha=-1 case has to be separated before S-V...
	%%%%
	\begin{equ}\label{eq:coerc-4}
		%2_{H^{\alpha+1}_M}\scal{u,\Delta A(u)}_{H^{\alpha-1}_M}+\sum_{n\geq 0}\|g(u)e^n_\eps\|_{H^\alpha_M}^2
		E^\alpha_M(u)\leq -2\|\tilde A(u)\|_{H^{\alpha+1}_M}^2+N\|u\|_{H^\alpha_M}^2+\kappa\|u\|_{L^2}^2+\bar N M^{\alpha+1/2}\|g(u)\|_{L^2}^2.
	\end{equ}
	First we choose $\alpha=-1$. Then since we have $\|\tilde A(u)\|_{L^2}^2\geq \lambda\|u\|_{L^2}^2$, it remains to choose first $M$ large enough and then $\kappa$ small enough to get
	\begin{equ}
		E^{-1}_M(u)\leq -\lambda\|u\|_{L^2}^2+N\big(\|u\|_{H^{-1}_M}^2+g(0)^2\big),
	\end{equ}
	which is precisely the required coercivity condition $A_3)$. This concludes Step 1.
	
	\emph{Step 2.} Next we show \eqref{eq:Ito-apriori} for $\gamma\in(-1,-1/2)$ and $\eps\in(0,1]$. Note that in this case  the equation
	\begin{equ}
		dv_t=\d_x(\tilde a_t^\eps \d_x v_t)\,dt+\sum_{n\geq 0}g(v_t)e^\eps_n\,dW^n_t,
	\end{equ}
	with $\tilde a^\eps_t=a(u^\eps_t)$,  is also well-posed in the Gelfand triple $H^1\subset L^2 \subset H^{-1}$, and since $u^\eps$ is a solution, we get that $u^\eps\in L^2(\Omega\times[0,1];\mathcal{P}; H^1)\cap L^2(\Omega;\mathcal{F}_1;C([0,1],L^2(\mbbT)))$. In particular, we can write It\^o's formula for the square of its $H^\gamma$ norm. Using the notation \eqref{eq:coercivity 1}, we can write
	\begin{equ}\label{eq:Ito-Hgamma}
		\|u_t^\eps\|_{H^\gamma}^2=\E\|u_0\|_{H^\gamma}^2+\int_0^t E^\gamma(u^\eps_s)\,ds+2\sum_{n\geq 0}\int_0^t \scal{u^\eps_s,g(u^\eps_s)e^\eps_n}_{H^\gamma}\,dW^n_s.
	\end{equ}
	From \eqref{eq:coerc-4} and  interpolation between the $H^\alpha$ spaces we have that for any $\kappa'>0$ there exist a constant $N=N(\gamma,\kappa')$ such that
	\begin{equ}
		E^{\gamma}(u)\leq -2\|\tilde A(u)\|_{H^{\gamma+1}}^2+N\big(\|u\|_{H^\gamma}^2+g(0)^2\big)+\kappa'\|u\|_{H^{\gamma+1}}^2.
	\end{equ}
	Using Proposition \ref{prop:triv} (ii) for the first term and then choosing $\kappa'$ sufficiently small, we get that with some $\theta=\theta(\lambda,\gamma)>0$
	\begin{equ}
		E^{\gamma}(u)\leq -\theta\|u\|_{H^{\gamma+1}}^2+N\big(\|u\|_{H^\gamma}^2+g(0)^2\big).
	\end{equ}
	Next, the quadratic variation of the martingale in \eqref{eq:Ito-Hgamma} is given by
	\begin{equs}
		4\sum_{n\geq 0}\int_0^t\scal{u^\eps_s,g(u^\eps_s)e^\eps_n}_{H^\gamma}^2\,ds&\leq 4 \sup_{s\in[0,t]}\|u_s^\eps\|_{H^\gamma}^2\int_0^t\sum_{n\geq 0}\|g(u^\eps_s)e^\eps_n\|_{H^\gamma}^2\,ds
		\\
		&\leq N(\gamma)\sup_{s\in[0,t]}\|u_s^\eps\|_{H^\gamma}^2\int_0^t\|g(u^\eps_s)\|_{L^2}^2\,ds,
	\end{equs}
	using \eqref{eq:krylov-1} in the last inequality.
	Therefore from \eqref{eq:Ito-Hgamma}, applying Burkholder-Davis-Gundy and Young inequalities, followed by interpolation between the $H^\alpha$ spaces,
	we get that with some $N=N(\lambda,\gamma)$,
	\begin{equs}
	&	\E\sup_{s\in[0,t]}\|u_t^\eps\|_{H^\gamma}^2
	\\ &\leq \E\|u_0\|^2_{H^\gamma}+Ng(0)^2+\int_0^t -\theta\E\|u_s^\eps\|_{H^{\gamma+1}}^2+N\E\sup_{r\in[0,s]}\|u_r^\eps\|^2_{H^\gamma}\,ds
		\\
		&\qquad+\frac{1}{2}\E\sup_{s\in[0,t]}\|u_s^\eps\|_{H^\gamma}^2+N\int_0^t\|g(u^\eps_s)\|_{L^2}^2\,ds
		\\
		&\leq \E\|u_0\|_{H^\gamma}^2+Ng(0)^2+\int_0^t -(\theta/2)\E\|u_s^\eps\|_{H^{\gamma+1}}^2+N\E\sup_{r\in[0,s]}\|u_r^\eps\|^2_{H^\gamma}\,ds
		\\
		&\qquad+\frac{1}{2}\E\sup_{s\in[0,t]}\|u_s^\eps\|_{H^\gamma}^2,
	\end{equs}
	or, after rearranging,
	\begin{equs}
	&	\E\sup_{s\in[0,t]}\|u_s^\eps\|_{H^\gamma}^2+\int_0^t\E\|u_s^\eps\|_{H^{\gamma+1}}^2\,ds
		\\ &
		\leq N\Big(
		\E\|u_0\|_{H^\gamma}^2+g(0)^2+\int_0^t\E\sup_{r\in[0,s]}\|u_r^\eps\|^2_{H^\gamma}\,ds\Big).
	\end{equs}
	Gronwall's lemma finishes Step 2.
	
	\emph{Step 3.} Next we show \eqref{eq:Ito-error}. Note that by Fatou's lemma this will also prove \eqref{eq:Ito-apriori} in the missing case ($\eps=0$, $\gamma\in(-1,-1/2)$).
	From Step 1 we have $M$ fixed and we already know that It\^o's formula can be applied for the process $\|u-u^\eps\|_{H^{-1}}^2$. Combined with an elementary Young's inequality, we get
	\begin{equs}\label{eq:Ito-difference}
	\begin{aligned}	\|u_t-u^\eps_t\|_{H^{-1}_M}^2&\leq\int_0^t D_M^0(u_s,u^\eps_s)+2\sum_{n\geq 0}\|g(u^\eps_s)(e^\eps_\eps-e_n)\|_{H^{-1}_M}^2\,ds
		\\ & +I_t
		\end{aligned}
		%\\
		%&\leq \int_0^t -(\lambda/2)\|u_s-u^\eps_s\|_{L^2}^2+N\|u_s-u^\eps_s\|_{H^{-1}_M}^2+2\sum_{n\geq 0}\|g(u^\eps_s)(e^\eps_\eps-e_n)\|_{H^{-1}_M}^2\,ds+m_t
	\end{equs}
	with some martingale $I$.
	Similarly to \eqref{eq:krylov-1} we can write, for any $w\in H^\gamma$
	\begin{equs}
		\sum_{n\geq 0}\|w(e^\eps_n-e_n)\|_{H^{-1}_M}^2 &= \sum_{n\geq 0} \sum_{m\geq 0}(M+\pi^2 \bar m^2)^{-1}|\scal{w(e^\eps_n-e_n), e_m}_{L^2}|^2
		\\
		&=\sum_{m\geq 0}(M+\pi^2 \bar m^2)^{-1} \sum_{n\geq 0} |\scal{ e_n, (\rho^\eps-\delta_0)\ast(w e_m)}_{L^2}|^2
		\\&=
		\sum_{m\geq 0}(M+\pi^2 \bar m^2)^{-1}\|(\rho^\eps-\delta_0)\ast(w  e_m)\|^2_{L^2} .
	\end{equs}
	Choose $\tilde\gamma>\gamma+1$ such that $2\tilde\gamma-2<-1$.
	Using Proposition \ref{prop:triv} (iv)-(v) and the fact that $\|e_m\|_{C^{\tilde\gamma}}\leq 10 m^{\tilde\gamma}$ for $m\geq 1$, we can continue with
	\begin{equs}
	\label{eq:krylov-2}	\begin{aligned}
		\sum_{n\geq 0}\|w(e^\eps_n-e_n)\|_{H^{-1}_M}^2 & \leq N \eps^{2(\gamma+1)}\sum_{m\geq 0}(M+\pi^2 \bar m^2)^{-1}\|w  e_m\|^2_{H^{\gamma+1}}
		\\
		& \leq N \eps^{2(\gamma+1)}\sum_{m\geq 0}(M+\pi^2 \bar m^2)^{-1}m^{2\tilde \gamma}\|w \|_{H^{ \gamma+1}}^2
		\\
		&\leq N\eps^{2(\gamma+1)}\|w \|_{H^{ \gamma+1}}^2. \end{aligned}
	\end{equs}
	Using \eqref{eq:monotone 3} and \eqref{eq:krylov-2} in \eqref{eq:Ito-difference} and taking expectations, we get
	\begin{equs}
		\E\|u_t-u^\eps_t\|_{H^{-1}_M}^2
		& \leq \E\int_0^t-(\lambda/2)\|u_s-u^\eps_s\|_{L^2}^2+N\|u_s-u^\eps_s\|_{H^{-1}_M}^2
		\\ & +N\eps^{2(\gamma+1)}\|g(u^\eps_s)\|_{H^{\gamma+1}}^2\,ds.
	\end{equs}
	Rearranging, using Proposition \ref{prop:triv} (iii), and the bound from Step 2, we get 
	\begin{equ}
		\E\|u_t-u^\eps_t\|_{H^{-1}_M}^2+ \E\int_0^t\|u_s-u^\eps_s\|_{L^2}^2\leq N\Big( \eps^{2(\gamma+1)}+\int_0^t\E\|u_s-u^\eps_s\|_{H^{-1}_M}^2\,ds\Big),
	\end{equ}
	hence Gronwall's lemma yields the claimed bound \eqref{eq:Ito-error}.

\end{proof}

%\begin{corollary}
%For all $\eps\in[0,1]$, the equation
%\begin{equ}\label{eq:Ito-general}
%du_t=\d_x\big(a(u_t)\d_x u_t\big)\,dt+\sum_{n\geq 0}u_t e^n\,dW^n_t.
%\end{equ}
%\end{corollary}

In the above proof we used the following simple properties of Sobolev spaces, for the sake of completeness we provide short proofs.
\begin{proposition}\label{prop:triv}
	For any $\alpha\in(0,1)$, one has
	\begin{enumerate}
		\item[(i)] The norm
		\begin{equ}\label{eq:slobodeckij}
			\|f\|_{L^2}+\sqrt{\int_\mbbT\int_\mbbT\frac{|f(x)-f(y)|^2}{|x-y|^{1+2\alpha}}\,dx\,dy}
		\end{equ}
		is equivalent to $\|f\|_{H^\alpha}$.
		\item[(ii)] For any Lipschitz continuous function $F$ with $F(0)=0$, there exists a constant $N=N(\alpha, \|F'\|_{L^\infty})$ such that for all $u\in H^\alpha$, $\|F(u)\|_{H^{\alpha}}\leq N \|u\|_{H^\alpha}.$
		\item[(iii)] For any increasing function $F$ with $F(0)=0$ and $F'\geq\lambda$ for some $\lambda>0$, there exists a constant $C=C(\alpha,\lambda)>0$ such that for all $u\in H^\alpha$, $\|F(u)\|_{H^{\alpha}}\geq C \|u\|_{H^\alpha}$.
		%the map $f\mapsto F(f)$ is continuous from $H^{\alpha}_M$ to itself.
		\item[(iv)] For any $\alpha'>\alpha$ there exists a constant $N=N(\alpha,\alpha')$ such that for all
		$f\in H^{\alpha}$ and $g\in \CC^{\alpha'}$, $\|fg\|_{H^\alpha}\leq N\|f\|_{H^\alpha}\|g\|_{\CC^{\alpha'}}.$
		\item[(v)] There exists a constant $N=N(\rho)$ such that for all $\eps>0$ and $f\in H^\alpha$, $\|(\rho_\eps-\delta_0)\ast f\|_{L^2}\leq N\eps^{\alpha}\|f\|_{H^\alpha}.$
	\end{enumerate}
\end{proposition}
\begin{proof}
	Part (i) is well-known, its short proof is as follows:
	switching to complex Fourier series by setting $\tilde e_{\pm n}=e^{2i\pi n}$ for $n\in\mathbb{Z}$,
	after a change of variables the integral in \eqref{eq:slobodeckij} can be written as
	\begin{equ}
		\int_\mbbT\frac{\|f(\cdot)-f(\cdot+z)\|_{L^2}^2}{|z|^{1+2\alpha}}\,dz=\int_\mbbT\sum_{n\in\mathbb{Z}}\frac{|\scal{f(\cdot)-f(\cdot+z),\tilde e_n}|^2}{|z|^{1+2\alpha}}\,dz.
	\end{equ}
	One clearly has $\scal{f(\cdot+z),\tilde e_n}=\scal{f,\tilde e_n}e^{2\pi i n z}$. 
	Therefore, the $n=0$ term vanishes and for all others we can write
	%As for the other Fourier coefficients, take a nonzero even $n$, so that $e_n(x)=\sqrt{2}\cos(n\pi x)$ and $e_{n-1}=\sqrt{2}\sin(n\pi x)$. Then
	%\begin{equs}
	%\scal{f(\cdot+z),e_n}&=\scal{f,e_n(\cdot-z)}=\scal{f,e_n}\cos(n\pi z)+\scal{f,e_{n-1}}\sin(n\pi z),
	%\\
	%\scal{f(\cdot+z),e_{n-1}}&=\scal{f,e_{n-1}(\cdot-z)}=\scal{f,e_{n-1}}\cos(n\pi z)-\scal{f,e_n}\sin(n\pi z).
	%\end{equs}
	%Therefore, 
	%\begin{equ}
	%\scal{f(\cdot)-f(\cdot+z),e_n}^2+\scal{f(\cdot)-f(\cdot+z),e_{n-1}}^2
	%=\big(\scal{f,e_n}^2+\scal{f,e_{n-1}}^2\big)2\big(1-\cos(n\pi z)\big).
	%\end{equ}
	%By one more trigonometric identity $1-\cos(x)=2\sin^2(x/2)$, all that remains to note is that
	\begin{equs}
		\int_\mbbT\frac{|\scal{f(\cdot)-f(\cdot+z),\tilde e_n}|^2}{|z|^{1+2\alpha}}\,dz
		%&=|\scal{f,\tilde e_n}|^2\int_0^1\frac{|1-e^{2\pi i n z}|}{|z|^{1+2\alpha}}\,dz\\
		&=|\scal{f,\tilde e_n}|^2|n|^{2\alpha}\int_0^{|n|}\frac{|1-e^{\pm 2\pi i  z}|}{|z|^{1+2\alpha}}\,dz.
	\end{equs}
	The latter integral is bounded both from above and away from $0$ uniformly in $n\in\mathbb{Z}\setminus\{0\}$, from which (i) readily follows.
	Parts (ii) and (iii) follow immediately from (i) for $\alpha\in(0,1)$ and are trivial for $\alpha=0$.
	As for (iv), the inequality $\|fg\|_{L^2}\leq\|f\|_{L^2}\|g\|_{L^\infty}$ is obvious, and to estimate the integral in \eqref{eq:slobodeckij} with $fg$ in place of $f$, we can bound it by
	\begin{equs}
		%\vn{fg}_{\dot H^\alpha}^2&\leq 
		2\int_\mbbT \int_\mbbT&\frac{|f(x)-f(y)|^2g(y)^2}{|x-y|^{1+2\alpha}}\,dx\,dy+2\int_\mbbT \int_\mbbT \frac{f(x)^2|g(x)-g(y)|^2}{|x-y|^{1+2\alpha}}\,dx\,dy
		\\
		&\leq 2\|f\|_{H^\alpha}^2\|g\|_{L^\infty}^2+2\int_\mbbT f(x)^2\int_\mbbT\frac{|g(x)-g(x-z)|^2}{|z|^{1+2\alpha}}\,dz\,dx
		\\
		&\lesssim \|f\|_{ H^\alpha}^2\|g\|_{L^\infty}^2+\|f\|_{L^2}^2\|g\|_{ \CC^{\alpha'}}^2,
	\end{equs}
	using $\alpha'>\alpha$ in the last inequality to guarantee the finiteness of the integral over $z$.
	Concerning (v), one has by Jensen's inequality followed by trivialities,
	\begin{equs}
		\|(\rho_\eps-\delta_0)\ast f\|_{L^2}^2&=\int_\mbbT \Big|\int_\mbbT \big(f(x)-f(y)\big)\rho_\eps(x-y)\,dy\Big|^2\,dx
		\\
		&\leq \int_\mbbT \int_\mbbT \big|f(x)-f(y)\big|^2\rho_\eps(x-y)\,dy\,dx
		\\
		&\lesssim\eps^{2\alpha}\int_\mbbT \int_\mbbT \big|f(x)-f(y)\big|^2\eps^{-1-2\alpha}\bone_{|x-y|\leq\eps}\,dy\,dx
		\\
		&\leq\eps^{2\alpha}\int_\mbbT\int_\mbbT \big|f(x)-f(y)\big|^2|x-y|^{-1-2\alpha}\,dy\,dx\lesssim\eps^{2\alpha}\|f\|_{ H^\alpha}^2.
	\end{equs}
\end{proof}

\begin{lemma}[Stroock-Varopoulos inequality]\label{lem:S-V}
	Let $f$ and $g$ be Lipshitz functions such that $f'=(g')^2$. Then for any $M\geq 1$, $\beta\in[0,1)$ one has for any smooth function $u$
	\begin{equ}
		\scal{u,(M-\Delta)^\beta f(u)}_{L^2}\geq \scal{(M-\Delta)^{\beta/2}g(u),(M-\Delta)^{\beta/2} g(u)}_{L^2}=\|g(u)\|_{H^\beta_M}^2.
	\end{equ}
\end{lemma}
\begin{proof}
	The inequality holds for any operator $L$ in place of $(M-\Delta)^\beta$ having an integral representation of the form
	\begin{equ}
		L u(x)=\int \big(u(x)-u(y)\big)K(x,y)\,dy,
	\end{equ}
where $K$ is symmetric and nonnegative, see e.g. \cite[App.~B]{BFR}. The fact that $(M-\Delta)^\beta$ does have such a representation can be seen from the formula
	\begin{equ}
		(M-\Delta)^\beta u=c_\beta\int_0^\infty\frac{u-e^{t(M-\Delta)}u}{t^{1+\beta}}\,dt,
	\end{equ}
	with some positive constant $c_\beta$, which in turn can be readily verified by taking Fourier transform of both sides.
\end{proof}

%%%
%%%
%%%%%%%% To be continued here
%%%
%%%

\subsection{Regularity of It\^o solutions}
The following regularity estimate is the multiplicative version of \cite{OW-div}.
\begin{theorem}\label{thm:Ito-regularity}
	Let $a$ be Lipshitz continuous and satisfy $\lambda\leq a\leq \lambda^{-1}$ for some constant $\lambda>0$.
	Then there exists an $\alpha_0>0$ depending only on $\lambda$ such that for all $\alpha\in(0,\alpha_0)$ the following holds.
	Let $g$ be Lipschitz continuous.
	Let $u_0\in \mathcal{C}^{\alpha}(\mbbT)$.
	Then the It\^o solution of \eqref{eq:Ito-div} with initial condition $u_0$ admits a continuous modification, and one has the bound
	\begin{equ}\label{eq:Ito-Holderbound}
		\E\|u\|_{\mathcal{C}^{\alpha}([0,1]\times\mbbT)}^2\leq N\big(1+\E\|u_0\|_{\mathcal{C}^{\alpha}(\mbbT)}^2\big),
	\end{equ}
	where the constant $N$ depends on $\lambda$, $\alpha$, $|g(0)|$, and $\|g'\|_{L^\infty}$.
\end{theorem}
\begin{proof}
	Take first $\eps\in(0,1]$ and $u_\eps$ from Theorem \ref{thm:Ito-WP}.
	Since we already have the convergence $u_\eps\to u$ from Theorem \ref{thm:Ito-WP}, it suffices to show \eqref{eq:Ito-Holderbound} for $u_\eps$ in place of $u$, uniformly over $\eps\in(0,1]$.
	Further, by a standard stopping time argument we may and will assume that $g$ is also bounded.
	As mentioned, the additive $g\equiv 1$ case is proven in \cite{OW-div}. Our goal is to reduce the mutliplicative case to the additive.
	
	We write
	\begin{equ}
		\d_t u_\eps-\d_x\big( a(u_\eps)\d_x u_\eps\big)=\d_t v_\eps-\d_x^2 v_\eps,
	\end{equ}
	where $v_\eps$ is the solution of
	the linear equation
	\begin{equ}\label{eq:IR-v}
		\d_t v_\eps-\d_x^2v_\eps=g(u_\eps)\xi_\eps,%\qquad v^m(0,\cdot)=\psi(\cdot),
	\end{equ}
	with initial condition $u_0$. For $v_\eps$ one has the explicit solution formula
	\begin{equ}
		v^\eps_t=\mathcal{P}_tu_0+\sum_{n\in \mbbN}\int_0^t \mathcal{P}_{t-s}g(u^\eps_s)e^\eps_n\,dW^n_s,
	\end{equ}
	where $\mathcal{P}_t=e^{t\d_x^2}$ is the heat kernel on $\mbbT$.
	Denote from now on $Q=[0,1]\times\mbbT$.
It is well-known from regularity of mild It\^o SPDEs (see e.g. \cite[Cor.~3.4]{Walsh} for a prototype calculation that applies here with minimal changes) that for any $\alpha<1/4$ there exists a $p\in[2,\infty)$ such that for all $q\in[2,p]$ one has
%If needed, add details, but rather standard...
	\begin{equ}\label{eq:IR-classicalbound}
		\E\|v^\eps\|_{\mathcal{C}^{\alpha}(Q)}^q\leq N\big(\E\|u_0\|_{\mathcal{C}^{\alpha}(\mbbT)}^q+\E\|g(u^\eps)\|_{L^p(Q)}^q\big),
	\end{equ}
	where the constant $N$ depends only on $\alpha$, $p$, $q$, in particular it is uniform over $\eps\in[0,1]$.
	
Next, a simple adaptation \cite[Thm.~2.6]{DDH} to the periodic and infinite dimensional (but spatially smooth) noise setting yields $u_\eps\in L^p(\Omega,\cC^{\alpha_1}(Q))$ for some $\alpha_1>0$ for all $p\in[2,\infty)$.
\iffalse
%%%%% If details are needed, here is some sketch...
 we claim that for $\eps>0$, $u^\eps$ belongs to $L^p(\Omega,\cC^{\alpha_1}(Q))$ for some $\alpha_1>0$ (depending only on $\lambda$), for all $p\in[2,\infty)$.
As in Step 2 of the proof of Theorem \ref{thm:Ito-WP}, we have that $u^\eps\in L^2(\Omega\times[0,1];H^1)$.
Therefore, defining $\tilde v^\eps$ as the solution of \eqref{eq:IR-v} with initial condition $0$, we have that 
\begin{equ}
\partial_x \tilde v^\eps_t=\sum_{n\geq 0}\int_0^t\cP_{t-s}\big(g(u^\eps_s)\partial_x\rho^\eps+g'(u^\eps_s)\partial_x
\end{equ}
{\color{red}to be continued. Sketch: from here $\nabla v$ is a stochastic integral of an $L^2$ function, so should belong to any $L^p$, $p<\infty$. Then $u-v$ solves a deterministic quasilinear PDE with RHS $\nabla g$, $g\in L^{\infty-}$, which by De Giorgi should give the Holder bound. Details to be added.}
	
	%For $\eps>0$, the noise is spatially smooth therefore $u^\eps$ is known to belong to $L^p(\Omega,\cC^{\alpha_1}(Q))$ for some $\alpha_1>0$ (depending only on $\lambda$), for all $p\in[2,\infty)$ (see e.g. \cite{DDH}). {\color{red} (here another small issue of technical differences from the reference. they work on domain with Dirichlet b.c. instead of torus, but the latter is always easier. but this might need expanding)}
\fi
%With the qualitative information $\|u^\eps\|_{\cC^{\alpha_1}(Q)}<\infty$ a.s. in hand, one has the following result \cite[Lemma~1]{OW-div}. 
With this qualitative information in hand one can invoke the following result \cite[Lem.~1]{OW-div}:
for some $\alpha_2\in(0,\alpha_1]$, for all $\alpha\in(0,\alpha_2)$, one has the bound
	\begin{equ}\label{eq:OW-bound}
\|u_\eps(\omega)\|_{\cC^{\alpha}(Q)}\leq N\|v_\eps(\omega)\|_{\cC^{\alpha}(Q)}
		%_{L^\infty([0,T],\CC^\alpha(\mbbT))}
	\end{equ}
	for any $\omega\in\Omega$ such that the left-hand side is finite,
	where the constant $N$ depends only on $\lambda$ and $\alpha$, in particular it is uniform over $\eps\in(0,1]$ and $\omega\in\Omega$.
	\begin{remark}
		We actually use \cite[Lem.~1]{OW-div} with a few technical modifications.
Firstly, therein the bounds are stated in terms of the antiderivatives of $u_\eps$ and $v_\eps$. Moving from one formulation to the other is rather straightforward and is also exposed in \cite[Sec.~1]{OW-div}.
Secondly, \cite[Lem.~1]{OW-div} implicitly assumes $0$ initial condition on $u_\eps$, but the proof in fact provides first the bound \eqref{eq:OW-bound} for $w_\eps=u_\eps-v_\eps$ in place of $u_\eps$, which is of course equivalent.
		Since $w_\eps$ does have $0$ initial conditions in our setting as well, having initial conditions does not change anything.
Thirdly, in \cite[Lem.~1]{OW-div} the right-hand side of \eqref{eq:OW-bound} involves the $\CC^{\alpha}([-1,1]\times\mbbT)$ norm of $v_\eps$, more precisely, of its extension $\bar v_\eps$ by $0$ to negative times. This would force a certain vanishing of $v_\eps$ at $0$.
		One can again easily check in the proof that the temporal regularity of $\bar v_\eps$ is not used, only its ${L^\infty([-1,1],\CC^\alpha(\mbbT))}$-norm. This is clearly bounded by the right-hand side of \eqref{eq:OW-bound}.
	\end{remark}
	Choose $\alpha<\alpha_2\wedge(1/2)$ and the corresponding $p\in[2,\infty)$ from \eqref{eq:IR-classicalbound}.
Putting \eqref{eq:OW-bound} and \eqref{eq:IR-classicalbound} together, using the Lipschitz assumption on $g$, interpolation of the $L^p$ norm, and \eqref{eq:Ito-apriori}, we get, for any $\kappa>0$, 
	\begin{equs}
		\E\|u_\eps\|_{\mathcal{C}^{\alpha}(Q)}^2
		&\leq N\big(\E\|u_0\|_{\mathcal{C}^{\alpha}(\mbbT)}^2+\E\|g(u_\eps)\|_{L^p(Q)}^2\big)
		\\
		&\leq N\big(1+\E\|u_0\|_{\mathcal{C}^{\alpha}(\mbbT)}^2+\E\|u_\eps\|_{L^p(Q)}^2\big)
		\\
		&\leq N\big(1+\E\|u_0\|_{\mathcal{C}^{\alpha}(\mbbT)}^2+\kappa\E\|u_\eps\|_{L^\infty(Q)}^2+N'(\kappa)\E\|u_\eps\|_{L^2(Q)}^2\big)
		\\
		&\leq N\kappa\E\|u_\eps\|_{L^\infty(Q)}^2+N'(\kappa)(1+\E\|u_0\|_{\mathcal{C}^{\alpha}(\mbbT)}^2).
	\end{equs}
	Since $N$ does not depend on $\kappa$ (only $N'$ does), we can choose $\kappa$ small enough that $N\kappa<1/2$, absorb the first term in the left-hand side, and thus conclude the proof.
\end{proof}

\subsection{ Wong-Zakai approximation of non-singular quasilinear SPDEs}\label{sec:smoothWZ}

The purpose of this section is to formulate an appropriate version of a (temporal) Wong-Zakai theorem for quasilinear SPDEs of the form
\begin{equ}\label{eq:qsmooth}
dv=\partial_x \big(a(v)\partial_x v\big)\,dt+\sum_{k\in\mbbN}g(v)h_k\,dW^k_t
\end{equ}
with some initial condition $u_0\in L^2(\mbbT)$,
under some (strong) coloring condition in the spatial variable. 
%\begin{assumption}\label{asn:WZ}
%The sum 
%\begin{equ}
%\sum_{k\in\mbbN} h_k
%\end{equ}
%converges in $L_\infty$.
%\end{assumption}
\begin{theorem}\label{thm:WZ}
Let $u_0\in L^2(\mbbT)$, $g\in\CC^2$, and suppose that the sum
$
\sum_{k\in\mbbN} h_k
$
converges in $L_\infty(\mbbT)$.
 For $\bar \eps>0$ let $W^{k,\bar\eps}=\rho^+_{\bar\eps}\ast W^k$.
Consider the random PDE
\begin{equ}\label{eq:qsmoothWZ}
\partial_t v_{\bar \eps}-\partial_x(a(v_{\bar\eps})\partial_x v_{\bar\eps})=\sum_{k\in\mbbN}g(v_{\bar\eps})h_k\partial_t W^{k,\bar\eps}_t-\frac{1}{2}\sum_{k\in\mbbN}g'(v_{\bar\eps})g(v_{\bar\eps})h_k^2
\end{equ}
with initial condition $u_0$.
Then as $\bar \eps\to 0$, one has
\begin{equ}
\E\big(\sup_{t\in[0,1]}\|v_{\bar\eps}-v\|_{L^2(\mbbT)}^2\big)\to 0.
\end{equ}
\end{theorem}
Results on Wong-Zakai approximations of SPDEs with colored noise are plentiful in the literature, and even in the quasilinear case there are several versions \cite{Twardowska, Ma-Zhu, HN}. Unfortunately, none of them imply Theorem \ref{thm:WZ} in this form. That said, Theorem \ref{thm:WZ} being a minor technical variation of well-established results, we do not aim for a self-contained proof and we only detail the steps that differ from the corresponding calculations of \cite{Ma-Zhu}.
Indeed, their setup with $H=L^2$, $V=H^1$, $\alpha=2$, $\beta=0$, arbitrary $p>2$, $U=\ell^2$, $A(v)=\partial_x(a(v)\partial_x v)$, $B(v)=(B_k(v))_{k\in\mbbN}=(g(v)h_k)_{k\in\mbbN}$, almost accommodated \eqref{eq:qsmoothWZ}, except for the choice of noise approximation: instead of mollification, \cite{Ma-Zhu} takes polygonal approximation combined with a truncation in $k$.
Below we outline how to reduce our setup to that of \cite{Ma-Zhu, HN}.

\begin{proof}
For $K\in\mbbN\cup\{\infty\}$ introduce $v^{(K)}$ ($v^{(K)}_{\beps}$, resp.) as the solution of \eqref{eq:qsmooth} (\eqref{eq:qsmoothWZ}, resp.) when replacing the sum (sums, resp.) over $k\in\mbbN$ by $k< K$. In particular, $v^{(\infty)}=v$, $v^{(\infty)}_{\beps}=v_{\beps}$.

One starts by establishing uniform in $\beps,K$ estimates for \eqref{eq:qsmoothWZ}. This is done similarly to \cite[Lem.~4.1]{Ma-Zhu}, in the simplified setting $B_1=B_3=0$.
%The other terms are $A(v)=\partial_x\big(a(v)\partial_x v\big)$ as before and $F(v)=\frac{1}{2}\sum_{k\in\mbbN}g'(v)g(v)h_k^2$.
We shall omit the rather standard Galerkin approximation step (i.e. take $m=\infty$ in the notation of \cite{Ma-Zhu}).
The a priori estimate is obtained from an energy estimate to $\|v^{(K)}_{\bar\eps}\|_{L^2}^p$. We now consider $\beps>0, K\in\mbbN\cup\{\infty\}$ fixed and for brevity we denote $y=v^{(K)}_{\bar\eps}$.
The only nonstandard term that has some blowup in $\beps\to0$ (and the only term that differs from our setting to that of \cite{Ma-Zhu}) is the contribution of the smoothed noise, namely $I_{n,m}(t,3)$ in the notation of \cite{Ma-Zhu}; we denote the corresponding term for us by
\begin{equ}
I(t,3):=p\sum_{k<K}\int_0^t\|y_s\|_{L^2}^{p-2}\scal{g(y_s)h_k\partial_s W^{k,\beps}_s,y_s}_{L^2}\,ds.
\end{equ}
Since $\partial_s W^{k,\beps}_s=\int_{s-\beps}^s\varrho^+(s-r)\,dW^k_r$, we have
\begin{equs}
I(t,3):&=p\sum_{k<K}\int_0^t\int_{s-\beps}^s\varrho^+(s-r)\|y_s\|_{L^2}^{p-2}\scal{g(y_s)h_k,y_s}_{L^2}\,dW^k_r\,ds
\\
&=p\sum_{k<K}\int_{-\beps}^t
\int_{r\vee 0}^{(r+\beps)\wedge t}
\varrho^+(s-r)\|y_s\|_{L^2}^{p-2}\scal{g(y_s)h_k,y_s}_{L^2}
\,ds\,dW^k_r.
\end{equs}
By the Burkholder-David-Gundy inequality, and Young's convolutional inequality one gets
%comment as reminder: here Young's is particularly simple: one does first C-S in the s variable with rho^1/2 and rho^1/2 F, where F is all the rest. The (rho^1/2)^2 integrates to 1. In the other term one does Fubini and notice that in r we again integrate (rho^1/2)^2. This argument also shows where the summation should stand
\begin{equs}
\E &\sup_{t\in[0,1]}\big|I(t,3)\big|
\\
&\lesssim
\E\Big(\int_0^t
\sum_{k<K}\Big(\int_{r\vee 0}^{(r+\beps)\wedge t}
\varrho^+(s-r)\|y_s\|_{L^2}^{p-2}\scal{g(y_s)h_k,y_s}_{L^2}
\,ds\Big)^2\,dr\Big)^{1/2}
\\
&\lesssim \E\Big(\sum_{k<K}\int_0^t\|y_s\|_{L^2}^{2p-4}\scal{g(y_s)h_k,y_s}_{L^2}^2\,ds\Big)^{1/2}
\end{equs}
By the assumption on $(h_k)_{k\in\mbbN}$ and Young's inequality we get for any $\mu>0$ there exists a constant $C$ depending on $\mu, (h_k)_{k\in\mbbN}, g, p$ such that 
\begin{equ}
\E \sup_{t\in[0,1]}\big|I(t,3)\big|\leq \E\Big(\mu\sup_{t\in[0,1]}\|y_t\|_{L^2}^p+C\int_0^1 1+\|y_s\|_{L^2}^p\,ds\Big).
\end{equ}
This is the exact analogue of \cite[Eq.~(A.7)]{Ma-Zhu}, and the rest of the proof can be concluded as therein. This leads to the a priori estimate
\begin{equ}
\sup_{K\in\mbbN\cup\{\infty\}}\sup_{\bar\eps\in(0,1]}\E\Big(\sup_{t\in[0,1]}\|v^{(K),\beps}_t\|_{L^2}^p+\int_0^1\|v^{(K),\beps}_t\|_{H^1}^2\,dt\Big)<\infty.
\end{equ}

Using the a priori estimates one easily gets the convergences
\begin{equ}
\sup_{\bar\eps\in(0,1]}\E\big(\sup_{t\in[0,1]}\|v^{(K),\beps}_t-v^{\beps}_t\|_{L^2}^2\big)\to 0,\qquad K\to \infty,
\end{equ}
and similarly one has for the limiting equation
\begin{equ}
\E\big(\sup_{t\in[0,1]}\|v^{(K)}_t-v_t\|_{L^2}^2\big)\to 0,\qquad K\to \infty.
\end{equ}
Therefore it suffices to verify, for each $K<\infty$, the convergence
\begin{equ}\label{eq:...}
\E\big(\sup_{t\in[0,1]}\|v^{(K),\beps}_t-v^{(K)}_t\|_{L^2}^2\big)\to 0.
\end{equ}
Since the equation for $v^{(K)}$ is driven by a finite dimensional Brownian noise, it can simply be viewed as a quasilinear rough PDE, for which a Wong-Zakai result like \eqref{eq:...} is just an instance of stability with respect to the driving noise, see \cite[Thm.~2.12]{HN}.
\end{proof}

\end{document}